\documentclass[10pt,oneside,english,leqno]{amsart}
\usepackage{mathptmx}

\usepackage[T1]{fontenc}
\usepackage[latin9]{inputenc}
\usepackage{geometry}
\usepackage{xcolor}
\usepackage{pdfcolmk}
\usepackage{babel}
\usepackage{amsthm}
\usepackage{amstext}
\usepackage{amssymb}
\usepackage{graphicx}
\usepackage{setspace}
\usepackage{esint}
\PassOptionsToPackage{normalem}{ulem}
\usepackage{ulem}
\usepackage[unicode=true,pdfusetitle,
 bookmarks=true,bookmarksnumbered=false,bookmarksopen=false,
 breaklinks=false,pdfborder={0 0 1},backref=false,colorlinks=false]
 {hyperref}

\makeatletter

\providecolor{lyxadded}{rgb}{0,0,1}
\providecolor{lyxdeleted}{rgb}{1,0,0}

\numberwithin{equation}{section}
\numberwithin{figure}{section}
 \theoremstyle{definition}
 \newtheorem*{defn*}{\protect\definitionname}
\theoremstyle{plain}
\newtheorem{thm}{\protect\theoremname}
  \theoremstyle{plain}
  \newtheorem{cor}[thm]{\protect\corollaryname}
  \theoremstyle{definition}
  \newtheorem{defn}[thm]{\protect\definitionname}
  \theoremstyle{plain}
  \newtheorem{lem}[thm]{\protect\lemmaname}

\makeatother

  \providecommand{\corollaryname}{Corollary}
  \providecommand{\definitionname}{Definition}
  \providecommand{\lemmaname}{Lemma}
\providecommand{\theoremname}{Theorem}

\begin{document}
\begin{singlespace}

\title{LARGE DEVIATIONS FOR GENERALIZED POLYA URNS WITH ARBITRARY URN FUNCTION}
\end{singlespace}

\author{{\footnotesize{}SIMONE FRANCHINI} }

\address{Simone Franchini, Department of Mathematics and Physics, Universita
degli studi Roma Tre, Via della Vasca Navale 84, 00146 Roma}
\begin{abstract}
{\footnotesize{}We consider a generalized two-color Polya urn (black
and withe balls) first introduced by Hill, Lane, Sudderth \cite{Hill Lane Sudderth},
where the urn composition evolves as follows: let $\pi:\left[0,1\right]\rightarrow\left[0,1\right]$,
and denote by $x_{n}$ the fraction of black balls after step $n$,
then at step $n+1$ a black ball is added with probability $\pi\left(x_{n}\right)$
and a white ball is added with probability $1-\pi\left(x_{n}\right)$.
Originally introduced to mimic attachment under imperfect information,
this model has found applications in many fields, ranging from Market
Share modeling to polymer physics and biology.}{\footnotesize \par}

{\footnotesize{}In this work we discuss large deviations for a wide
class of continuous urn functions $\pi$. In particular, we prove
that this process satisfies a Sample-Path Large Deviations principle,
also providing a variational representation for the rate function.
Then, we derive a variational representation for the limit 
\[
\phi\left(s\right)=\lim_{n\rightarrow\infty}{\textstyle \frac{1}{n}}\log\mathbb{P}\left(nx_{n}=\left\lfloor sn\right\rfloor \right),\, s\in\left[0,1\right],
\]
where $nx_{n}$ is the number of black balls at time $n$, and use
it to give some insight on the shape of $\phi\left(s\right)$. Under
suitable assumptions on $\pi$ we are able to identify the optimal
trajectory. We also find a non-linear Cauchy problem for the Cumulant
Generating Function and provide an explicit analysis for some selected
examples. In particular we discuss the linear case, which embeds the
Bagchi-Pal Model \cite{Bagchi-Pal}, giving the exact implicit expression
for $\phi$ in terms of the Cumulant Generating Function.}{\footnotesize \par}
\end{abstract}
\maketitle
\tableofcontents{}

\section{Introduction.}

\label{Section1}Urns\textbf{}%
\footnote{\textit{Key words and phrases:} large deviations, urn models, Markov
chains %
}\textbf{}%
\footnote{\textit{AMS 2010 subject classifications:} primary 60J10, secondary
60J80%
} are simple probabilistic models that had a broad theoretical development
and applications for several decades, gaining a prominent position
within the framework of adaptive stochastic processes. In general,
single-urn schemes are Markov chains that start with a set (urn) containing
two or more elements of different types: at each step a number of
elements is added or removed with some probabilities depending on
the composition of the urn. Since their introduction these models
where intended to describe phenomena where an underlying tree growth
is present \cite{Pemantle,Mahmoud,Johnson_Koz,MahmoudBook}.

Given the general definition above, an impressive number of variants
have been introduced, depending on the number of colors, extraction
and replacement rules, etc. This work focuses on Large Deviations
Principles (LDP) for a generalization of the classical Polya-Eggenberger
two-colors urn scheme, first introduced by Hill, Lane and Sudderth
\cite{Hill Lane Sudderth,Hill Lane Sudderth 2}. Let us consider an
infinite capacity urn which contains two kinds of elements, say black
and white balls, and denote by $X_{n}:=\left\{ X_{n,k}:1\leq k\leq n\right\} $
the number of black balls during the urn evolution from time $0$
to $n$: at time $k$ there are $k$ balls in the urn, $X_{n,k}$
of which are black. Given a map $\pi:\left[0,1\right]\rightarrow\left[0,1\right]$
(usually referred to as urn function) the urn evolves as follows:
let $x_{n,k}:=k^{-1}X_{n,k}$, $1\leq k\leq n$ be the fraction of
black balls in the urn at step $k$, then a new ball is added at step
$k+1$, whose color is black with probability $\pi\left(x_{n,k}\right)$
and white with probability $1-\pi\left(x_{n,k}\right)=\bar{\pi}\left(x_{n,k}\right)$
(hereafter we denote the complementary probability by an upper bar),
\begin{equation}
X_{n,k+1}=\left\{ \begin{array}{l}
X_{n,k}+1\\
X_{n,k}
\end{array}\begin{array}{l}
with\ probability\\
with\ probability
\end{array}\begin{array}{l}
\pi\left(x_{n,k}\right),\\
\bar{\pi}\left(x_{n,k}\right).
\end{array}\right.
\end{equation}
Apart form the wide range of behaviors depending on the choice of
the urn function, which makes this generalized urn scheme challenging
and rich by itself, attention arises from its relevance to branching
phenomena, stochastic approximation and reinforced random walks \cite{Hill Lane Sudderth,Hill Lane Sudderth 2,Gouet,Kotz Balakrishnan,Mahmoud,Pemantle},
as well as in in Market Share modeling \cite{Dosi Ermoliev Kaniovski,Arthur,Arthur Dosi Ermoliev,DosieErmoliev2-1,Ermoliev-Arthur2,Ermoliev-Arthur3}
and other fields \cite{Khanin,Cotar Limic,Olivera,Dinea Frizer Mitzenmacher}.
We remark it has also been generalized to multicolor urns, whose strong
convergence properties have been investigated by Arthur et Al. in
a series of papers \cite{DosieErmoliev2-1,Ermoliev-Arthur2,Ermoliev-Arthur3},
but in the present work we restrict our attention to the two-colors
case. 

The paper is organized as follows: in this introductory section we
briefly review the main known results about the Generalized Polya
(GP) urn of Hill, Lane and Sudderth, discussing the classes of urn
functions we will consider and introducing some notation. Our results
on large deviations are in Section \ref{Section2}: in particular,
we will present our theorems concerning the Sample Path Large Deviations
Principles, a large deviations analysis for the event $\left\{ X_{n,n}=\left\lfloor sn\right\rfloor \right\} $,
$s\in\left[0,1\right]$ and the Cumulant Generating Function (CGF),
also discussing some applications to paradigmatic examples from literature.
All proofs have been collected in a dedicated section (Section \ref{Section3})
which contains almost all the technical features of this work.

\subsection{The urn function $\pi$.}

\label{Section1.11}In the following we formally present the GP urns
of Hill, Lane and Shuddery, and introduce some non-standard notation
which will be useful when dealing with LDPs: we tried to reduce new
notation to minimum, keeping the common urn terminology everywhere
this was possible. 

As we shall see, the initial conditions do not affect the LDPs for
the class of urn functions we will consider, unless the urn has some
intervals of $s$ for which $\pi\left(s\right)=1$ or $0$. Then,
if not specified otherwise, in this work we set $X_{n,1}$ to be a
random variable uniformly distributed on $\left[0,1\right]$ by convention,
ie 
\begin{equation}
\mathbb{P}\left(X_{n,1}\in\left[s_{1},s_{2}\right]\right):=\left|s_{2}-s_{1}\right|,\,\forall\left[s_{1},s_{2}\right]\subset\left[0,1\right].\label{eq:uniformstart}
\end{equation}
We remark that in the above definition $X_{n,1}$ does not represent
the number of black balls at the initial stage of the urn evolution,
it is just a convenient initial condition for the Eq. (\ref{eq:1.2})
below. We will further elaborate the effect of realistic initial conditions
on the LDPs in Section \ref{Section2}, after the statement of Corollary
2. That said, our process $X_{n}:=\left\{ X_{n,k}:1\leq k\leq n\right\} $
is the Markov Chain with transition matrix:
\begin{equation}
\mathbb{P}\left(X_{n,k+1}=X_{n,k}+i|\, X_{n,k}=j\right):=\pi\left(j/k\right)\mathbb{I}_{\left\{ i=1\right\} }+\bar{\pi}\left(j/k\right)\mathbb{I}_{\left\{ i=0\right\} }.\label{eq:1.2}
\end{equation}
We denote by $\delta X_{n}$ the associated sequence $\delta X_{n,k}:=X_{n,k+1}-X_{n,k}\in\left\{ 0,1\right\} $
for $0\leq k\leq n-1$. For notational convenience, the dependence
on $\pi$ is not specified. Throughout this work we will consider
a sub-class \textit{$\mathcal{U}$} of continuous functions $\pi$:
$\left[0,1\right]\rightarrow\left[0,1\right]$ defined as follows:
\begin{defn*}
\textit{We say that $\pi$: $\left[0,1\right]\rightarrow\left[0,1\right]$
continuous belongs to $\mathcal{U}$ if some function $f>0$ with
\begin{equation}
\lim_{\epsilon\rightarrow0}\epsilon\int_{\epsilon}^{1}{\textstyle dz}\, f\left(z\right)/z^{2}=0
\end{equation}
exists such that $\left|\pi\left(x+\delta\right)-\pi\left(x\right)\right|\leq f\left(\left|\delta\right|\right)$
for $\delta\rightarrow0$, $x\in\left[0,1\right]$. For example, in
the Polya-Eggenberger urn we can take  $f\left(z\right)=z$ and the
above condition becomes $\lim_{\epsilon\rightarrow0}\epsilon\log\left(\epsilon\right)=0$.}
\end{defn*}
Even if this class of functions is slightly smaller than those considered
in \cite{Hill Lane Sudderth,Pemantle,Mahmoud,Pemantle  2}, where
most results are obtained for continuous functions, it still includes
all Lipschitz and $\alpha-$H�lder functions. This class has been
constructed to include most of the interesting cases that can be described
by urn functions while keeping properties that allow a reasonably
straight application of the Varadhan lemma. We will discuss this in
Section \ref{Section3}.

In the following we introduce some new notation which is intended
to ease the description of our results, as well as the limit properties
of $X_{n}$. Define the following sets: 
\begin{equation}
C_{\pi}:=\left\{ s\in\left(0,1\right):\,\pi\left(s\right)=s\right\} ,\,\partial C_{\pi}:=C_{\pi}\setminus\mathrm{int}\left(C_{\pi}\right),\label{eq:1.4}
\end{equation}
where $\mathrm{int\left(C_{\pi}\right)}$ is the interior of $C_{\pi}$.
We will refer to the elements of $C_{\pi}$ as \textit{contacts}.
Note that for the considered urn functions $C_{\pi}$ may not be a
set of isolated points, since our definition of $\mathcal{U}$ allows
$\pi\left(s\right)=s$ for some interval $s\in\left[s_{1},s_{2}\right]$
(see the region $K_{\pi,3}$ in Figure 1.1). On the contrary $\partial C_{\pi}$
is always a finite set of isolated points since it collects the boundaries
of the regions in which $\pi\left(s\right)-s$ has a definite sign.
We denote by $N:=\left|\partial C_{\pi}\right|$ the number of such
points in $\partial C_{\pi}$ for a given $\pi$.

We can further distinguish the elements of $\partial C_{\pi}$ by
considering the behavior of $\pi\left(s\right)$ in their neighborhood:
to do so, we will introduce a partition of the interval $\left[0,1\right]$.
We remark that the notation we are going to define is not a standard
of urn literature, but it will prove useful in describing of our results
when dealing with optimal trajectories. First, let us organize the
elements of $\partial C_{\pi}$ by increasing order, labeling them
as
\begin{equation}
\partial C_{\pi}=:\left\{ s_{i},1\leq i\leq N:\, s_{i}<s_{i+1}\right\} .\label{eq:1.4.11}
\end{equation}
Then, we can define the following sequence of intervals (see Figure
\ref{fig:1})
\begin{equation}
K_{\pi}:=\left\{ K_{\pi,i},\,0\leq i\leq N:\, K_{\pi,0}:=\left(0,s_{1}\right),K_{\pi,N}:=\left(s_{N},1\right),\, K_{\pi,j}:=\left(s_{j},s_{j+1}\right)\right\} .\label{eq:1.4.1}
\end{equation}
By definition of $\partial C_{\pi}$$,$ the above intervals are such
that $\pi\left(s\right)-s$ does not change sign for $s\in K_{\pi,i}$.
Then we can associate a variable $a_{\pi,i}\in\left\{ -1,0,1\right\} $
to each interval $K_{\pi,i}$ which expresses the sign of $\pi\left(s\right)-s$.
We denote such sequence by 
\begin{equation}
A_{\pi}:=\{a_{\pi,i},\,0\le i\leq N:\, a_{\pi,i}={\textstyle \frac{\pi\left(s\right)-s}{\left|\pi\left(s\right)-s\right|}\mathbb{I}_{\left\{ \pi\left(s\right)\neq s\right\} }},\, s\in K_{\pi,i}\}.\label{eq:1.4.2}
\end{equation}
Some words should be spent on the correct use of this notation when
the urn function has $\pi\left(0\right)=0$ or $\pi\left(1\right)=1$,
or both. Consider the first case: if $\pi\left(0\right)=0$ then the
smallest element of $\partial C_{\pi}$ is $s_{1}=0$. Following our
definition of $K_{\pi,0}$ as open interval we would have that $K_{\pi,0}=\textrm{�}$
and $a_{\pi,0}$ not well defined. To patch this, we set by convention
that $a_{\pi,0}=1$ if $K_{\pi,0}=\textrm{�}$ and $a_{\pi,N}=-1$
if $K_{\pi,N}=\textrm{�}$.

\begin{figure}[h]
\begin{singlespace}
\centering{}~~~\\
 ~\includegraphics[scale=3]{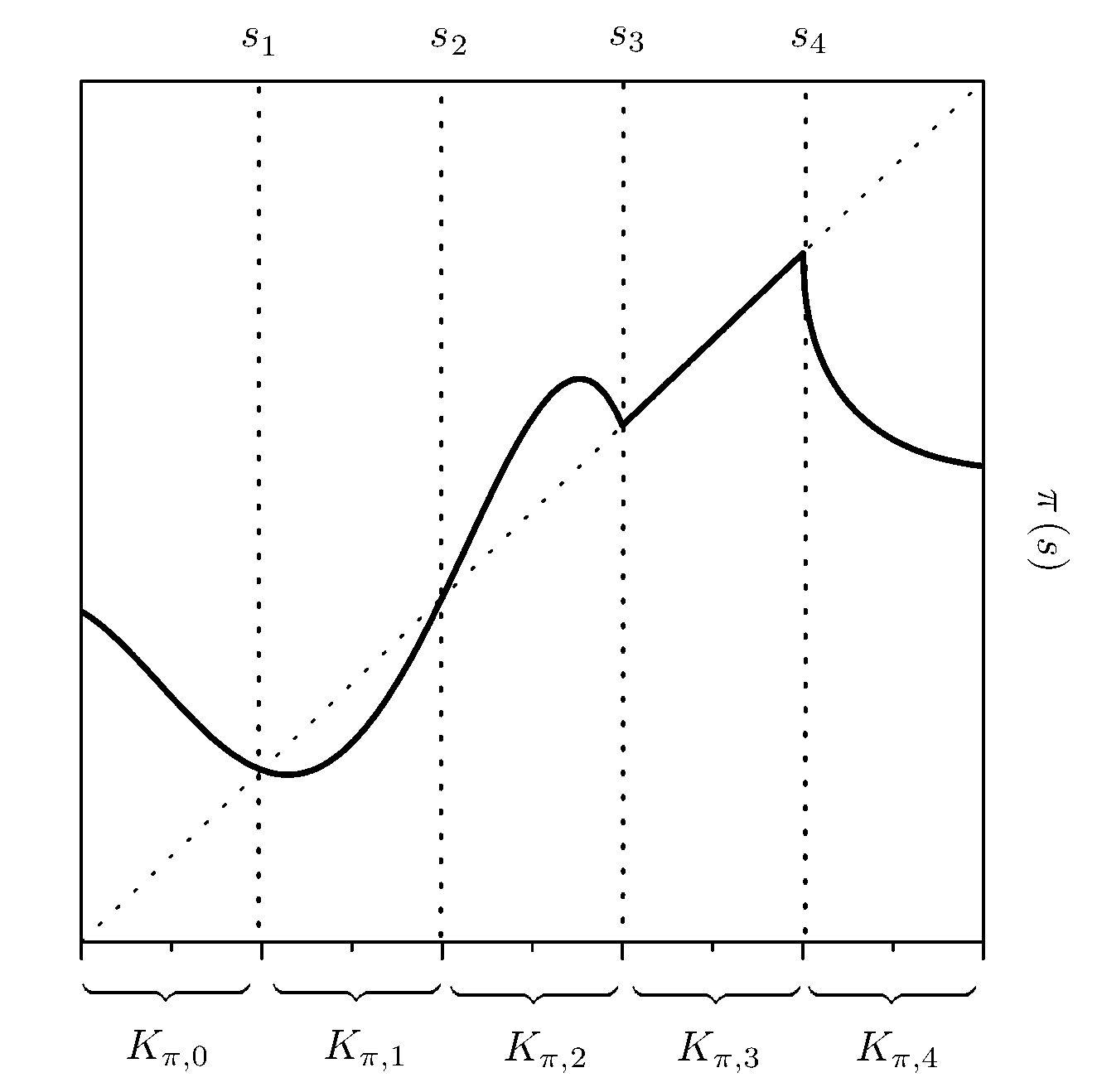}\caption{\textit{\label{fig:1}Example of urn function $\pi\in\mathcal{U}$
to illustrate the notation introduced in Eq.s (\ref{eq:1.4.1}), (\ref{eq:1.4.11}).
For the function above we have $C_{\pi}=\left\{ 1/5,\,2/5,\,3/5,\,4/5\right\} \cup\left(3/5,\,4/5\right)$,
then $\partial C_{\pi}=\left\{ 1/5,\,2/5,\,3/5,\,4/5\right\} $, $K_{\pi,0}=\left[0,1/5\right)$,
$K_{\pi,4}=\left(4/5,\,1\right]$, $K_{\pi,i}:=\left(i/5,\,\left(i+1\right)/5\right)$,
$i\in\left\{ 1,2,3\right\} $ and $A_{\pi}=\left\{ 1,\,-1,\,1,\,0,\,-1\right\} $.
Also, $s_{1}=1/5\in C_{\pi}\left(+,-\right)$ is a downcrossing, $s_{2}=2/5\in C_{\pi}\left(-,+\right)$
is an upcrossing, $s\in\left(3/5,4/5\right)$ is a dense region of
Polya-like contacts $C_{\pi}\left(0,0\right)$ while $s_{3}=3/5$,
$s_{4}=4/5$ are its left and right boundaries $C_{\pi}\left(+,0\right)$,
$C_{\pi}\left(0,-\right)$ respectively. }}
\end{singlespace}
\end{figure}

Using the above notation we can now define the subsets $C_{\pi}\left(\alpha,\beta\right)$
of those $s\in\partial C_{\pi}$ such that $\alpha\in\left\{ +,0,-\right\} $
is the sign of $\pi{\textstyle \left(s'\right)-s'}$ for $s'-s\rightarrow0^{-}$
and $\beta\in\left\{ +,0,-\right\} $ is the sign of $\pi{\textstyle \left(s'\right)-s'}$
for $s'-s\rightarrow0^{+}$. 
\begin{equation}
C_{\pi}\left(\alpha,\beta\right):=\left\{ s_{i}\in\partial C_{\pi}:\,\mathrm{sign}\left(a_{\pi,i-1}\right)=\alpha,\,\mathrm{sign}\left(a_{\pi,i}\right)=\beta\right\} 
\end{equation}
References \cite{Hill Lane Sudderth,Pemantle,Mahmoud,Pemantle  2}
call $C_{\pi}\left(+,-\right)$ and $C_{\pi}\left(-,+\right)$ respectively
\textit{downcrossings} and \textit{upcrossings}, while $C_{\pi}\left(+,+\right)$
and $C_{\pi}\left(-,-\right)$ are \textit{touchpoints.} Note that
our classification also allows contacts of the kind $C_{\pi}\left(\alpha,0\right)$
and $C_{\pi}\left(0,\beta\right)$, which are the boundaries of those
intervals $K_{\pi,i}$ for which $\pi\left(s\right)=s$ ($a_{\pi,i}=0$).

\subsection{Strong convergence.}

\label{Section1.2}Here we review some of the main known results on
\textit{strong convergence}, ie, on the almost sure convergence of
$x_{n,n}$. This topic has been widely investigated in \cite{Hill Lane Sudderth,Hill Lane Sudderth 2,Gouet,Pemantle  2,Mahmoud,Pemantle}).
As example, consider the simplest non trivial urn model, the so called
Polya-Eggenberger urn \cite{Polya Eggenberger}, which evolves as
follows: at each step draw a ball, if it is black then add a black
ball, and add a white one otherwise. This urn is represented in our
context by the urn function $\pi\left(s\right)=s$. In this case $\mathbb{E}\left(x_{n,k+1}|x_{n,k}\right)=x_{n,k}$,
so that $x_{n,k}$ is a martingale and $\lim_{n}x_{n,n}$ exists almost
surely. 

The existence of $\lim_{n}x_{n,n}$ has been shown in \cite{Hill Lane Sudderth}
for a wider class of urn functions (including some non-continuous
$\pi$). In \cite{Hill Lane Sudderth} it has been shown that if $\pi$
is a continuous function then $\lim_{n}x_{n,n}$ exists almost surely,
and $\lim_{n}x_{n,n}\in C_{\pi}$. The same result holds if $\pi$
is non-continuous, provided the points $s$ where $\pi\left(s\right)-s$
oscillates in sign are not dense in an interval. 

Clearly, not all the points of $C_{\pi}$ can be the limit of $x_{n,n}$
and several efforts were made to determine whether a point belongs
to the support of $\lim_{n}x_{n,n}$ for a given $\pi$ \cite{Hill Lane Sudderth,Pemantle  2}.
We say that $s\in\left[0,1\right]$ belongs to the support of $\lim_{n}x_{n,n}$
if $\mathbb{P}\left(\left|\lim_{n}x_{n,n}-s\right|<\delta\right)>0$,
$\forall\delta>0$. In general, we can summarize from \cite{Hill Lane Sudderth,Pemantle  2}
what is known about the support of $\lim_{n}x_{n,n}$ in our setting
($\pi\in\mathcal{U}$ and $X_{n,1}$ uniform on $\left[0,1\right]$).
Let $X_{n}$ be the urn process generated by the urn function $\pi\in\mathcal{U}$,
and define $\Delta_{\pi,\epsilon}\left(s\right):=\epsilon^{-1}\left[\pi\left(s+\epsilon\right)-\pi\left(s\right)\right]$.
Then the limit $\lim_{n}x_{n,n}$ exists almost surely and
\begin{enumerate}
\item Downcrossings $C_{\pi}\left(+,-\right)$ always belong to the support
of $\lim_{n}x_{n,n}$ while upcrossings $C_{\pi}\left(-,+\right)$
never do. 
\item If $s\in C_{\pi}\left(+,+\right)$, then it belongs to the support
of $\lim_{n}x_{n,n}$ if and only if some $\delta>0$ exists such
that $\Delta_{\pi,\epsilon}\left(s\right)\in\left(1/2,1\right)$ for
$\epsilon\in\left(-\delta,0\right)$.
\item If $s\in C_{\pi}\left(-,-\right)$, then it belongs to the support
of $\lim_{n}x_{n,n}$ if and only if some $\delta>0$ exists such
that $\Delta_{\pi,\epsilon}\left(s\right)\in\left(1/2,1\right)$ for
$\epsilon\in\left(0,\delta\right)$.
\end{enumerate}
The proof that downcrossings belong to the support of $\lim_{n}x_{n,n}$
while upcrossings don't can be found in reference \cite{Hill Lane Sudderth}:
it involves Markov chain coupling together with martingale analysis.
The statement that touchpoints $C_{\pi}\left(+,+\right)$ with $1/2<\Delta_{\pi,\epsilon}\left(s\right)<1$
from the left ($\epsilon<0$) and $C_{\pi}\left(-,-\right)$ with
$1/2<\Delta_{\pi,\epsilon}\left(s\right)<1$ from the right ($\epsilon>0$)
belong to the support of $\lim_{n}x_{n,n}$ has been proved in \cite{Pemantle  2}
by Pemantle. This seemingly paradoxical statement is actually a deep
observation about the dynamics of the process: if the condition on
$\Delta_{\pi,\epsilon}\left(s\right)$ is fulfilled, then $x_{n,n}$
converges so slowly to $s\in C_{\pi}\left(+,+\right)$ from the left
(to $s\in C_{\pi}\left(-,-\right)$ from the right) that it almost
surely never crosses this point, accumulating in its left (right)
neighborhood. If not, then $x_{n,n}$ crosses $s$ in finite time
almost surely, and gets pushed away from the other side toward the
closest stable equilibrium (ie, the closest point that belongs to
the support of $\lim_{n}x_{n,n}$ ).

Even if we left out the cases $C_{\pi}\left(\alpha,0\right)$, $C_{\pi}\left(0,\beta\right)$
and $s\in K_{\pi,i}$ with $a_{\pi,i}=0$ from the above statement
it is clear that they always belong to the support of $\lim_{n}x_{n,n}$
since in some neighborhood of these points the process behaves like
a Polya-Eggenberger urn.

We remark that almost sure convergence is strongly affected by initial
conditions: since a detailed discussion of this topic would be far
from the scope of this work, we defer to the reviews \cite{Hill Lane Sudderth,Pemantle,Mahmoud,Pemantle  2}.

\section{Main results.}

\label{Section2}While the almost sure convergence properties of such
urns are quite well understood also in multicolor generalizations
(see \cite{DosieErmoliev2-1,Ermoliev-Arthur2,Ermoliev-Arthur3}),
Large Deviations properties are not. Apart from the Polya-Eggenberger
urn, for which we can explicitly compute the exact urn composition
at each time, to the best of our knowledge large deviations results
in urn models have been pioneered by Flajolet et Al. \cite{Fajolet-Analytic Urns,Fajolet2,Panhozer},
which provided a detailed analysis of the Bagchi-Pal urn using generating
function methods. Since then other authors extended this approach
to many related models (of particular interest is \cite{Stochastic urns},
a Bagchi-Pal urn with stochastic reinforcement matrix). Another early
work on Large Deviations has been provided by Bryc et Al. in \cite{Bryc},
where a special Bagchi-Pal type urn is studied as model for preferential
attachment and an explicit expression of the Cumulant Generating Function
is obtained in integral form (see the end of this section for an introduction
to the Bagchi-Pal model).

This section mostly contains the statements of our results. Most of
the proofs of the following statements are grouped in Section \ref{Section3}:
we will specify where to find them.

\subsection{Sample-Path Large Deviation Principle.}

\label{Section2.1}As preliminary result, we need a Sample-Path Large
Deviation principle which holds for any $\pi\in\mathcal{U}$. Then,
define the function $\chi_{n}:\left[0,1\right]\rightarrow\left[0,1\right]$
as follows: 
\begin{equation}
\chi_{n}:=\left\{ \chi_{n,\tau}=n^{-1}\left[X_{n,\left\lfloor n\tau\right\rfloor }+\left(n\tau-\left\lfloor n\tau\right\rfloor \right)\delta X_{n,\left\lfloor n\tau\right\rfloor }\right]:\tau\in\left[0,1\right]\right\} ,\label{eq:2.1-2}
\end{equation}
where $\left\lfloor \cdot\right\rfloor $ denotes the lower integer
part, and introduce the subspace of Lipschitz-continuous functions
\begin{equation}
\mathcal{Q}:=\left\{ \varphi\in C\left(\left[0,1\right]\right):\,\varphi_{0}=0,\,\varphi_{\tau+\delta}-\varphi_{\tau}\in\left[0,\delta\right],\,\delta>0,\,\tau\in\left[0,1\right]\right\} ,\label{eq:Q-def}
\end{equation}
where $C\left(\left[0,1\right]\right)$ is the set of continuous functions
on $\left[0,1\right]$. Denote by $\left\Vert \varphi\right\Vert :=\sup_{\tau\in\left[0,1\right]}\left|\varphi_{\tau}\right|$
the usual supremum norm, and consider the normed metric space $\left(\mathcal{Q},\,\left\Vert \cdot\right\Vert \right)$.
We show that a good rate function $I_{\pi}:\mathcal{Q}\rightarrow\left[0,\infty\right)$
exists such that for every Borel subset $\mathcal{B}\subseteq\mathcal{Q}$:
\begin{equation}
\liminf_{n\rightarrow\infty}\, n^{-1}\log\mathbb{P}\left(\chi_{n}\in\mathrm{int}\left(\mathcal{B}\right)\right)\geq-\inf_{\varphi\in\mathrm{int}\left(\mathcal{B}\right)}I_{\pi}\left[\varphi\right],\label{eq:1.8}
\end{equation}
\begin{equation}
\limsup_{n\rightarrow\infty}\, n^{-1}\log\mathbb{P}\left(\chi_{n}\in\ \mathrm{cl}\left(\mathcal{B}\right)\right)\leq-\ \,\inf_{\varphi\in\mathrm{cl}\left(\mathcal{B}\right)}I_{\pi}\left[\varphi\right].\label{eq:1.9}
\end{equation}
To describe the rate function we introduce a functional $S_{\pi}:\mathcal{Q}\rightarrow\left(-\infty,0\right]$,
defined as follows: 
\begin{equation}
S_{\pi}\left[\varphi\right]:=\int_{\tau\in\left[0,1\right]}\left[\, d\varphi_{\tau}\,\log\pi\left(\varphi_{\tau}/\tau\right)+d\tilde{\varphi}_{\tau}\,\log\bar{\pi}\left(\varphi_{\tau}/\tau\right)\right],\label{eq:1.10}
\end{equation}
where we denoted $\bar{\pi}\left(s\right)=1-\pi\left(s\right)$ and
$\tilde{\varphi}_{\tau}=\tau-\varphi_{\tau}$. Then, the following
theorem gives the Sample-Path LDP for $\chi_{n}$:
\begin{thm}
\label{Theorem 1}Let $\pi\in\mathcal{U}$, $\varphi\in\mathcal{Q}$,
define the function $H\left(s\right):=s\log s+\bar{s}\log\bar{s}$,
and the functional $J:\mathcal{Q}\rightarrow\left[-\log2,\infty\right)$
as follows: 
\begin{equation}
J\left[\varphi\right]=\left\{ \begin{array}{l}
\int_{0}^{1}d\tau\, H\left(\dot{\varphi}_{\tau}\right)\\
\infty
\end{array}\ \begin{array}{l}
if\ \varphi\in\mathcal{AC}\\
otherwise,
\end{array}\right.
\end{equation}
where $\mathcal{AC}$ is the class of absolutely continuous functions
(we assume the same definition given in Theorem 5.1.2 of \cite{Dembo Zeitouni})
and $\dot{\varphi}_{\tau}:=\frac{d\varphi_{\tau}}{d\tau}$. Also,
define the good rate function
\begin{equation}
I_{\pi}\left[\varphi\right]=J\left[\varphi\right]-S_{\pi}\left[\varphi\right],
\end{equation}
with $S_{\pi}$ as in Eq. (\ref{eq:1.10}). Then, the law of $\chi_{n}$
with initial condition $X_{n,1}$ of Eq. (\ref{eq:uniformstart})
uniformly distributed on the interval $\left[0,1\right]$ satisfies
a Sample-Path LDP as in Eq.s (\ref{eq:1.8}) and (\ref{eq:1.9}),
with good rate function $I_{\pi}\left[\varphi\right]$. 
\end{thm}
The proof is quite standard, and based on a change of measure and
an application of the Varadhan Integral Lemma plus some surgery on
the set $\mathcal{Q}$ to a priori exclude those trajectories which
create issues in proving the continuity of $S_{\pi}\left[\varphi\right]$
on $\left(\mathcal{Q},\left\Vert \cdot\right\Vert \right)$ (see the
approximation argument of Lemma \ref{Lemma 5}). 

Let us now consider a process with some specific initial condition,
say $X_{n,m}=X_{m}^{*}$ for some $0<m\leq n$ and $0\leq X_{m}^{*}\leq m$.
If we call by $\chi_{n}^{*}$ a process defined as in Eq. (\ref{eq:2.1-2})
with the additional condition $\mathbb{P}\left(\chi_{n,m/n}=n^{-1}X_{m}^{*}\right)=1$,
then we can resume the effects of such constraint in the following
corollary
\begin{cor}
\label{cor:CorrollariINIT}Let $\pi\in\mathcal{U}$ and denote by
$\chi_{n}^{*}$ a process defined as in Eq. (\ref{eq:2.1-2}) with
the additional condition that $\chi_{n,m/n}^{*}=n^{-1}X_{m}^{*}$
for some $0<m\leq n$ and $0\leq X_{m}^{*}\leq m$. Define $0\leq z_{-}^{*}<z_{+}^{*}\leq1$
as follows 
\begin{equation}
z_{-}^{*}:=\liminf_{n\rightarrow\infty}\left\{ z_{-}:\,\mathbb{P}\left(X_{n,n}\leq z_{-}n\,|\, X_{n,m}=X_{m}^{*}\right)>0\right\} ,
\end{equation}
\begin{equation}
z_{+}^{*}:=\limsup_{n\rightarrow\infty}\left\{ z_{+}:\,\mathbb{P}\left(X_{n,n}\geq z_{+}n\,|\, X_{n,m}=X_{m}^{*}\right)>0\right\} ,
\end{equation}
and a modified urn function $\pi^{*}$ 
\begin{equation}
\pi^{*}\left(s\right):=\mathbb{I}_{\{s\in\left[0,z_{-}^{*}\right)\}}+\pi\left(s\right)\mathbb{I}_{\{s\in\left[z_{-}^{*},z_{+}^{*}\right]\}}.
\end{equation}
Then, the law of $\chi_{n}^{*}$ with initial condition $x_{n,m}=m^{-1}X_{m}^{*}$
satisfies a Sample-Path LDP with good rate function $I_{\pi^{*}}$,
as for $\chi_{n}$ with $X_{n,1}$ uniform on $\left[0,1\right]$
and $\pi^{*}$ in place of $\pi$.
\end{cor}
The above results tell us that initial conditions of the kind $\mathbb{P}\left(X_{n,m}=k\right)=\mathbb{I}_{\left\{ k=X_{m}^{*}\right\} }$
can affect the rate function if and only if $\pi\left(s\right)$ is
$0$ or $1$ for some values of $s$. We can easily convince ourselves
of this by observing that if $\pi\in\left(0,1\right)$ then $X_{n,n}$
can reach any point in $\left\{ X_{m},X_{m}+1,\,...\,,X_{m}+\left(n-m\right)\right\} $
in finite time $n-m$ from $X_{n,m}$, while the presence of intervals
with $\pi\left(s\right)=0$ or $1$ can prevent the process from crossing
some values. The proof of the above corollary is in Section \ref{Section3.1.1}.
Notice that we can define $z_{-}^{*}$ and $z_{+}^{*}$ also for $X_{n,1}$
uniform on $\left[0,1\right]$, and in this case we can take 
\begin{equation}
z_{-}^{*}:={\textstyle \inf}\left\{ s:\pi\left(s\right)<1\right\} ,\, z_{+}^{*}:=\sup\left\{ s:\pi\left(s\right)>0\right\} .
\end{equation}
In the following we will consider the above definition, unless some
different initial condition is specified.

Before going ahead some words should be spent on non homogeneous urn
functions. Then, take $\pi\in\mathcal{U}$ with $\pi\in\left(0,1\right)$
and consider a sequence of urn functions $\left\{ \pi_{n}\in\mathcal{U}:\, n\geq0\right\} $
such that for every $n\geq0$ we have $\pi_{n}\left(s\right)\in\left(0,1\right)$
for $s\in\left[0,1\right]$ and $\pi_{n}\rightarrow\pi$ uniformly
on $\left[0,1\right]$. In Section \ref{Section3.1.1} we show that
\begin{cor}
\label{cor:Coroll inhomog.}Take $\pi\in\mathcal{U}$ with $\pi\left(s\right)\in\left(0,1\right)$
and let $\pi_{n}\in\mathcal{U}$ such that $\pi_{n}\left(s\right)\in\left(0,1\right)$
and $\left|\pi_{n}\left(s\right)-\pi\left(s\right)\right|\leq\delta_{n}$,
$\lim_{n}\delta_{n}=0$ for all $s\in\left[0,1\right]$. Then, the
non homogeneous urn process defined by $\pi_{n}$ satisfies the same
Sample-Path LDP of $\pi$.
\end{cor}
We restricted our statement to urns with $\pi\left(s\right)\in\left(0,1\right)$,
$\pi_{n}\left(s\right)\in\left(0,1\right)$ to avoid some technical
issues which would arise if we consider the whole set $\mathcal{U}$,
but it is possible to generalize this result on the basis of the same
considerations made for Theorem \ref{Theorem 1}. We hope to address
this extension in a future work.

\subsection{Entropy of the event $X_{n,n}=\left\lfloor sn\right\rfloor $.}

\label{Section2.2}Our main interest in Theorem \ref{Theorem 1} comes
from the fact that Sample-Path LDPs allow to approach some important
Large Deviation questions about the urn evolution from the point of
view of functional analysis. In this work our attention will mainly
focus on the entropy of the event $X_{n,n}=\left\lfloor sn\right\rfloor $,
$s\in\left[0,1\right]$. First we show that the limit 
\begin{equation}
\phi\left(s\right):=\lim_{n\rightarrow\infty}n^{-1}\log\mathbb{P}\left(X_{n,n}=\left\lfloor sn\right\rfloor \right),\label{eq:1.15}
\end{equation}
exists for every $\pi\in\mathcal{U}$, and has the following variational
representation:
\begin{thm}
\label{Theorem 2}The limit $\phi\left(s\right)$ defined in Eq. (\ref{eq:1.15})
exists for any $\pi\in\mathcal{U}$ and is given by the variational
problem 
\begin{equation}
\phi\left(s\right)=-\inf_{\varphi\in\mathcal{Q}_{s}}I_{\pi}\left[\varphi\right],
\end{equation}
where $\mathcal{Q}_{s}:=\left\{ \varphi\in\mathcal{Q}:\,\varphi_{1}=s\right\} $
and $I_{\pi}$ is the rate function of Theorem \ref{Theorem 1}. If
we consider an initial condition $\chi_{n,m/n}^{*}=n^{-1}X_{m}^{*}$
for some $0<m\leq n$ and $0\leq X_{m}^{*}\leq m$ the same result
holds with $I_{\pi^{*}}$ in place of $I_{\pi}$ and $\pi^{*}$ as
in Corollary \ref{cor:CorrollariINIT}.
\end{thm}
Notice that Theorem \ref{Theorem 1} can not be directly applied to
the Eq. (\ref{eq:1.15}) in order to obtain Theorem \ref{Theorem 2},
since this is a stronger statement than what one obtains by the contraction
principle. To prove Theorem \ref{Theorem 2} we integrated Theorem
\ref{Theorem 1} with a combinatorial argument: the proof can be found
in Section \ref{Section3.2.1}.

\subsubsection{Optimal trajectories.}

\label{Section2.2.1}Since the variational problem in Theorem \ref{Theorem 2}
heavily depends on the choice of $\pi$, a general characterization
of $\phi\left(s\right)$ would be a quite hard nut to crack. Anyway,
we still can prove many interesting facts on the shape of $\phi\left(s\right)$.
Most important, we can prove that $\phi\left(s\right)=0$ when $s\in\left[\,\inf C_{\pi},\,\sup C_{\pi}\right]$
and $\phi\left(s\right)<0$ otherwise. 
\begin{cor}
\label{Corollary 3}For any $\pi\in\mathcal{U}$: $\phi\left(s\right)=0$
when $s\in\left[\,\inf\, C_{\pi},\,\sup\, C_{\pi}\right]$ and $\phi\left(s\right)<0$
otherwise, where $C_{\pi}$ is the contact set of $\pi$ defined by
Eq. (\ref{eq:1.4}). Moreover, $\phi\left(s\right)>-\infty$ for $s\in\left(z_{-}^{*},\inf C_{\pi}\right)$
and $s\in\left(\sup\, C_{\pi},z_{+}^{*}\right)$, while $\phi\left(s\right)=-\infty$
for $s\in\left[0,z_{-}^{*}\right]$ and $s\in\left[z_{+}^{*},1\right]$.
\end{cor}
The above corollary is obtained by proving that we can find a trajectory
$\varphi^{*}\in\mathcal{Q}_{s}$ such that $I_{\pi}\left[\varphi^{*}\right]=0$
for any $s\in\left[\,\inf C_{\pi},\,\sup C_{\pi}\right]$, while this
is not possible if $s\in K_{\pi,0}$ or $K_{\pi,N}$. Also, we are
able to give an explicit characterization of the optimal trajectories
$\varphi^{*}$. We enunciate this result in two separate corollaries:
the first deals with trajectories that end in $s\in K_{\pi,i}$, $1\leq i\leq N-1$,
while the second deals with trajectories that end in $s\in\partial C_{\pi}$
(as we shall see, Corollary \ref{Corollary 3} is an almost direct
consequence of the following two) .
\begin{cor}
\label{Corollary4}Let $K_{\pi}$, $A_{\pi}$ be as in Eq.s (\ref{eq:1.4.1}),
(\ref{eq:1.4.2}). For any $s\in K_{\pi,i}$ a zero-cost trajectory
$\varphi^{*}\in\mathcal{Q}_{s}$ with \textup{$\tau^{-1}\varphi_{\tau}^{*}\in K_{\pi,i}\cup\partial K_{\pi,i}$,
$\tau\in\left[0,1\right]$} exists such that $I_{\pi}\left[\varphi^{*}\right]=0$,
and it can be constructed as follows. If $a_{\pi,i}=0$ then we can
take $\varphi^{*}=s\tau$ as in the Polya-Eggenberger urn. If $a_{\pi,i}\neq0$
let 
\begin{equation}
F_{\pi}\left(s,u\right):=\int_{u}^{s}\frac{dz}{\pi\left(z\right)-z}.
\end{equation}
Also, for $s\in K_{\pi,i}$ define the constants
\begin{equation}
s_{i}^{*}:=\mathbb{I}_{\left\{ a_{\pi,i}=1\right\} }\inf K_{\pi,i}+\mathbb{I}_{\left\{ a_{\pi,i}=-1\right\} }\sup K_{\pi,i},
\end{equation}
\begin{equation}
\tau_{s,i}^{*}:=\exp\left(-{\textstyle \lim_{\, a_{\pi,i}(u-s_{i}^{*})\rightarrow0^{+}}\left|F_{\pi}\left(s,u\right)\right|}\right).
\end{equation}
and denote by $F_{\pi,s}^{-1}$ the inverse function of $F_{\pi}\left(s,u\right)$
for \textup{$u\in K_{\pi,i}\cup\partial K_{\pi,i}$: }
\begin{equation}
F_{\pi,s}^{-1}:=\left\{ F_{\pi,s}^{-1}\left(q\right),\, q\in\left[0,\log\left(1/\tau_{s,i}^{*}\right)\right):\, F_{\pi}\left(s,F_{\pi,s}^{-1}\left(q\right)\right)=q\right\} .
\end{equation}
Then, if $a_{\pi,i}\neq0$ the zero-cost trajectory is given by $\varphi_{\tau}^{*}=\tau u_{\tau}^{*}$,
with
\begin{equation}
u_{\tau}^{*}:=F_{\pi,s}^{-1}\left(\log\left(1/\tau\right)\right)\,\mathbb{I}_{\left\{ \tau\in\left(\tau_{s,i}^{*},1\right]\right\} }+s_{i}^{*}\,\mathbb{I}_{\left\{ \tau\in\left[0,\tau_{s,i}^{*}\right]\right\} }.
\end{equation}

\end{cor}
The proof relies on the fact that any $\varphi^{*}$ for which $I_{\pi}\left[\varphi^{*}\right]=0$
must satisfy the Homogeneous equation $\dot{\varphi}_{\tau}^{*}=\pi\left(\varphi_{\tau}^{*}/\tau\right)$.
This is shown in Section \ref{Section3.2.2}. 

The above corollary states that the optimal strategy to achieve the
event $\left\{ X_{n,n}=\left\lfloor sn\right\rfloor \right\} $, $s\in\left[\,\inf\, C_{\pi},\,\sup\, C_{\pi}\right]$
emanates from the closest unstable equilibrium point which is on the
left of $s$ if $\pi\left(s\right)$$<s$ and on the right if $\pi\left(s\right)>s$,
see Figure \ref{fig:2} for an example. Notice that $u_{\tau}^{*}$
is always invertible on $(\tau_{s,i}^{*},1]$, since it is strictly
decreasing from $\sup\, K_{\pi,i}$ to $s$ if $a_{\pi,i}=-1$, and
strictly increasing from $\mathrm{inf}\, K_{\pi,i}$ to $s$ if $a_{\pi,i}=1$. 

\begin{figure}
\begin{singlespace}
\centering{}~~~~~~~~~\includegraphics[scale=2.5]{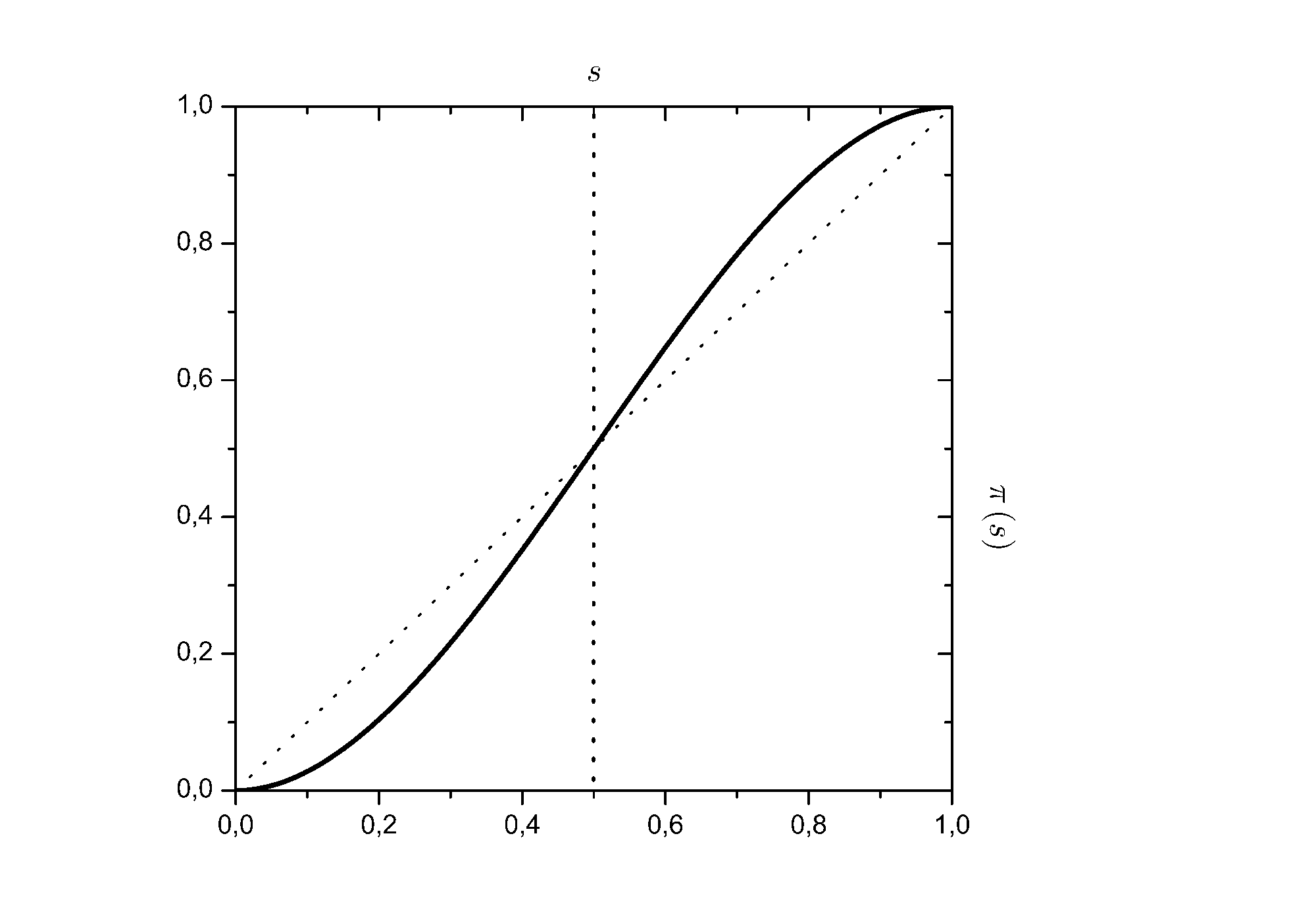}\\
\includegraphics[scale=2.63]{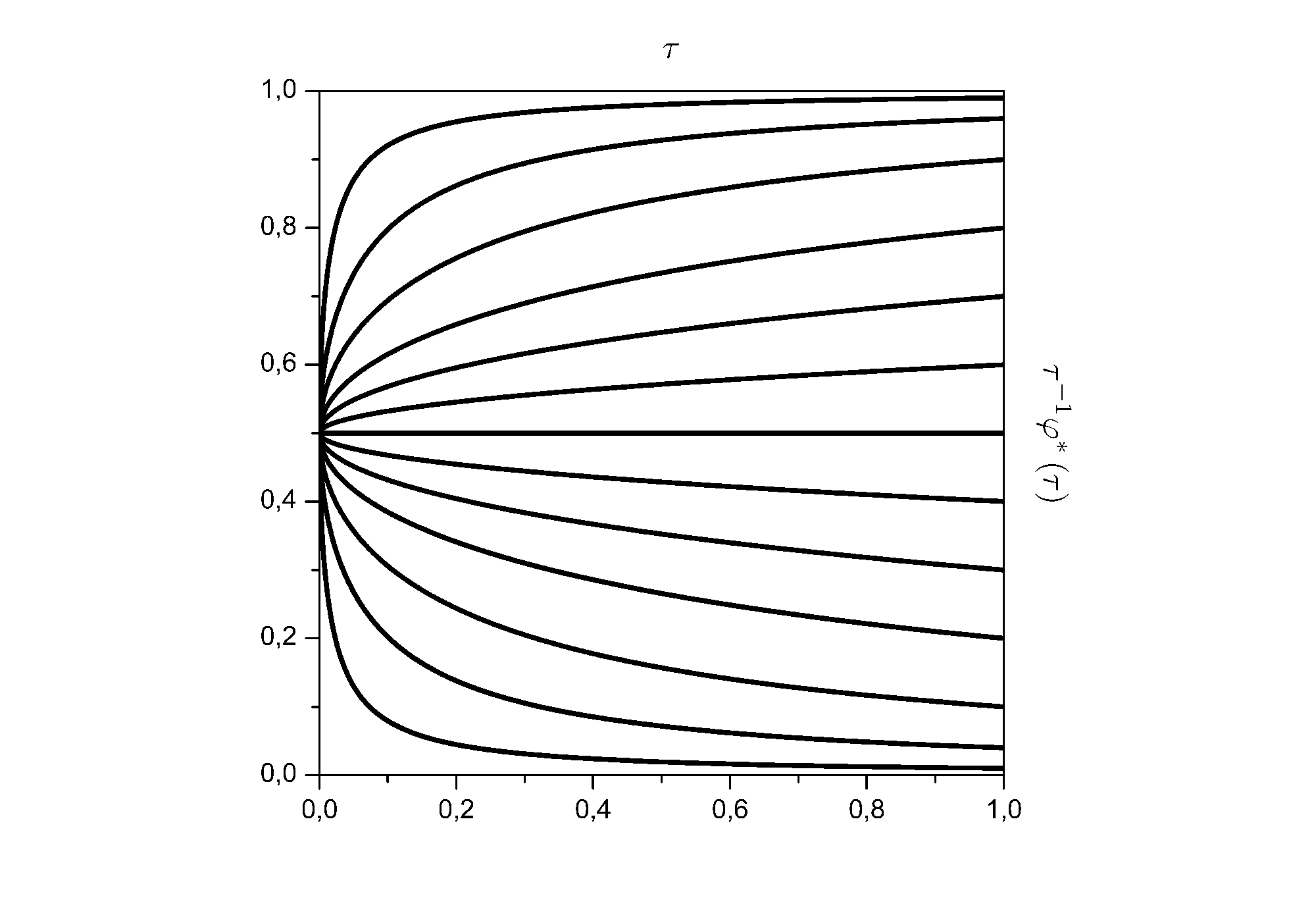}\caption{\textit{\label{fig:2}We can provide an explicit example on how to
use Corollary \ref{Corollary4} using the urn function $\pi\left(s\right)=3s^{2}-2s^{3}$,
that represents an urn process in which at each time three balls are
extracted from the urn, and then a black ball is added if there is
a majority of black balls and a white ball is added otherwise. This
urn has been first introduced by Arthur et Al. in \cite{Arthur} as
a model of market share between two competing commercial products.
We will refer to it as majority urn. Since $3s^{2}-2s^{3}=s$ has
three solutions at $0$, $1/2$ and $1$ we have $K_{\pi,1}=\left(0,1/2\right)$
and $K_{\pi,2}=\left(1/2,1\right)$, with $a_{\pi,1}=-1$ and $a_{\pi,2}=1$
. Applying Corollary \ref{Corollary4} we find that in both cases
$s\in K_{\pi,1}$ and $s\in K_{\pi,2}$ we have $\tau_{s,1}^{*}=0$,
$\tau_{s,2}^{*}=0$, and the optimal trajectory satisfies ${\textstyle 2\tau^{-1}\varphi{}_{\tau}^{*}=}1-(1\pm\varrho\left(s\right)/\tau)^{-1/2}$
, with $\varrho\left(s\right)=4s(1-s)/(2s-1)^{2}$. Notice that much
useful information can be extracted from this curves as they describe
the relative market placement of the considered product at each time
backward on a scale $O\left(n\right)$ by only asking for the final
state $s$. The lower figure shows some zero-cost trajectories of
the above $\pi$ for $s\in\left\{ 0.99,\,0.96,\,0.9,\,0.8,\,0.7,\,0.5,\,0.4,\,0.3,\,0.2,\,0.1,\,0.04,\,0.01\right\} $. }}
\end{singlespace}
\end{figure}

\subsubsection*{Time-inhomogeneous trajectories}

A curious fact is that an optimal trajectory can be time-inhomogeneous
depending on integrability of $1/\left(\pi\left(s\right)-s\right)$
as $s\rightarrow s_{i}^{*}$. If the singularity is integrable (not
the case of Figure \ref{fig:2}) then the equilibrium $s_{i}^{*}$
is so unstable that the processes will leave its neighborhood at some
$\tau_{s,i}^{*}>0$ to end in $s$. We discuss this interpretation
after stating our results for trajectories that end in $s\in\partial C_{\pi}$.
\begin{cor}
\label{Corollary4.1}Let $K_{\pi}$, $A_{\pi}$ as in Eq.s (\ref{eq:1.4.1}),
(\ref{eq:1.4.2}), and consider $K_{\pi,i}$ for some $1\leq i\leq N-1$.
Let $F_{\pi}\left(s,u\right)$ and $s_{i}^{*}$ as in Corollary \ref{Corollary4}
and define
\begin{equation}
s_{i}^{\dagger}:=\mathbb{I}_{\left\{ a_{\pi,i}=-1\right\} }\inf K_{\pi,i}+\mathbb{I}_{\left\{ a_{\pi,i}=1\right\} }\sup K_{\pi,i}.
\end{equation}
If $a_{\pi,i}=0$ the trajectory $\varphi^{*}=s_{i}^{\dagger}\tau$
is the unique zero-cost trajectory ending in \textup{$s_{i}^{\dagger}$}.
If $a_{\pi,i}\neq0$ then a family of zero-cost trajectories $\varphi^{*}\in\mathcal{Q}_{s_{i}^{\dagger}}$
with \textup{$\tau^{-1}\varphi_{\tau}^{*}\in K_{\pi,i}\cup\partial K_{\pi,i}$,
$\tau\in\left[0,1\right]$} can exist such that $I_{\pi}\left[\varphi^{*}\right]=0$.
If ${\textstyle \lim_{\, a_{\pi,i}(s_{i}^{\dagger}-s)\rightarrow0^{+}}\left|F_{\pi}\left(s,\cdot\right)\right|}=\infty$
then $\varphi_{\tau}^{*}=s_{i}^{\dagger}\tau$ is the unique zero-cost
trajectory. If ${\textstyle \lim_{\, a_{\pi,i}(s_{i}^{\dagger}-s)\rightarrow0^{+}}\left|F_{\pi}\left(s,\cdot\right)\right|}<\infty$
we define 
\begin{equation}
\theta_{i}^{*}:=\exp\left(-{\textstyle \lim_{\, a_{\pi,i}(u-s_{i}^{*})\rightarrow0^{+}}\lim_{\, a_{\pi,i}(s_{i}^{\dagger}-s)\rightarrow0^{+}}\left|F_{\pi}\left(s,u\right)\right|}\right)
\end{equation}
and the function $F_{\pi,s_{i}^{\dagger}}^{-1}$ as in Corollary \ref{Corollary4},
with $s_{i}^{\dagger}$, $\theta_{i}^{*}$ on place of $s$, $\tau_{s,i}^{*}$.
Then $\varphi_{\tau}^{*}=\tau u_{\tau}^{*}$ with 
\begin{equation}
u_{\tau}^{*}:=s_{i}^{\dagger}\mathbb{I}_{\left\{ \tau\in\left(t,1\right]\right\} }+\, F_{\pi,s_{i}^{\dagger}}^{-1}\left(\log\left(t/\tau\right)\right)\,\mathbb{I}_{\left\{ \tau\in\left(\theta_{i}^{*}t,t\right]\right\} }+s_{i}^{*}\,\mathbb{I}_{\left\{ \tau\in\left[0,\theta_{i}^{*}t\right]\right\} },
\end{equation}
is a zero-cost trajectory for any $t\in\left[0,1\right]$. Concerning
trajectories $\varphi^{*}\in\mathcal{Q}_{s_{i}^{*}}$ with \textup{$\tau^{-1}\varphi_{\tau}^{*}\in K_{\pi,i}\cup\partial K_{\pi,i}$,
$\tau\in\left[0,1\right]$,} we have that $\varphi_{\tau}^{*}=s_{i}^{*}\tau$
is the unique zero-cost trajectory.
\end{cor}
As we can see, the set of zero-cost trajectories that end in a stable
equilibrium point can be degenerate. Again, this depends only on the
integrability of the singular behavior of $1/\left(\pi\left(s\right)-s\right)$
for $s\rightarrow s_{i}^{\dagger}$: if 
\begin{equation}
{\textstyle \lim_{\, a_{\pi,i}(s_{i}^{\dagger}-s)\rightarrow0^{+}}\left|F_{\pi}\left(s,\cdot\right)\right|}=\infty
\end{equation}
the trajectory is simply $\varphi^{*}=s_{i}^{\dagger}\tau$ and it
is unique. If instead 
\begin{equation}
{\textstyle \lim_{\, a_{\pi,i}(s_{i}^{\dagger}-s)\rightarrow0^{+}}\left|F_{\pi}\left(s,\cdot\right)\right|}<\infty
\end{equation}
then we have a family of time-inhomogeneous trajectories, parametrized
by the time $t$ at which they hit $s_{i}^{\dagger}$, that emanates
from the unstable equilibrium $s_{i}^{*}$ on the other side of $K_{\pi,i}$.
Moreover, if $s_{i}^{\dagger}$ is a downcrossing then $s_{i}^{\dagger}=\inf K_{\pi,i}=\sup K_{\pi,i-1}$
with $a_{\pi,i}=-1$, $a_{\pi,i-1}=1$, so that optimal trajectories
ending in $s_{i}^{\dagger}$ can emanate also from $\inf K_{\pi,i-1}$.
Notice that if $1/\left(\pi\left(s\right)-s\right)$ is integrable
also for $s\rightarrow s_{i}^{*}$ then the $\theta_{i}^{*}>0$ and
our optimal trajectories would be doubly time-inhomogeneous, emanating
from $s_{i}^{*}$ at some $\tau=\theta_{i}^{*}t$ and hitting $s_{i}^{\dagger}$
at $\tau=t$. More explicitly, integrability in the neighborhood of
an unstable equilibrium point (like an integrable upcrossing) make
it so unstable that the probability mass is expelled form its neighborhood
on a time scale $O\left(n\right)$, and makes it convenient to use
a time-inhomogeneous strategy. The inverse picture arises for integrable
stable points, for example an integrable downcrossings, where the
process is so attracted that it becomes entropically convenient to
hit the equilibrium point in a finite fraction $t\in\left[0,1\right)$
of the whole time span (of order $O\left(n\right)$), instead of approaching
it asymptotically (an example is in Figure \ref{fig:3}).

\begin{figure}
\begin{singlespace}
\centering{}\includegraphics[scale=2.63]{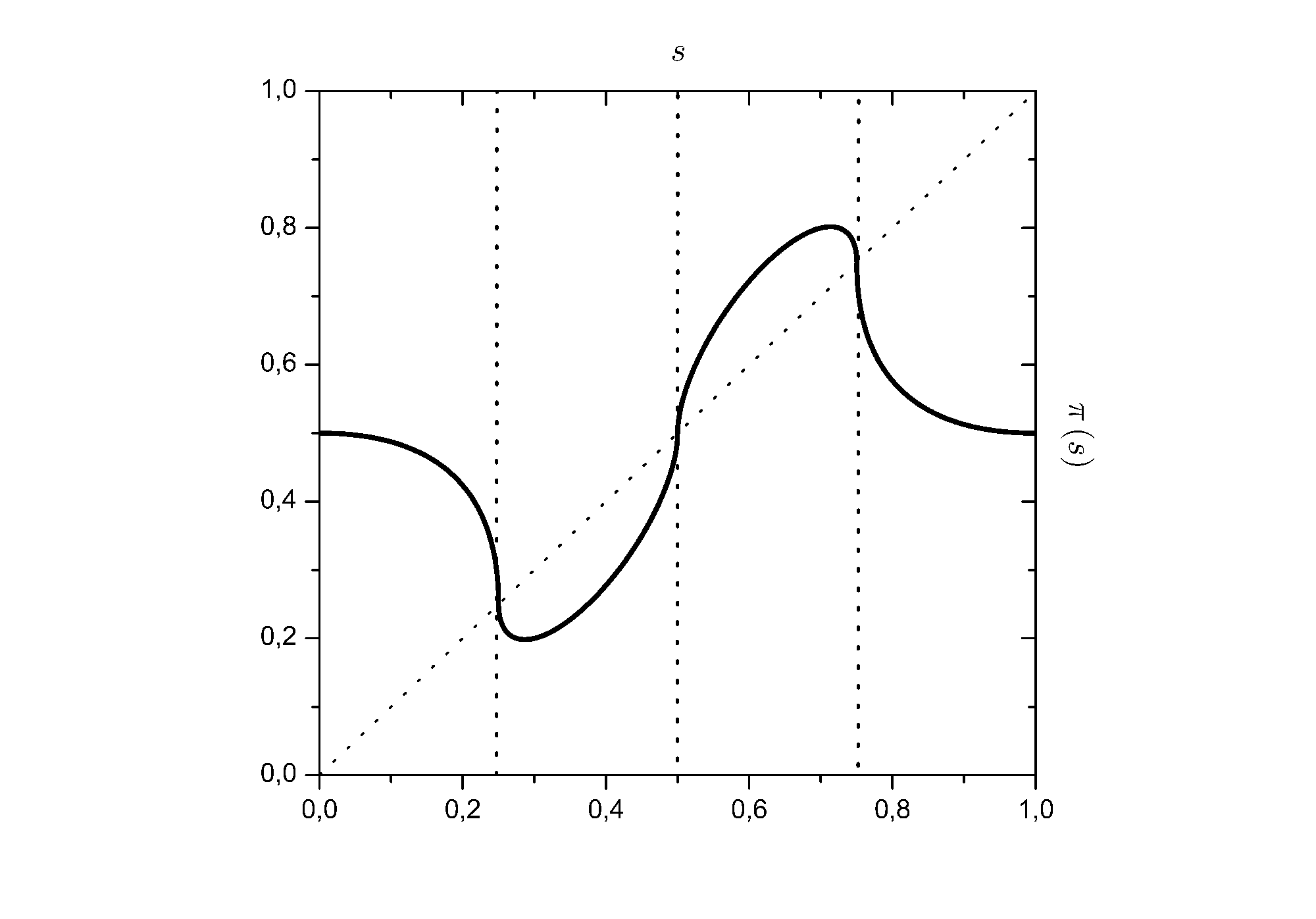}\\
~~~~~~~~~~\includegraphics[scale=2.63]{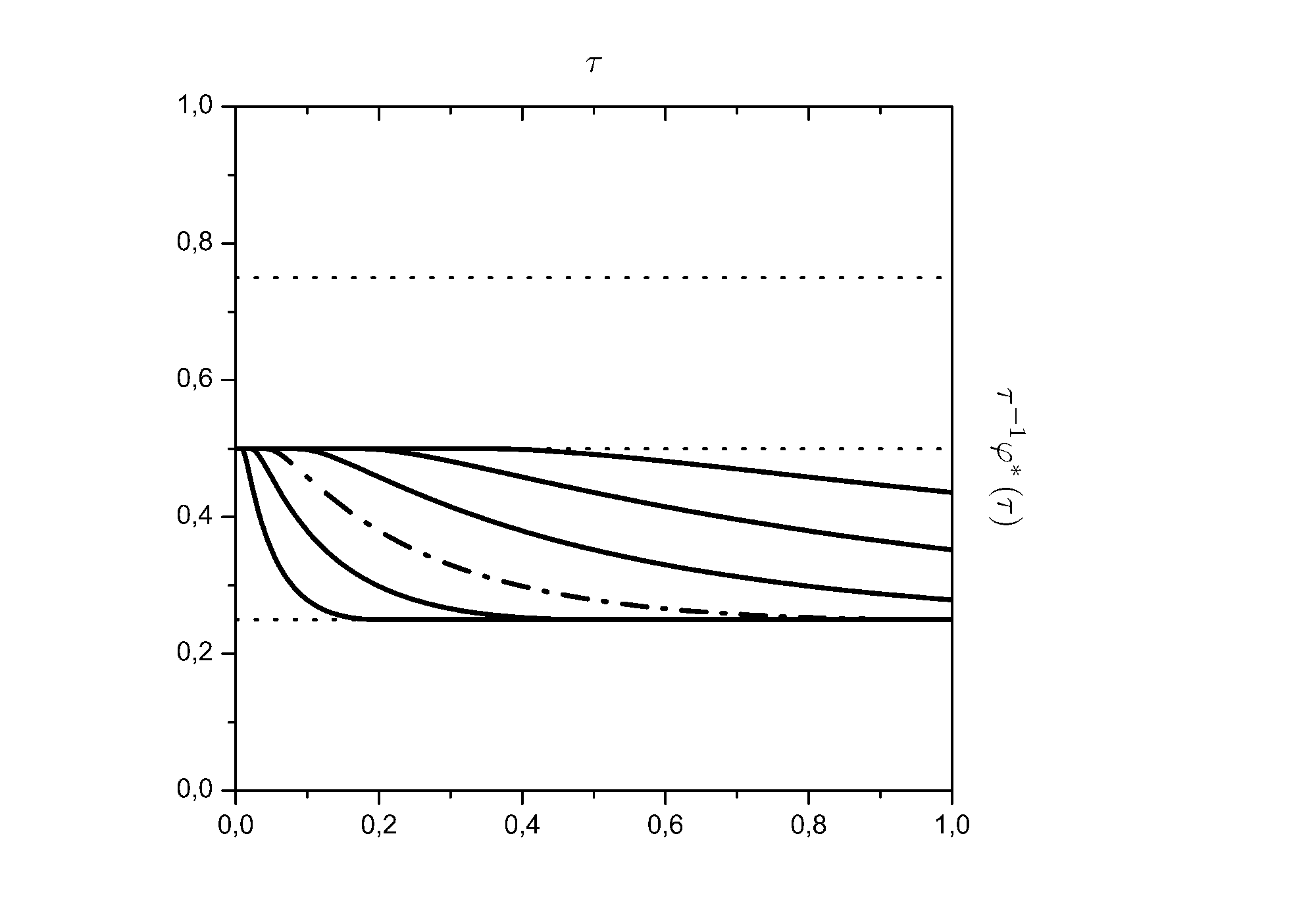}\caption{\label{fig:3}\textit{In the above figures we give an example to show
the effects of integrability on stable and unstable points. Consider
the urn function $\pi\left(s\right)=s+\mathbb{I}_{\left\{ s\in\left[01/4\right]\right\} }(1/4-s)^{1/2}-\mathbb{I}_{\left\{ s\in\left(1/4,1/2\right]\right\} }(1/4-s)^{1/2}(s-1/2){}^{1/2}+\,\mathbb{I}_{\left\{ s\in\left(1/2,3/4\right]\right\} }(1/2-s)^{1/2}(s-3/4)^{1/2}-\mathbb{I}_{\left\{ s\in\left(3/4,1\right]\right\} }(s-3/4)^{1/2}$
in the interval $s\in\left[1/4,1/2\right]$, then we have $F_{\pi}\left(s,u\right)=\left[2\arcsin\left(\sqrt{4z-1}\right)\right]_{u}^{s}$
and $\theta_{i}^{*}=\exp(-\pi)$. By Corollary \ref{Corollary4.1}
we find that the family of trajectories ending in $s_{i}^{\dagger}=1/2$
is ${\scriptstyle {\textstyle 4\tau^{-1}\varphi_{\tau}^{*}=\mathbb{I}_{\left\{ \tau\in\left[1,t\right]\right\} }+2\,\mathbb{I}_{\left\{ \tau\in\left[0,e^{-\pi}t\right]\right\} }}+{\textstyle \left[1+\sin^{2}\left(\frac{1}{2}\log\left(t/\tau\right)\right)\right]\mathbb{I}_{\left\{ \tau\in\left[e^{-\pi}t,t\right]\right\} }}}$
for $t\in\left[0,1\right]$, while for each $s\in\left(1/4,1/2\right]$
we have $\tau_{s,i}^{*}=\exp[2\arcsin\left(\sqrt{4s-1}\right)-\pi]$
and $4\tau^{-1}\varphi_{\tau}^{*}=[1+\sin^{2}(\frac{1}{2}\log(\tau_{s,i}^{*}/\tau))]\,\mathbb{I}_{\{\tau\in(\tau_{s,i}^{*},1]\}}+2\,\mathbb{I}_{\{\tau\in[0,\tau_{s,i}^{*}]\}},$
with $\lim_{s\rightarrow s_{i}^{\dagger}}\tau_{s,i}^{*}=\theta_{i}^{*}$
as expected. The urn function and some zero-cost trajectories in $K_{\pi,1}\cup\partial K_{\pi,1}=\left[1/4,\,1/2\right]$
are shown in lower figure, with $s=\frac{1}{4}\left[1+\sin^{2}\left(\frac{1}{2}\log\left(k\right)\right)\right]$,
$k\in\left\{ 2,\,4,\,8\right\} $ and with $t=\left\{ 1/8,\,1/2,\,1\right\} $.
The dash-dotted line is the critical trajectory with $t=1$.}}
\end{singlespace}
\end{figure}

It is an interesting result that no trajectory with $\lim_{\tau\rightarrow0}\left(\varphi_{\tau}/\tau\right)\notin\partial C_{\pi}$
can be optimal if $a_{\pi,i}\neq0$, not even if we chose $\varphi_{1}$
to be in a set of stable equilibrium like downcrossings (ie, $\varphi_{1}\in C_{\pi}\left(+,-\right)$).
We can interpret this result in terms of time spent in a given state:
it seems that a process starting with initial conditions $m^{-1}X_{n,m}\notin\partial C_{\pi}$
concentrates its mass in the neighborhood of the points of convergence
in times that are of order $o\left(n\right)$, and only those that
are in the neighborhood of unstable points can remain there for times
$O\left(n\right)$, eventually reaching the stable points according
to the mechanism suggested by Corollaries \ref{Corollary4} and \ref{Corollary4.1}.

\subsubsection{A comment on moderate deviations}

The above formulas for optimal trajectories are of particular interest,
since represent a first step to deal with the much richer problem
of \textit{moderate deviations}, ie, to compute limits of the kind
\begin{equation}
\phi_{\left\{ \sigma_{n}\right\} }\left(s_{1},s_{2}\right)=\lim_{n\rightarrow\infty}\sigma{}_{n}^{-1}\log\mathbb{P}\left(n^{-1}X_{n,n}\in\left(s_{1},s_{2}\right)\right)
\end{equation}
for some $\sigma_{n}=o\left(n\right)$, $s\in\left[\,\inf\, C_{\pi},\,\sup\, C_{\pi}\right]$.
To illustrate how this can be obtained we provide the following argument.
Let $u_{\tau,s}^{*}$ be an optimal trajectory ending in $s\in K_{\pi,i}$
(ie, $u_{1,s}^{*}=s$). Since any finite deviation from this trajectory
has an exponential cost on a time scale $O\left(n\right)$, the probability
mass current can move along these trajectories only. Moreover, Corollaries
\ref{Corollary4} and \ref{Corollary4.1} guarantee uniqueness of
the solutions and $u_{\tau,s_{1}}^{*}<u_{\tau,s}^{*}<u_{\tau,s_{2}}^{*}$
for any $\tau\in\left(0,1\right]$ and $s_{1}<s<s_{2}$. Hence, we
find that the probability current passing through $\left(u_{\tau,s_{1}}^{*},u_{\tau,s_{2}}^{*}\right)$
is constant for $\tau>0$,
\begin{equation}
\mathbb{P}\left(x_{n,\tau n}\in\left(u_{\tau,s_{1}}^{*},u_{\tau,s_{2}}^{*}\right)\right)=\mathbb{P}\left(x_{n,n}\in\left(s_{1},s_{2}\right)\right),\,\tau\in\left(0,1\right].
\end{equation}
Then, let $\tau\left(s_{i},\epsilon\right)$ such that $u_{\tau\left(s_{i},\epsilon\right)}^{*}=s_{i}+\epsilon$
and let consider the case $\tau\left(s_{1},\epsilon\right)>\tau\left(s_{2},\epsilon\right)$
for $s_{1}<s_{2}$. Corollaries \ref{Corollary4} and \ref{Corollary4.1}
also guarantee invertibility of the zero-cost trajectories, then we
can write 
\begin{equation}
\mathbb{P}\left(x_{n,n}\in\left(s_{1},s_{2}\right)\right)=\mathbb{P}\left(x_{n,\tau\left(s_{1},\epsilon\right)n}-s_{i}^{*}<\epsilon\right)-\mathbb{P}\left(x_{n,\tau\left(s_{2},\epsilon\right)n}-s_{i}^{*}<\epsilon\right).
\end{equation}
Given that $\mathbb{P}\left(x_{n,k}\in\left(\alpha,\beta\right)\right)=\mathbb{P}\left(x_{k,k}\in\left(\alpha,\beta\right)\right)$,
the problem of computing $\mathbb{P}\left(x_{n,n}\in\left(s_{1},s_{2}\right)\right)$
is reduced to that of computing $\mathbb{P}\left(x_{n,n}-s_{i}^{*}<\epsilon\right)$
for some arbitrary small $\epsilon>0$. For $\left(s_{1},s_{2}\right)\subseteq K_{\pi,i}$
a martingale analysis suggests the conjecture that $\phi_{\left\{ n^{\nu}\right\} }\left(s_{1},s_{2}\right)=0$
for any $\nu\in\left(0,1\right)$, and that $\phi_{\left\{ \log n\right\} }\left(s_{1},s_{2}\right)=1-\left[\partial_{s}\pi\left(s\right)\right]_{s=s_{i}^{*}}$.

\subsection{Cumulant Generating Function.}

\label{Section2.3}Except the fact that $\phi\left(s\right)<0$, for
$s\in\left[z_{-}^{*},\inf C_{\pi}\right)$ or $s\in\left(\sup\, C_{\pi},z_{+}^{*}\right]$
we couldn't extract more informations on the shape of $\phi\left(s\right)$
from its variational representation, because in these cases the variational
problem can't be simplified by Lemma \ref{lemma 12}, see Section
\ref{Section3.2.2}. Anyway, the existence of $\phi$ proved in Theorem
\ref{Theorem 2} introduces some critical simplifications that allows
to approach the problem using analysis, provided that $\pi$ obeys
to some additional regularity conditions. For example, we can prove
the convexity of $-\phi\left(s\right)$, $s\in\left[z_{-}^{*},\inf C_{\pi}\right)$,
or $s\in\left(\sup\, C_{\pi},z_{+}^{*}\right]$ in case $\pi$ is
invertible on the same intervals and the inverse functions
\begin{equation}
\pi_{-}^{-1}:\,\left[\pi\left(z_{-}^{*}\right),\pi(\inf C_{\pi})\right)\rightarrow\,\left[\, z_{-}^{*},\,\inf C_{\pi}\right),
\end{equation}
\begin{equation}
\pi_{+}^{-1}:\left(\pi(\sup\, C_{\pi}),\pi\left(z_{+}^{*}\right)\right]\rightarrow\left(\sup\, C_{\pi},z_{+}^{*}\right],
\end{equation}
are absolutely continuous Lipschitz functions. Such result can be
obtained by analyzing the scaling of the Cumulant Generating Function
(CGF) 
\begin{equation}
\psi\left(\lambda\right):=\lim_{n\rightarrow\infty}n^{-1}\log\mathbb{E}\left(e^{\lambda X_{n,n}}\right),\ \lambda\in\left(-\infty,\infty\right).\label{eq:1.18}
\end{equation}
First, notice that Theorem \ref{Theorem 2} implies that $\psi$ is
well defined \cite{Dembo Zeitouni}. Then, let $-\hat{\phi}\left(s\right)=\mathrm{conv\left(-\phi\left(s\right)\right)}$
be the convex envelope of $-\phi\left(s\right)$ for $s\in\left[0,1\right]$.
By Theorem \ref{Theorem 2} and Corollary \ref{Corollary 3} it follows
that $\hat{\phi}\left(s\right)=0$ when $s\in\left[\,\inf\, C_{\pi},\,\sup\, C_{\pi}\right]$
and $\hat{\phi}\left(s\right)<0$ otherwise. In addition, it holds
that
\begin{defn}
\textit{Let $\hat{\phi}_{-}:\left[z_{-}^{*},\inf C_{\pi}\right)\rightarrow\left(-\infty,0\right]$,
$\hat{\phi}_{+}:\left(\sup\, C_{\pi},z_{+}^{*}\right]\rightarrow\left(-\infty,0\right]$
such that $\hat{\phi}\left(s\right)=\hat{\phi}_{-}\left(s\right)$
when $s\in\left[z_{-}^{*},\inf C_{\pi}\right)$ and $\hat{\phi}\left(s\right)=\hat{\phi}_{+}\left(s\right)$
when $s\in\left(\sup\, C_{\pi},z_{+}^{*}\right]$. Also define $\psi_{-}:\left(-\infty,0\right]\rightarrow\left(-\infty,0\right]$,
$\psi_{+}:\left[0,\infty\right)\rightarrow\left[0,\infty\right)$
such that $\psi\left(\lambda\right)=\psi_{-}\left(\lambda\right)$
when $\lambda\in\left(-\infty,0\right]$ and $\psi\left(\lambda\right)=\psi_{+}\left(\lambda\right)$
when $\lambda\in\left[0,\infty\right)$. One can show that $-\hat{\phi}_{-}$
and $-\hat{\phi}_{+}$ are the Frenchel-Legendre transforms of $-\psi_{-}$
and $-\psi_{+}$ respectively: 
\begin{equation}
\hat{\phi}_{-}\left(s\right)=\inf_{\lambda\in\left(-\infty,0\right]}\left\{ \lambda s+\psi_{-}\left(\lambda\right)\right\} ,\:\hat{\phi}_{+}\left(s\right)=\inf_{\lambda\in\left[0,\infty\right)}\left\{ \lambda s+\psi_{+}\left(\lambda\right)\right\} .
\end{equation}
}
\end{defn}
Since the existence of $-\hat{\phi}$ implies the existence of $\psi$
for every $\pi\in\mathcal{U}$, while its convexity ensures that $\psi\in\mathcal{AC}$,
we have enough informations to approach $\psi$ by analytic methods.
Here we show that the Cumulant Generating Function satisfies the non-linear
implicit ODE, 
\begin{equation}
\pi\left(\partial_{\lambda}\psi\left(\lambda\right)\right)={\textstyle \frac{e^{\psi\left(\lambda\right)}-1}{e^{\lambda}-1}},\label{eq:1.KKK}
\end{equation}
for any $\lambda$. We stress that the CGF satisfies the above equation
for all $\pi\in\mathcal{U}$, but any information would be hard to
be extracted if $\pi$ is not invertible at least on $\left[z_{-}^{*},\inf C_{\pi}\right)$
and $\left(\sup\, C_{\pi},z_{+}^{*}\right]$. If this is the case,
then the following theorem provides the Cauchy problems for $\psi_{-}$
and $\psi_{+}$:
\begin{thm}
\label{Theorem 4}Let $\pi\in\mathcal{U}$ be invertible on $\left[z_{-}^{*},\inf C_{\pi}\right)$,
and denote by $\pi_{-}^{-1}:\left[\pi\left(z_{-}^{*}\right),\pi(\inf C_{\pi})\right)\rightarrow\left[z_{-}^{*},\inf C_{\pi}\right)$
its inverse. If $\pi_{-}^{-1}$ is $\mathcal{AC}$ and Lipschitz,
then for $\lambda\in\left(-\infty,0\right)$ we have $\psi\left(\lambda\right)=\psi_{-}\left(\lambda\right)$,
with $\psi_{-}\left(\lambda\right)$ solution to the Cauchy problem
\begin{equation}
\partial_{\lambda}\psi_{-}\left(\lambda\right)=\pi_{-}^{-1}\left({\textstyle \frac{e^{\psi_{-}\left(\lambda\right)}-1}{e^{\lambda}-1}}\right),\,\lim_{\lambda\rightarrow0^{-}}\,\partial_{\lambda}\psi_{-}\left(\lambda\right)=\pi_{+}\left(\inf\, C_{\pi}\right),\,\lim_{\lambda\rightarrow-\infty}\,\partial_{\lambda}\psi_{-}\left(\lambda\right)=z_{-}^{*},\label{eq:1.21}
\end{equation}
Let $\pi$ be invertible on $\left(\sup\, C_{\pi},z_{+}^{*}\right]$,
with $\pi_{+}^{-1}:\left(\pi(\sup\, C_{\pi}),\pi\left(z_{+}^{*}\right)\right]\rightarrow\left(\sup\, C_{\pi},z_{+}^{*}\right]$
its inverse function. If $\pi_{+}^{-1}$ is $\mathcal{AC}$ and Lipschitz,
then for $\lambda\in\left(0,\infty\right)$ we have $\psi\left(\lambda\right)=\psi_{+}\left(\lambda\right)$,
with $\psi_{+}\left(\lambda\right)$ solution to the Cauchy problem
\begin{equation}
\partial_{\lambda}\psi_{+}\left(\lambda\right)=\pi_{+}^{-1}\left({\textstyle \frac{e^{\psi_{+}\left(\lambda\right)}-1}{e^{\lambda}-1}}\right),\,\lim_{\lambda\rightarrow0^{+}}\,\partial_{\lambda}\psi_{+}\left(\lambda\right)=\pi_{+}\left(\sup\, C_{\pi}\right),\,\lim_{\lambda\rightarrow\infty}\,\partial_{\lambda}\psi_{+}\left(\lambda\right)=z_{+}^{*}.\label{eq:1.22}
\end{equation}
A unique global solution exists for both Cauchy problems (\ref{eq:1.21}),
(\ref{eq:1.22}), it is $\mathcal{AC}$ and has continuous first derivative.
\end{thm}
Although the above result is obtained for urn functions belonging
to a subset of $\mathcal{U}$ we consider it of special importance
from the applicative side as it allows to explicitly compute $\phi$
(at least numerically) in those intervals of $s$ where $\phi$ is
nontrivial, thus providing a substantial improvement of Corollary
\ref{Corollary 3}. 

Another trivial but potentially useful application is the inverse
problem of deciding weather a given function $f$ can be the rate
function of some urn process. Since $-\hat{\phi}\left(s\right)$ is
convex by definition, then $\psi\left(\lambda\right)=\hat{\phi}\left(\partial_{\lambda}\psi\right)-\lambda\partial_{\lambda}\psi$,
from which follows that $\lambda\left(s\right)=-\partial_{s}\hat{\phi}\left(s\right)$
and $\psi\left(\lambda\left(s\right)\right)=\hat{\phi}\left(s\right)-s\partial_{s}\hat{\phi}\left(s\right)$.
If $-\phi$ is convex, then obviously $\phi=\hat{\phi}$ and we can
state the following corollary:
\begin{cor}
Let $f:\left[0,1\right]\rightarrow\left(-\infty,0\right]$ be a bounded
and concave $\mathcal{AC}$ function, and define the function $\pi_{f}$
as follows: 
\begin{equation}
\pi_{f}\left(s\right)={\textstyle {\displaystyle \frac{e^{f\left(s\right)-s\,\partial_{s}f\left(s\right)}-1}{e^{-\partial_{s}f\left(s\right)}-1}}},\, s\in\left[0,1\right].\label{eq:1.KKK-1}
\end{equation}
If the function $f$ is such that $\pi_{f}\in\mathcal{U}$ and $f\left(0\right)=\log(1-\pi_{f}\left(0\right))$,
$f\left(1\right)=\log(\pi_{f}\left(1\right))$ then the limit $\phi$
defined in Eq. (\ref{eq:1.15}) for an urn process with urn function
$\pi_{f}$ is $\phi=f$.
\end{cor}
We believe that such result could find useful applications in those
stochastic approximation algorithms for which the process is required
to satisfy some given LDP. Notice that these results quite immediately
imply the convexity of $-\phi$ since if the cumulants $-\psi_{-}$
and $-\psi_{+}$ have continuous first derivatives their Frenchel-Legendre
transforms $-\hat{\phi}_{-}$, $-\hat{\phi}_{+}$ must be strictly
convex, with $\hat{\phi}_{-}=\phi_{-}$ and $\hat{\phi}_{+}=\phi_{+}$.
\begin{cor}
\label{Corollary 5}Let $\pi\in\mathcal{U}$ invertible on $\left[z_{-}^{*},\inf C_{\pi}\right)$,
and denote by $\pi_{-}^{-1}:\left[\pi\left(z_{-}^{*}\right),\pi(\inf C_{\pi})\right)\rightarrow\left[z_{-}^{*},\inf C_{\pi}\right)$
its inverse function. If $\pi_{-}^{-1}$ is $\mathcal{AC}$ and Lipschitz,
then $\phi_{-}$ is in $\mathcal{AC}$, is strictly concave on $\left[z_{-}^{*},\inf C_{\pi}\right)$,
and strictly increasing from $\log\bar{\pi}\left(z_{-}^{*}\right)$
to $0$. Let $\pi$ be invertible on $\left(\sup\, C_{\pi},z_{+}^{*}\right]$,
with inverse function $\pi_{+}^{-1}:\left(\pi(\sup\, C_{\pi}),\pi\left(z_{+}^{*}\right)\right]\rightarrow\left(\sup\, C_{\pi},z_{+}^{*}\right]$.
If $\pi_{+}^{-1}$ is $\mathcal{AC}$ and Lipschitz, then $\phi_{+}$
is in $\mathcal{AC}$, it is strictly concave on $\left(\sup\, C_{\pi},z_{+}^{*}\right]$,
and strictly decreasing from $0$ to $\log\pi\left(z_{+}^{*}\right)$. 
\end{cor}

\subsubsection*{Linear urns and the Bachi-Pal Model}

The last topic we present is the application to the Baghi-Pal model,
a widely investigated model due to its relevance in studying branching
phenomena and random trees (see \cite{Pemantle,Mahmoud,MahmoudBook,Johnson_Koz,Kotz Mahmoud}
for some reviews). Consider an urn with black and white balls: at
each step a ball is extracted uniformly from the urn and some new
balls are added or discarded according to the square matrix 
\begin{equation}
A:=\left(\begin{array}{cc}
a_{11} & a_{12}\\
a_{21} & a_{22}
\end{array}\right),
\end{equation}
with $a_{ij}\in\mathbb{Z}$, such that if the extraction resulted
in a black ball we add $a_{11}$ black balls and $a_{12}$ white balls,
otherwise we add $a_{21}$ black balls and $a_{22}$ white balls.
If $a_{11}+a_{12}=a_{21}+a_{22}=M$, then the number of balls increases
(ore decreases) by some deterministic rate $M$ and the urn is said
to be \textit{balanced}, if $M>0$ the urn is said to be also \textit{tenable}. 

Beside the many applicative aspects, our interest araises from the
fact that this is the first nontrivial model for which some large
deviations results have been obtained. In \cite{Fajolet-Analytic Urns,Fajolet2}
the so-called \textit{subtractive} case (negative diagonal entries)
is fully analyzed by purely analytic methods, obtaining an explicit
characterization of the rate function and other important results.
Another LDP study on linear urns involving more probabilistic techniques
has been provided by Bryc et Al. \cite{Bryc}. In this paper they
consider a process with urn function $\pi\left(s\right)=1-s/\alpha$,
$\alpha\in\left(1,\infty\right)$, giving an expression for the Cumulant
Generating Function and other related results. 

\begin{onehalfspace}
\begin{figure}
\begin{singlespace}
\centering{}\includegraphics[scale=2.5]{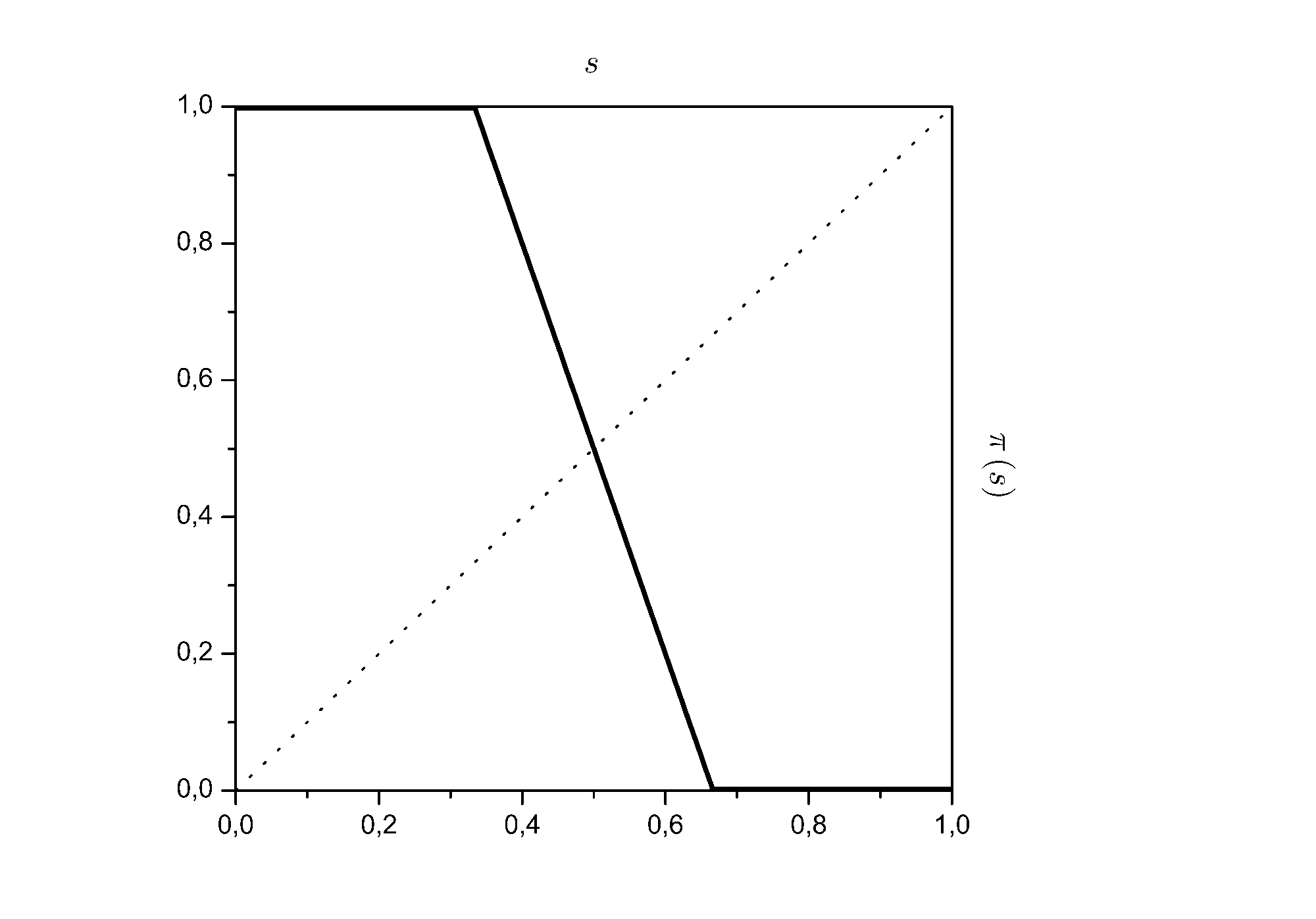}\\
\includegraphics[scale=2.5]{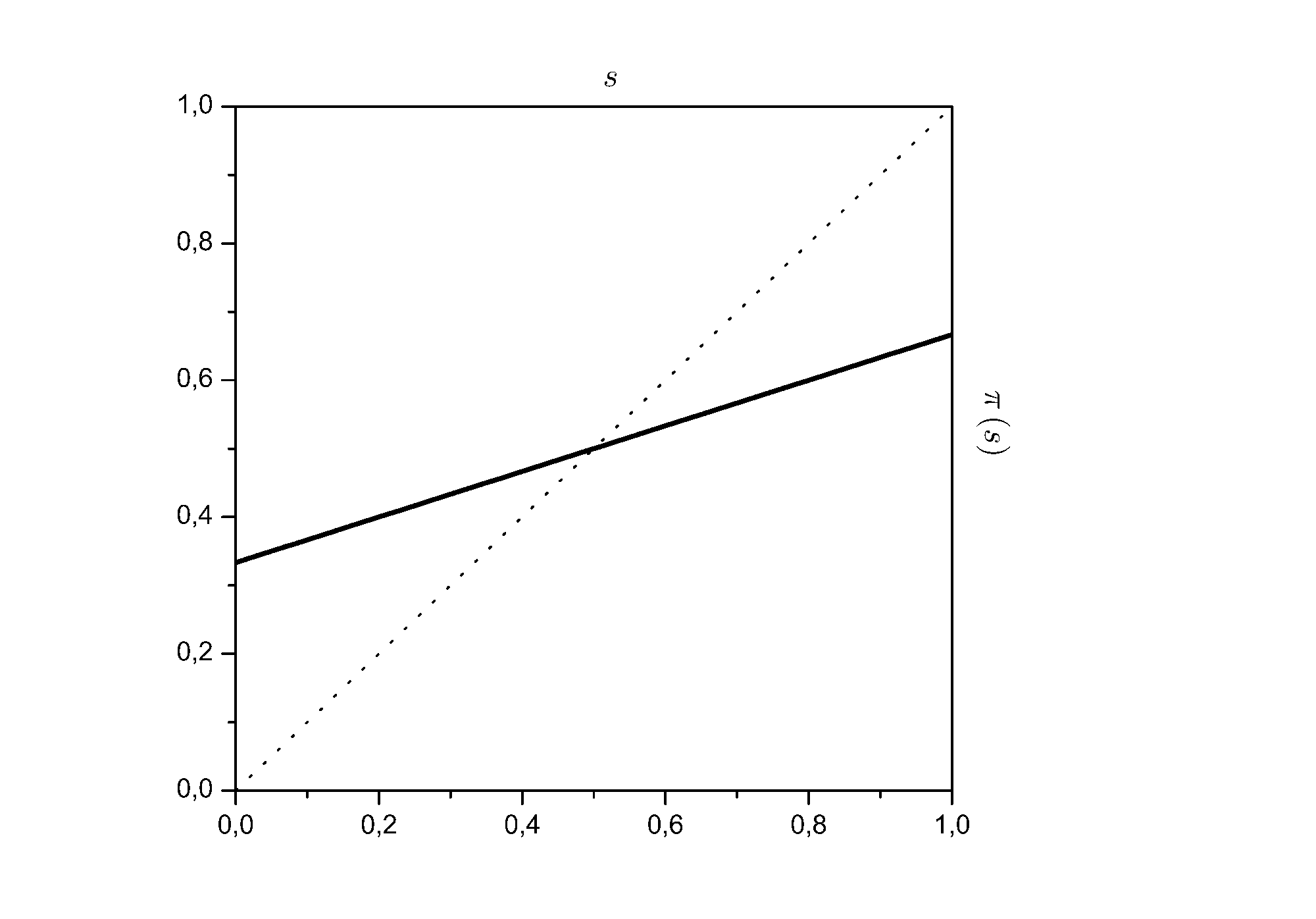}\caption{{\small{}\label{fig:4}}\textit{\small{}Urn functions from Eq.s (\ref{eq:2.12})
and (\ref{eq:2.12.2}) of Bagchi-Pal models for $a_{11}=a_{22}=-1$,
$a_{21}=a_{12}=2$ (upper figure) and $a_{11}=a_{22}=2$, $a_{21}=a_{12}=1$
(lower figure). The first one is a subtractive urn of the kind considered
in \cite{Fajolet-Analytic Urns}, while the second is an additive
and tenable urn.}}
\end{singlespace}
\end{figure}

\end{onehalfspace}

Let show that the above model is equivalent to a linear urn function
$\pi\left(s\right)=s_{0}+b\left(s-s_{0}\right)$ provided that $A$
fulfills some self-consistency conditions. Let $B_{k}$ and $W_{k}$
be the number of black and withe balls of a Bagchi-Pal urn at time
$k$, let $T_{k}=B_{k}+W_{k}$ be the total number of balls and 
\begin{equation}
A=\left(\begin{array}{cc}
a_{11} & M-a_{11}\\
M-a_{22} & a_{22}
\end{array}\right)
\end{equation}
the reinforcement matrix, where we used the balancing constraint $a_{11}+a_{12}=a_{21}+a_{22}=M$.
Since the balancing ensures that $T_{k}=\left(B_{0}+W_{0}\right)+Mk$,
we can rescale $k\rightarrow k-M^{-1}\left(B_{0}+W_{0}\right)$ and
consider $k\geq m=M^{-1}\left(B_{0}+W_{0}\right)$, $T_{k}=Mk$. Then,
define the variable 
\begin{equation}
X_{n,k}=\frac{B_{k}-\left(M-a_{22}\right)k}{a_{11}+a_{22}-M},
\end{equation}
with $a_{11}+a_{22}-M\ne0$: we can show that the process $\left\{ X_{n,k}:m\leq k\leq n\right\} $
defined by the urn function $\pi\left(s\right)=s_{0}+b\left(s-s_{0}\right)$,
with
\begin{equation}
s_{0}=\frac{a_{22}-M}{2M-a_{11}-a_{22}},\, b=\frac{a_{11}+a_{22}}{M}-1,\, X_{n,m}=\frac{B_{0}-\left(M-a_{22}\right)m}{a_{11}+a_{22}-M}.
\end{equation}
is equivalent to a Bagchi-Pal model with reinforcement matrix
\begin{equation}
A=M\left(\begin{array}{cc}
b+s_{0}\left(1-b\right) & \left(1-s_{0}\right)\left(1-b\right)\\
s_{0}\left(1-b\right) & 1-s_{0}\left(1-b\right)
\end{array}\right).\label{eq:2.12}
\end{equation}
Since the Bagchi-Pal model usually considers an integer reinforcement
matrix, we need $M$, $s_{0}$, $b$, $m$ such that both $B_{0}+W_{0}$
and the elements of $A$ are integers. If $a_{12}=a_{21}=0$ we recover
the Polya Urn ($b=1$), while we obviously have to discard the case
$a_{11}=a_{21}$ (deterministic evolution of the urn: $a_{11}+a_{22}-M=0$).
Usually some \textit{tenability} conditions are assumed which ensures
that the process can't be stopped, ie, that the total number of balls
is deterministic and always growing ($M>0$), that $a_{12}\geq0$,
$a_{21}\geq0$ and if $a_{11}<0$ then $\left(W_{0}/a_{11}\right),\left(a_{21}/a_{11}\right)\in\mathbb{Z}$,
if $a_{22}<0$ then $\left(B_{0}/a_{22}\right),\left(a_{12}/a_{22}\right)\in\mathbb{Z}$.
The last two conditions ensure that only balls of the same color of
that drawn can be removed from the urn: this prevents from stopping
the process by impossible removals. 

\label{Section2.3.1} According to the above discussion, and considering
that $B_{0}/m\in\left[0,1\right]$, it is possible to show that the
general urn function describing the \textit{balanced} Baghi-Pal urns
is 
\begin{equation}
\pi\left(s\right)=\mathbb{I}_{\left\{ s_{0}+b\left(s-s_{0}\right)\geq1\right\} }+\left(s_{0}+b\left(s-s_{0}\right)\right)\mathbb{I}_{\left\{ 0<s_{0}+b\left(s-s_{0}\right)<1\right\} }.\label{eq:2.12.2}
\end{equation}
As example, the subtractive urn $a_{11}=a_{22}=-1$, $a_{12}=a_{21}=2$
is described by the urn function (see Figure \ref{fig:4})
\begin{equation}
\pi\left(s\right)=\mathbb{I}_{\left\{ s\in\left[0,1/3\right)\right\} }+\left(2-3s\right)\mathbb{I}_{\left\{ s\in\left[1/3,2/3\right]\right\} }.
\end{equation}
In the following we provide a complete characterization of the CGF
for the case of linear urn, which also includes all cases of the balanced
Bagchi-Pal models. We only consider linear urn functions with $a>0$
and $a+b<1$ to exclude the ``trivial'' cases with $\pi\left(0\right)=0$
and $\pi\left(1\right)=1$, for which by Corollary \ref{Corollary 3}
we would find $\phi\left(s\right)=0$ for any $s\in\left[0,1\right]$,
and for which we can even compute the optimal trajectories by Corollaries
\ref{Corollary4}, \ref{Corollary4.1}.
\begin{cor}
\label{Corollary8}Let $\pi$ be as in Eq. (\ref{eq:2.12.2}) with
$a>0$ and $a+b<1$, $\psi$ as in Eq. (\ref{eq:1.18}) and define
the function
\begin{equation}
B\left(\alpha,\beta;x_{1},x_{2}\right)=\int_{x_{1}}^{x_{2}}dt\,\left(1-t\right)^{\alpha-1}t^{\beta-1}.
\end{equation}
Then, for $\lambda>0$ we have $\psi=\psi_{+}$, with 
\begin{equation}
\psi_{+}\left(\lambda\right)=\psi_{+}\left(\lambda;b<0\right)\mathbb{I}_{\left\{ b<0\right\} }+\psi_{+}\left(\lambda;b>0\right)\mathbb{I}_{\left\{ b>0\right\} },
\end{equation}
where \textup{$\psi_{+}\left(\lambda;b>0\right)$, $\psi_{+}\left(\lambda;b<0\right)$}
are given by the expressions 
\begin{equation}
e^{-\psi_{+}\left(\lambda;b>0\right)}=1-{\textstyle \frac{a}{b}}e^{\frac{a}{b}\lambda}\left({\textstyle 1-e^{-\lambda}}\right)^{\frac{1}{b}}B\left({\textstyle {\textstyle \frac{a}{b}},\frac{b-1}{b};1-e^{-\lambda},1}\right),
\end{equation}
\begin{equation}
e^{-\psi_{+}\left(\lambda;b<0\right)}=1+{\textstyle \frac{a}{b}}e^{\frac{a}{b}\lambda}\left({\textstyle 1-e^{-\lambda}}\right)^{\frac{1}{b}}B\left({\textstyle {\textstyle {\textstyle \frac{a}{b}},\frac{b-1}{b};0,\,}1-e^{-\lambda}}\right).
\end{equation}
If $\lambda<0$ we have instead $\psi=\psi_{-}$, with
\begin{equation}
\psi_{-}\left(\lambda\right)=\psi_{-}\left(\lambda;b<0\right)\mathbb{I}_{\left\{ b<0\right\} }+\psi_{-}\left(\lambda;b>0\right)\mathbb{I}_{\left\{ b>0\right\} },
\end{equation}
where \textup{$\psi_{-}\left(\lambda;b>0\right)$, $\psi_{-}\left(\lambda;b<0\right)$
are given by}
\begin{equation}
e^{-\psi_{-}\left(\lambda;b>0\right)}=1+{\textstyle \frac{a}{b}}e^{-\frac{1-a+b}{b}\lambda}\left({\textstyle 1-e^{\lambda}}\right)^{\frac{1}{b}}B\left({\textstyle \frac{1-a}{b},\frac{b-1}{b};1-e^{\lambda},1}\right),
\end{equation}
\begin{equation}
e^{-\psi_{-}\left(\lambda;b<0\right)}=1-{\textstyle \frac{a}{b}}e^{-\frac{1-a+b}{b}\lambda}\left({\textstyle 1-e^{\lambda}}\right)^{\frac{1}{b}}B\left({\textstyle \frac{1-a}{b},\frac{b-1}{b};0,1-e^{\lambda}}\right).
\end{equation}

\end{cor}
An intriguing property of the above solution is that if $b>0$ then
$\psi$ is non-analytic at $\lambda\rightarrow0^{-}$($\lambda\rightarrow0^{+}$).
We can see this, for example, from the expression of $\psi_{-}\left(\lambda;b>0\right)$:
expanding for small $\lambda$ we find a non vanishing therm $O\left(\lambda^{1/b}\log\left(\lambda\right)\right)$
if $1/b\in\mathbb{N}$ and $O\left(\lambda^{1/b}\right)$ if $1/b\notin\mathbb{N}$,
which implies that the derivatives of order $\left\lceil 1/b\right\rceil $
and higher are singular in $\lambda=0$. The singularity disappears
for $b<0$. 

This behavior is not observed in case of \textit{subtractive} urns
for which the rate function is always analytic in $\lambda=0$, as
first noticed in \cite{Fajolet-Analytic Urns}. This is not surprising
since these urns are affine to the case $b<0$ for which we also observe
a regular solution. Notice that a non-analytic point in $\lambda=0$
implies divergent cumulants from $\left\lceil 1/b\right\rceil $ order
onwards. Moreover, if $b>1/2$ the shape of $\phi\left(s\right)$
around its peak is not even Gaussian anymore, since we find a divergent
second cumulant $\partial_{\lambda}^{\,2}\,\psi\left(\lambda\right)=O\left(\lambda^{-\gamma}\right)$
with $\gamma=2-1/b>0$. If $b=1/2$ we a logarithmic divergence of
$\partial_{\lambda}^{\,2}\,\psi\left(\lambda\right)$ is observed
as expected from the moment analysis of the Bagchi-Pal model (see
\cite{MahmoudBook} for a review).

\section{Proofs.}

\label{Section3}In this section we collected most of the proofs and
technical features of the present work. The proofs are presented in
the order they appeared in the previous section. We will first deal
with the Sample-Path Large Deviation Principle, then the entropy of
the event $\left\{ X_{n}=\left\lfloor sn\right]\right\} $ and, finally,
with the Cumulant generating function. We assume that all random variables
and processes are defined in a common probability space $\left(\Omega,\mathcal{F},\mathbb{P}\right)$.

\subsection{Sample-Path Large Deviation Principle.}

\label{Section3.1}Here we prove the existence of Sample-Path LDPs
for $\chi_{n}$ using some standard Large Deviation tools, such as
Mogulskii Theorem and the Varadhan Integral Lemma. 

Before we get into the core of this, we recall that $\left\Vert \varphi\right\Vert :=\sup_{\tau\in\left[0,1\right]}\left|\varphi_{\tau}\right|$
is the usual supremum norm, and we consider the metric space $\left(\mathcal{Q},\,\left\Vert \cdot\right\Vert \right)$,
with $\mathcal{Q}$ defined in Eq. (\ref{eq:Q-def}). Note that $\mathcal{Q}$
is compact with respect to the supremum norm topology. Moreover, since
by definition $\left\Vert \varphi\right\Vert \leq1$ for any $\varphi\in\mathcal{Q}$
we trivially find that $\mathcal{Q}\subset L_{\infty}\left(\left[0,1\right]\right)$.

\subsubsection{Change of measure.}

\label{Section3.1.1}We need a variational representation for the
rate function of $\chi_{n}$ in terms of sample paths. Let $\varphi:=\left\{ \varphi_{\tau}:\tau\in\left[0,1\right]\right\} $,
and define 
\begin{equation}
\mathcal{Q}_{n}:=\{\varphi:\,\varphi_{\tau}=\frac{1}{n}\sum_{1\leq i\leq\left\lfloor n\tau\right\rfloor }\theta_{i}+(\tau-n^{-1}\left\lfloor n\tau\right\rfloor )\,\mathbf{\theta}_{\left\lfloor n\tau\right\rfloor }\,,\,\theta_{i}\in\left\{ 0,1\right\} \}.\label{eq:2.9-1}
\end{equation}
The above set is the support of $\chi_{n}$ for $n<\infty$: note
that $\mathcal{Q}_{n}\subset\mathcal{Q}$ for all $n$. We also introduce
the following notation: 
\begin{equation}
Y_{n,k}\left(\varphi\right):=n\varphi_{k/n},\,\delta Y_{n,k}\left(\varphi\right):=n\left(\varphi_{\left(k+1\right)/n}-\varphi_{k/n}\right),\label{eq:2.10}
\end{equation}
Then, let $\varphi\in\mathcal{\mathcal{Q}}_{n}$: by Eq. (\ref{eq:1.2})
we can write the sample-path probability $\mathbb{P}\left(\chi_{n}=\varphi\right)$
in terms of $\varphi$ as follows:
\begin{equation}
\mathbb{P}\left(\chi_{n}=\varphi\right)=\prod_{1\leq k\leq n-1}\pi\left(Y_{n,k}\left(\varphi\right)/k\right)^{\delta Y_{n,k}\left(\varphi\right)}\bar{\pi}\left(Y_{n,k}\left(\varphi\right)/k\right)^{1-\delta Y_{n,k}\left(\varphi\right)}.\label{eq:2.1-1}
\end{equation}
Our first step is to prove Theorem \ref{Theorem 1} under the additional
assumption that $\pi\left(s\right)\in\left(0,1\right)$ for all $s\in\left[0,1\right]$.
In this case the proof can be obtained by straight applications of
the Mogulskii Theorem, the Varadhan Integral Lemma and the following
two lemmas.

Let $S_{\pi}:\mathcal{Q}\rightarrow\left(-\infty,0\right]$ be as
in Eq. (\ref{eq:1.10}). The first lemma shows the continuity of $S_{\pi}$
with respect to the supremum norm for any compact subset of $\mathcal{\mathcal{Q}}$
and any $\pi\in\mathcal{U}$, $\pi\in\left(0,1\right)$. The second
gives an approximation argument to the functional $S_{\pi}$ for the
entropy of the event $\left\{ \chi_{n}=\varphi\right\} $ when $\varphi\in\mathcal{\mathcal{Q}}_{n}$.
\begin{lem}
\label{Lemma 5-1}Assume $\pi\in\mathcal{U}$ and $\pi\left(s\right)\in\left(0,1\right)$
for all $s\in\left[0,1\right]$. The functional $S_{\pi}:\mathcal{Q}\rightarrow\left(-\infty,0\right]$
is continuous on the metric space $\left(\mathcal{\mathcal{Q}},\,\left\Vert \cdot\right\Vert \right)$.
Moreover, a function $W_{\pi}:\left[0,1\right]\rightarrow\left[0,\infty\right)$
exists such that $\lim{}_{s\rightarrow0}W_{\pi}\left(s\right)=0$
and $\left|S_{\pi}\left[\varphi\right]-S_{\pi}\left[\eta\right]\right|\leq W_{\pi}\left(\left\Vert \varphi-\eta\right\Vert \right)$,
$\forall\varphi,\eta\in\mathcal{Q}$.\end{lem}
\begin{proof}
Take any $\varphi,\eta\in\mathcal{Q}$. By definition of $S_{\pi}$,
we can rearrange the terms as follows
\begin{multline}
S_{\pi}\left[\varphi\right]-S_{\pi}\left[\eta\right]=\int_{\tau\in\left[0,1\right]}d\varphi_{\tau}\,\log\pi\left(\varphi_{\tau}/\tau\right)-\int_{\tau\in\left[0,1\right]}d\eta_{\tau}\,\log\pi\left(\eta_{\tau}/\tau\right)+\\
+\int_{\tau\in\left[0,1\right]}d\tilde{\varphi}_{\tau}\,\log\bar{\pi}\left(\varphi_{\tau}/\tau\right)-\int_{\tau\in\left[0,1\right]}d\tilde{\eta}_{\tau}\,\log\bar{\pi}\left(\eta_{\tau}/\tau\right),\label{eq:2.0}
\end{multline}
where we used the notation $\tilde{\varphi}=\tau-\varphi$, $\tilde{\eta}=\tau-\eta$.
Let us first consider $\log\pi\left(s\right)$: by definition of the
set $\mathcal{U}$ and the assumption that $\pi\in\left(0,1\right)$
we have that $\left\Vert \log\pi\right\Vert <\infty$, and that \textit{$\left|\log\pi\left(x+\delta\right)-\log\pi\left(x\right)\right|\leq f\left(\left|\delta\right|\right)$
}and \textit{$\lim_{\epsilon\rightarrow0}\epsilon\int_{\epsilon}^{1}dz\, f\left(z\right)/z^{2}=0$}.
Then we can write
\begin{multline}
\int_{\tau\in\left[0,1\right]}d\varphi_{\tau}\,\log\pi\left(\varphi_{\tau}/\tau\right)-\int_{\tau\in\left[0,1\right]}d\eta_{\tau}\,\log\pi\left(\eta_{\tau}/\tau\right)=\\
=\int_{\tau\in\left[0,1\right]}d\varphi_{\tau}\,\left[\log\pi\left(\varphi_{\tau}/\tau\right)-\log\pi\left(\eta_{\tau}/\tau\right)\right]+\int_{\tau\in\left[0,1\right]}d\left(\varphi_{\tau}-\eta_{\tau}\right)\,\log\pi\left(\varphi_{\tau}/\tau\right).\label{eq:2.1}
\end{multline}
By the uniform continuity condition one has $\left|\log\pi\left(\varphi_{\tau}/\tau\right)-\log\pi\left(\eta_{\tau}/\tau\right)\right|\leq f\left(\left|\varphi_{\tau}-\eta_{\tau}\right|/\tau\right)$.
Moreover, since $\varphi_{\tau}\leq\tau$ and $\eta_{\tau}\leq\tau$,
we have
\begin{equation}
\left|\varphi_{\tau}-\eta_{\tau}\right|\leq\min\left\{ \tau,\,\left\Vert \varphi-\eta\right\Vert \right\} ,
\end{equation}
and $d\varphi_{\tau}\leq d\tau$. Then, if we define $s^{-1}H_{f}\left(s\right):=\int_{s}^{1}dz\, f\left(z\right)/z^{2}$
the first integral can be bounded as follows
\begin{equation}
\int_{\tau\in\left[0,1\right]}d\varphi_{\tau}\,\left|\log\pi\left(\varphi_{\tau}/\tau\right)-\log\pi\left(\eta_{\tau}/\tau\right)\right|\leq\left\Vert \bar{\pi}\right\Vert ^{-1}H_{f}\left(\left\Vert \varphi-\eta\right\Vert \right),
\end{equation}
while for the second we get
\begin{equation}
\int_{\tau\in\left[0,1\right]}d\left(\varphi_{\tau}-\eta_{\tau}\right)\,\left|\log\pi\left(\varphi_{\tau}/\tau\right)\right|\leq\left\Vert \log\pi\right\Vert \left\Vert \varphi-\eta\right\Vert .
\end{equation}
Since by definition $H_{f}\left(s\right)$ is positive for $s\in\left(0,1\right]$,
and $\lim_{s\rightarrow0}H_{f}\left(s\right)=0$, we can take the
limit $\left\Vert \varphi-\eta\right\Vert \rightarrow0$. Repeating
the same steps for the second part, with $\log\bar{\pi}$ on place
of of $\log\pi$ and $\tilde{\varphi}$, $\tilde{\eta}$ on place
of of $\varphi$, $\eta$ will complete the proof.\end{proof}
\begin{lem}
\label{Lemma 5}Assume $\pi\in\mathcal{U}$ and $\pi\left(s\right)\in\left(0,1\right)$
for all $s\in\left[0,1\right]$, take some $\varphi\in\mathcal{Q}_{n}$,
and let $S_{\pi}:\mathcal{Q}\rightarrow\left(-\infty,0\right]$ as
in Eq. (\ref{eq:1.10}): then, $n^{-1}\log\mathbb{P}\left(\chi_{n}=\varphi\right)=S_{\pi}\left[\varphi\right]+O\left(W_{\pi}\left(1/n\right)\right)$,
with $W_{\pi}$ as in Lemma \ref{Lemma 5-1}.\end{lem}
\begin{proof}
Let $\varphi\in\mathcal{Q}_{n}$. To estimate the difference between
$n^{-1}\log\mathbb{P}\left(\chi_{n}=\varphi\right)$ and $S_{\pi}\left[\varphi\right]$
we can proceed as follows. First, we define 
\begin{equation}
\epsilon_{n}:=\left\{ \epsilon_{n,\tau}=\left(n\tau/\left\lfloor n\tau\right\rfloor \right)\varphi_{\left\lfloor n\tau\right\rfloor /n}-\varphi_{\tau}:\,\tau\in\left[0,1\right]\right\} ,
\end{equation}
such that the difference between $n^{-1}\log\mathbb{P}\left(\chi_{n}=\varphi\right)$
and $S_{\pi}\left[\varphi\right]$ can be written as follows
\begin{multline}
n^{-1}\log\mathbb{P}\left(\chi_{n}=\varphi\right)-S_{\pi}\left[\varphi\right]=\int_{\tau\in\left[0,1\right]}d\varphi_{\tau}\,\left[\log\pi\left(\left(\varphi_{\tau}+\epsilon_{n,\tau}\right)/\tau\right)-\log\pi\left(\varphi_{\tau}/\tau\right)\right]+\\
+\int_{\tau\in\left[0,1\right]}d\tilde{\varphi}_{\tau}\,\left[\log\bar{\pi}\left(\left(\varphi_{\tau}+\epsilon_{n,\tau}\right)/\tau\right)-\log\bar{\pi}\left(\varphi_{\tau}/\tau\right)\right],\label{eq:2.0-1}
\end{multline}
Even if $\epsilon_{n}$ is discontinuous at each $\tau=\left\lfloor n\tau\right\rfloor /n$,
it still satisfies the condition $\epsilon_{n,\tau}\leq\min\left\{ \tau,\,\left\Vert \epsilon_{n,\tau}\right\Vert \right\} $.
Then, we can proceed as in Lemma \ref{Lemma 5-1}. First consider
the $\log\pi$ dependent integral. 
\begin{equation}
\int_{\tau\in\left[0,1\right]}d\varphi_{\tau}\,\left|\log\pi\left(\left(\varphi_{\tau}+\epsilon_{n,\tau}\right)/\tau\right)-\log\pi\left(\varphi_{\tau}/\tau\right)\right|\leq\left\Vert \bar{\pi}\right\Vert ^{-1}H_{f}\left(\left\Vert \epsilon_{n,\tau}\right\Vert \right).
\end{equation}
Since $\left\Vert \epsilon_{n,\tau}\right\Vert \leq1/n$ we conclude
that $H_{f}\left(\left\Vert \epsilon_{n,\tau}\right\Vert \right)\leq H_{f}\left(1/n\right)$.
Repeating the same steps for the $\log\bar{\pi}$ integral of Eq.
(\ref{eq:2.1-1}) completes the proof .
\end{proof}
Let us now introduce the binomial urn process $B_{n}:=\left\{ B_{n,k}:\,1\leq k\leq n\right\} $,
with constant urn function $\pi\left(s\right)=1/2$ and $B_{n,1}$
uniformly distributed on $\left[0,1\right]$. We define $\delta B_{n,k}:=B_{n,k+1}-B_{n,k}$.
The process $\delta B_{n}$ is a sequence of binary i.i.d. random
variables with $\mathbb{P}\left(\delta B_{n,k}=1\right)=\mathbb{P}\left(\delta B_{n,k}=0\right)=1/2$,
so that each $Y_{n}\left(\varphi\right)$, $\varphi\in\mathcal{Q}_{n}$
realization of $B_{n}$ up to time $n$ has constant measure $\mathbb{P}\left(B_{n}=Y_{n}\left(\varphi\right)\right)=2^{-n}$.
We denote by $\varphi_{n}:\left[0,1\right]\rightarrow\left[0,1\right]$
the linear interpolation of the $n^{-1}B_{k}$ sequence for $0\leq k\leq n$:
\begin{equation}
\beta_{n}:=\left\{ \beta_{n,\tau}=n^{-1}\left[B_{n,\left\lfloor n\tau\right\rfloor }+\left(n\tau-\left\lfloor n\tau\right\rfloor \right)\delta B_{n,\left\lfloor n\tau\right\rfloor }\right]:\tau\in\left[0,1\right]\right\} .\label{eq:2.9}
\end{equation}
Note that $\beta_{n}\in\mathcal{Q}_{n}\subset\mathcal{Q}$ for all
$n$. A sample-path LDP for the sequence of functions $\left\{ \beta_{n}\,:\, n\in\mathbb{N}\right\} $
is provided by the Mogulskii Theorem \cite{Dembo Zeitouni}.
\begin{lem}
\label{Lemma 3}The sequence $\left\{ \beta_{n}\,:\, n\in\mathbb{N}\right\} $
defined by Eq.(\ref{eq:2.9}) with support $\mathcal{\mathcal{Q}}$
satisfies a LDP in $\left(\mathcal{Q},\left\Vert \cdot\right\Vert \right)$,
with the good rate function 
\begin{equation}
I_{1/2}\left[\varphi\right]=\left\{ \begin{array}{l}
\log2+\int_{0}^{1}d\tau\, H\left(\dot{\varphi}_{\tau}\right)\\
\infty
\end{array}\ \begin{array}{l}
if\ \varphi\in\mathcal{AC}\\
otherwise,
\end{array}\right.
\end{equation}
where $\mathcal{AC}$ is the class of absolutely continuous functions,
and $H\left(s\right)=s\log s+\bar{s}\log\bar{s}$ as in Theorem \ref{Theorem 1}. \end{lem}
\begin{proof}
Since $\beta_{n}\in\mathcal{\mathcal{Q}}\subset L_{\infty}\left(\left[0,1\right]\right)$,
Mogulskii Theorem \cite{Dembo Zeitouni} predicts a LDP for the sequence
$\left\{ \beta_{n}\,:\, n\in\mathbb{N}\right\} $, with good rate
function $I_{1/2}\left[\varphi\right]=-\int_{0}^{1}d\tau\hat{\Lambda}\left(\dot{\varphi}_{\tau}\right)$
if $\varphi\in\mathcal{AC}$ and $I_{1/2}\left[\varphi\right]=\infty$
otherwise, and where $\hat{\Lambda}\left(s\right)$ is the Frenchel-Legendre
transform of the moment generating function $\Lambda\left(\lambda\right):=\mathbb{E}\left[\exp\left(\lambda\,\delta Y_{n,1}\right)\right]$.
In our case we have $\Lambda\left(\lambda\right)=\left(e^{\lambda}+1\right)/2$,
then $\hat{\Lambda}\left(s\right)=-\log2-H\left(s\right)$.
\end{proof}

\subsubsection{Proof of Theorem \ref{Theorem 1} for $\pi\in\left(0,1\right)$.}

\label{Section3.1.2}Here we show the theorem for $\pi\in\left(0,1\right)$.
We will use a corollary of the Varadhan Integral Lemma (Lemmas 4.3.2
and 4.3.4 of \cite{Dembo Zeitouni}) to prove the sample-path LDP
for the $\chi_{n}$ sequence stated in Theorem \ref{Theorem 1}.
\begin{proof}
Let $I_{\pi}\left[\varphi\right]:=J\left[\varphi\right]-S_{\pi}\left[\varphi\right]$
and let $\mathcal{B}$ be a subset of $\mathcal{Q}$: we define the
following $\mathcal{B}-$dependent functional: 
\begin{equation}
S_{\pi,\mathcal{B}}\left[\varphi\right]:=\left\{ \begin{array}{l}
S_{\pi}\left[\varphi\right]=J\left[\varphi\right]-I_{\pi}\left[\varphi\right]\\
-\infty
\end{array}\ \begin{array}{l}
if\ \varphi\in\mathcal{B}\\
otherwise.
\end{array}\right.
\end{equation}
and denote by $\mathbb{E}_{0}$ the expectation over the possible
realizations of the binomial process $\beta_{n}$. By equation (\ref{eq:2.1-1})
and Lemma \ref{Lemma 5} we find that 
\begin{multline}
\lim_{n\rightarrow\infty}n^{-1}\log\mathbb{P}\left(\chi_{n}\in\mathcal{B}\right)=\log2+\lim_{n\rightarrow\infty}n^{-1}\log\mathbb{E}_{0}\left(e^{nS_{\pi}\left[\beta_{n}\right]}\mathbb{I}_{\left\{ \beta_{n}\in\mathcal{B}\right\} }\right)=\\
=\log2+\lim_{n\rightarrow\infty}n^{-1}\log\mathbb{E}_{0}\left(e^{nS_{\pi,\mathcal{\mathcal{B}}}\left[\beta_{n}\right]}\right).\label{eq:-3}
\end{multline}
Then, consider $S_{\pi,\mathcal{\mathrm{cl}\left(\mathcal{B}\right)}}$:
since $\mathrm{cl}\left(\mathcal{B}\right)$ is a closed set and Lemma
\ref{Lemma 5-1} states that $S_{\pi}$ is a continuous functional
on $\left(\mathcal{Q},\left\Vert \cdot\right\Vert \right)$ it follows
that $S_{\pi,\mathcal{\mathrm{cl}\left(\mathcal{B}\right)}}$ is upper
semicontinuous on $\left(\mathcal{Q},\left\Vert \cdot\right\Vert \right)$,
and Lemma 4.3.2 of \cite{Dembo Zeitouni} gives the upper bound
\begin{multline}
\log2+\limsup_{n\rightarrow\infty}\, n^{-1}\log\mathbb{E}_{0}\left(e^{nS_{\pi,\mathcal{\mathrm{cl}\left(\mathcal{B}\right)}}\left[\beta_{n}\right]}\right)\leq\log2+\sup_{\varphi\in\mathcal{Q}}\left\{ S_{\pi,\mathcal{\mathrm{cl}\left(\mathcal{B}\right)}}\left[\varphi\right]-I_{1/2}\left[\varphi\right]\right\} =\\
=\log2+\sup_{\varphi\in\mathcal{\mathrm{cl}\left(\mathcal{B}\right)}}\left\{ S_{\pi}\left[\varphi\right]-\log2-J\left[\varphi\right]\right\} =-\inf_{\varphi\in\mathcal{\mathrm{cl}\left(\mathcal{B}\right)}}I_{\pi}\left[\varphi\right].\label{eq:2.18}
\end{multline}
 Now consider $S_{\pi,\mathrm{int}\left(\mathcal{B}\right)}$: $\mathrm{int}\left(\mathcal{B}\right)$
is open and this time we have a lower semicontinuous functional on
$\left(\mathcal{Q},\left\Vert \cdot\right\Vert \right)$, then by
Lemma 4.3.3 of \cite{Dembo Zeitouni} we can write 
\begin{equation}
\log2+\liminf_{n\rightarrow\infty}\, n^{-1}\log\mathbb{E}_{0}\left(e^{nS_{\pi,\mathrm{int}\left(\mathcal{B}\right)}\left[\beta_{n}\right]}\right)\geq-\inf_{\varphi\in\mathrm{int}\left(\mathcal{B}\right)}I_{\pi}\left[\varphi\right].
\end{equation}
which completes the main statement of Theorem \ref{Theorem 1} under
the assumption that $\pi\in\left(0,1\right)$.
\end{proof}
\begin{onehalfspace}
\begin{figure}
\begin{singlespace}
\centering{}\includegraphics[scale=0.45]{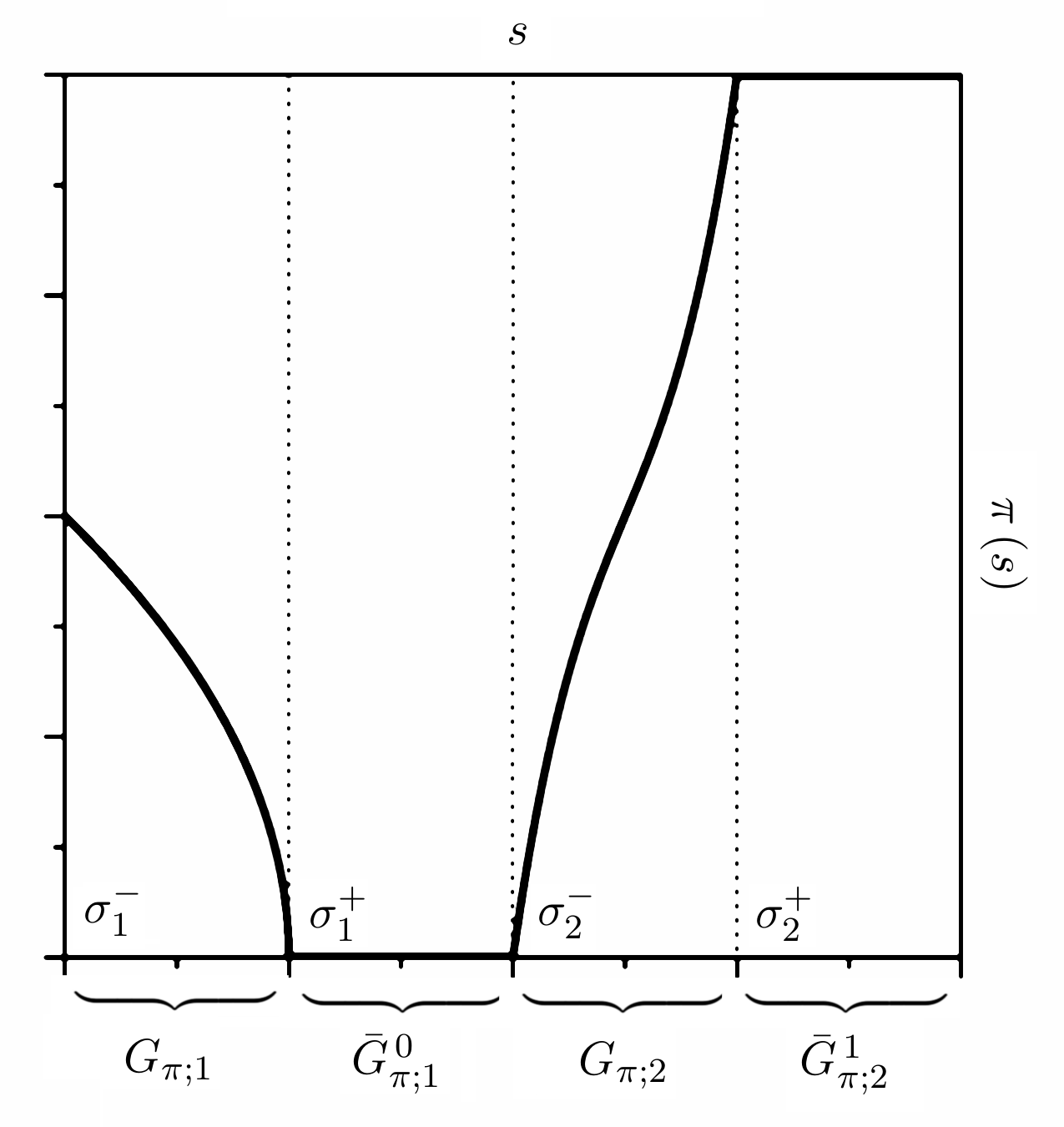}\\
~~\includegraphics[scale=0.45]{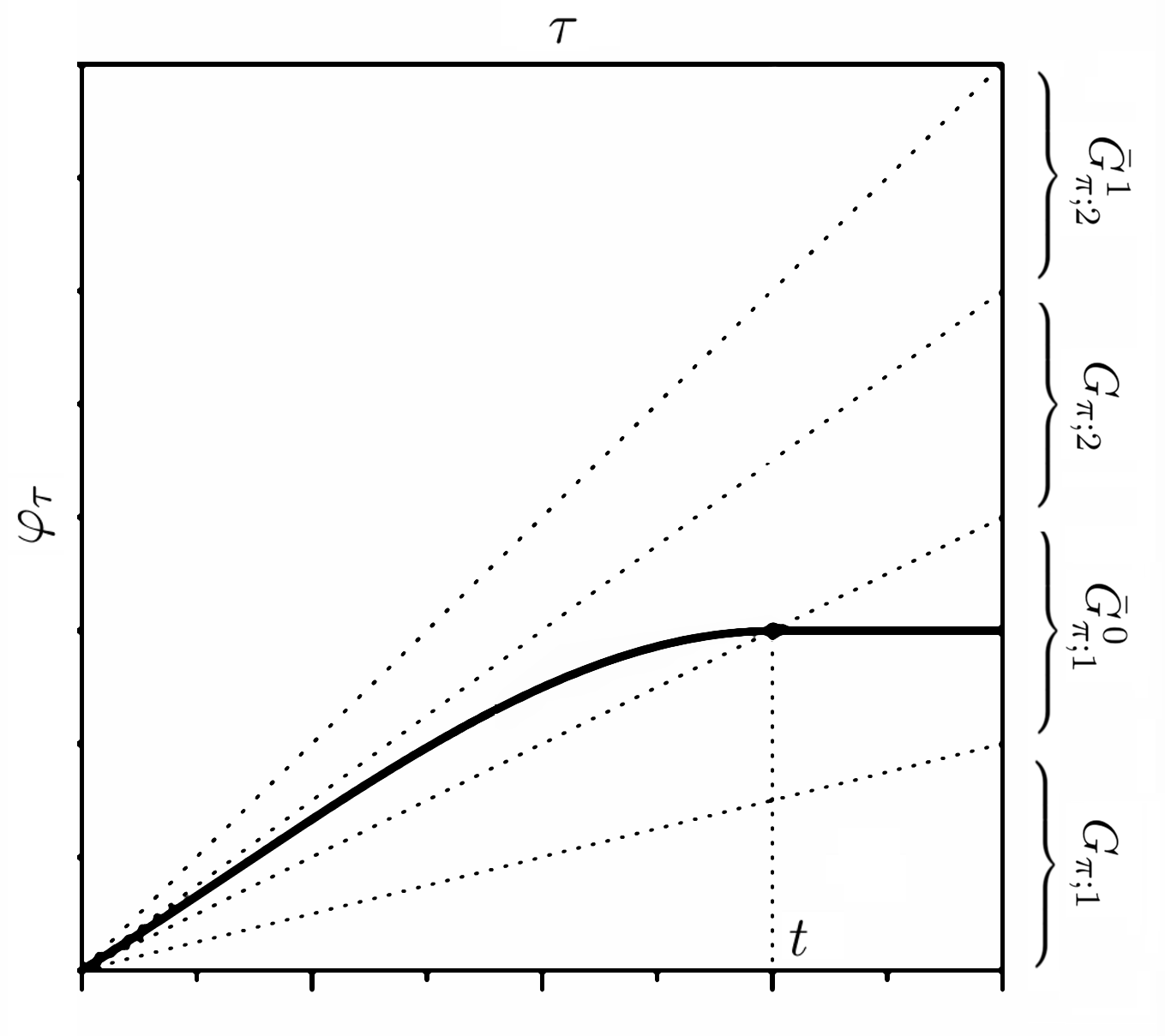}\caption{\label{fig:5}\textit{Example of urn function $\pi$ with relative
$G_{\pi;i}$, $\bar{G}_{\pi;i}^{\alpha}$ intervals (upper figure)
and trajectory with $\lim_{\tau\rightarrow0}\tau^{-1}\varphi_{\tau}\in G_{\pi,2}$,
$\varphi_{1}\in\bar{G}_{\pi;1}^{\,0}$, $S_{\pi}\left[\varphi\right]>-\infty$
(lower figure).}}
\end{singlespace}
\end{figure}

\end{onehalfspace}

\subsubsection{Extension to $\pi\in\left[0,1\right]$: surgery over $\mathcal{Q}$. }

\label{Section3.1.3}When we allow $\pi\left(s\right)$ to be eventually
$0$ or $1$ quantities like $\left\Vert \pi\right\Vert ^{-1}$, $\left\Vert \bar{\pi}\right\Vert ^{-1}$,
$\left\Vert \log\pi\right\Vert $, $\left\Vert \log\bar{\pi}\right\Vert $
may not be bounded and Lemmas \ref{Lemma 5-1} and \ref{Lemma 5}
don't hold anymore. Here we show that we can recover these two lemmas
by a suitable surgery over the set $\mathcal{Q}$ to a priori exclude
those trajectories for which $S_{\pi}\left[\varphi\right]=-\infty$. 
\begin{proof}
The key point is to notice that any $\varphi$ for which $\pi\left(\varphi_{\tau}/\tau\right)=0$
for $\tau\in\left[\tau_{1},\tau_{2}\right]$ with $\left|\tau_{1}-\tau_{2}\right|>0$
gives $S_{\pi}\left[\varphi\right]=-\infty$ unless $d\varphi_{\tau}=0$,
or $d\varphi_{\tau}=1$ if $\pi\left(\varphi_{\tau}/\tau\right)=1$,
in the same $\tau$ interval. To formally explain this we need some
notation. Then, define 
\begin{equation}
G_{\pi}:=\left\{ s\in\left[0,1\right]:\,\pi\left(s\right)\in\left(0,1\right)\right\} ,\,\partial G_{\pi}:=\mathrm{cl}\left(G_{\pi}\right)\setminus\mathrm{int}\left(G_{\pi}\right)\label{eq:FrontierPi}
\end{equation}
and organize the elements of $\partial G_{\pi}$ by increasing order
by labeling them as follows: 
\begin{equation}
\partial G_{\pi}=:\left\{ \sigma_{1}^{-},\sigma_{1}^{+},\sigma_{2}^{-},\sigma_{2}^{+},\,...\,,\,\sigma_{N}^{-},\sigma_{N}^{+}:\,\sigma_{i}^{-}<\sigma_{i}^{+},\sigma_{i}^{+}\leq\sigma_{i+1}^{-}\right\} 
\end{equation}
The above notation allows to define the sequence of intervals
\begin{equation}
G_{\pi;i}:=\left(\sigma_{i}^{-},\sigma_{i}^{+}\right),\,1\leq i\leq N_{g},
\end{equation}
such that $\pi\left(s\right)\in\left(0,1\right)$ for any $s\in G_{\pi;i}:=\left(\sigma_{i}^{-},\sigma_{i}^{+}\right)$
and $G_{\pi}:=\bigcup_{i}G_{\pi;i}$. We can also define the complementary
sequence 
\begin{multline}
\bar{G}_{\pi;0}^{\,\alpha_{0}}:=\left[0,\sigma_{1}^{-}\right],\,\bar{G}_{\pi;i}^{\,\alpha_{i}}:=\left[\sigma_{i}^{+},\sigma_{i+1}^{-}\right],\,\bar{G}_{\pi;N_{g}}^{\,\alpha_{N_{g}}}:=\left[\sigma_{N}^{-},1\right]:\,\alpha_{i}\in\left\{ 0,1\right\} ,\,1\leq i\leq N_{g},\label{eq:Gbar}
\end{multline}
where $\alpha_{i}=\pi\left(s\right)$ for $s\in\left[\sigma_{i}^{+},\sigma_{i+1}^{-}\right]$,
which is $0$ or $1$ by definition. By convention we take $\bar{G}_{\pi;0}^{\,\alpha_{0}}=\textrm{�}$
if $\pi\left(0\right)\in\left(0,1\right)$ and $\bar{G}_{\pi;N_{g}}^{\,\alpha_{N_{g}}}=\textrm{�}$
if $\pi\left(1\right)\in\left(0,1\right)$, and call by 
\begin{equation}
\alpha_{\pi}:=\left\{ \alpha_{i}:\,0\leq i\leq N_{g}\right\} 
\end{equation}
the sequence of the $\alpha_{i}$. Clearly if $\alpha_{0}$ and $\alpha_{N_{g}}$
are not well defined we can exclude them from the above sequence and
take $1\leq i\leq N_{g}-1$.

First we notice that every $\varphi$ such that $\tau^{-1}\varphi_{\tau}\in\bar{G}_{\pi;i}^{\,1}$,
$d\varphi_{\tau}<1$ or $\tau^{-1}\varphi_{\tau}\in\bar{G}_{\pi;0}^{\,0}$,
$d\varphi_{\tau}>0$ in some interval $\tau\in\left[\tau_{1},\tau_{2}\right]$
with $\left|\tau_{1}-\tau_{2}\right|>0$ gives $S_{\pi}\left[\varphi\right]=-\infty$.
Then, we can discard all these cases and restrict our attention to
the following subsets of $\mathcal{Q}$. The simplest subclasses of
$\mathcal{Q}$ for which $S_{\pi}\left[\varphi\right]$ can be a bounded
quantity are those where our $\varphi\in\mathcal{Q}$ is such that
$\tau^{-1}\varphi_{\tau}\in G_{\pi;i}:=\left(\sigma_{i}^{-},\sigma_{i}^{+}\right)$
\begin{equation}
\mathcal{Q}\left[G_{\pi;i}\right]:=\left\{ \varphi\in\mathcal{Q}:\,\tau^{-1}\varphi_{\tau}\in G_{\pi;i}\right\} .
\end{equation}
Anyway, we can build more functions that lives on contiguous intervals
by taking $d\varphi_{\tau}=0$ when $\tau^{-1}\varphi_{\tau}\in\bar{G}_{\pi;i}^{\,0}$
or $d\varphi_{\tau}=d\tau$ when $\tau^{-1}\varphi_{\tau}\in\bar{G}_{\pi;i}^{\,1}$.
As example, consider the subset of $\mathcal{Q}$ such that $\tau^{-1}\varphi_{\tau}\in\bar{G}_{\pi;i-1}^{\,0}\cup\, G_{\pi;i}$,
$\lim_{\tau\rightarrow0}\tau^{-1}\varphi_{\tau}\in G_{\pi,i}$ and
$\varphi_{1}\in\bar{G}_{\pi;i-1}^{\,0}$: we can take $\varphi\in\mathcal{Q}$
such that $\sigma_{i}^{-}<\tau^{-1}\varphi_{\tau}<\sigma_{i}^{+}$
until some time $t\in\left(0,1\right)$, then $\varphi_{\tau}=\sigma_{i}^{+}$
for $t\leq\tau\leq1$, with the obvious requirement that $t\geq\sigma_{i}^{+}/\sigma_{i-1}^{-}$
to ensure that $\varphi_{1}\in\bar{G}_{\pi;i-1}^{\,0}$ (see Figure
\ref{fig:5}). In the above trajectory the time interval $\left(t,1\right)$
in which $\log\pi\left(\tau^{-1}\varphi_{\tau}\right)=-\infty$ also
have $d\varphi_{\tau}=0$, so that its contribution to the total value
of $S_{\pi}$ is null. 
\begin{equation}
\int_{\tau\in\left[t,1\right]}\left[\, d\varphi_{\tau}\,\log\pi\left(\varphi_{\tau}/\tau\right)+d\tilde{\varphi}_{\tau}\,\log\bar{\pi}\left(\varphi_{\tau}/\tau\right)\right]=0.
\end{equation}
The same can be done if $\alpha=1$ and $\tau^{-1}\varphi_{\tau}\in G_{\pi;i}\cup\,\bar{G}_{\pi;i}^{\,1}$
(ie, if $\lim_{\tau\rightarrow0}\tau^{-1}\varphi_{\tau}\in G_{\pi,i}$
and $\varphi_{1}\in\bar{G}_{\pi;i}^{\,1}$): in this case we will
chose $\sigma_{i}^{-}<\tau^{-1}\varphi_{\tau}<\sigma_{i}^{+}$ until
some time $t\in\left[0,1\right]$, then $\varphi_{\tau}=\sigma_{i}^{+}t+\left(\tau-t\right)$
for $t\leq\tau\leq1$ with $t\geq\left(1-\sigma_{i+1}^{-}\right)/\left(1-\sigma_{i}^{+}\right)$.
In general, we can build functions that lives in arbitrary unions
of contiguous intervals, as example $G_{\pi;i}\cup\,\bar{G}_{\pi;i}^{\,\alpha_{i}}\cup\, G_{\pi;i+1}\cup\,\bar{G}_{\pi;i+1}^{\,\alpha_{i+1}}\,...\,\cup\,\bar{G}_{\pi;j}^{\,\alpha_{j}}\cup\, G_{\pi;j+1}$,
provided that $\alpha_{i}=\alpha_{i+1}=\,...\,=\alpha_{j}$. To give
a general characterization of those functions define the following
groups of intervals
\begin{equation}
G_{\pi;i,j}^{\,0}:=\left\{ G_{\pi;i},\,\bar{G}_{\pi;i}^{\,0},\, G_{\pi;i+1},\,\bar{G}_{\pi;i+1}^{\,0},\,...\,,\bar{G}_{\pi;j-1}^{\,0},\, G_{\pi;j}\right\} ,\label{eq:G1}
\end{equation}
\begin{equation}
G_{\pi;i,j}^{\,1}:=\left\{ G_{\pi;i},\,\bar{G}_{\pi;i}^{\,1},\, G_{\pi;i+1},\,\bar{G}_{\pi;i+1}^{\,1},\,...\,,\bar{G}_{\pi;j-1}^{\,1},\, G_{\pi;j}\right\} ,\label{eq:G2}
\end{equation}
\begin{equation}
\bar{G}_{\pi;i,j}^{\,0}:=\left\{ \bar{G}_{\pi;i-1}^{0},G_{\pi;i},\,\bar{G}_{\pi;i}^{\,0},\, G_{\pi;i+1},\,...\,,\bar{G}_{\pi;j-1}^{\,0},\, G_{\pi;j}\right\} ,\label{eq:G3}
\end{equation}
\begin{equation}
\bar{G}_{\pi;i,j}^{\,1}:=\left\{ G_{\pi;i},\,\bar{G}_{\pi;i}^{\,1},\, G_{\pi;i+1},\,.....\,,\bar{G}_{\pi;j-1}^{\,1},\, G_{\pi;j},\,\bar{G}_{\pi;j}^{1}\right\} .\label{eq:G4}
\end{equation}
From each of the above groups of intervals we can define a subset
of $\mathcal{Q}$ as follows. First consider $G_{\pi;i,j}^{\,0}$,
take some $s\in G_{\pi;i}$ and denote by $T_{i,j}$ a general time
sequence
\begin{equation}
T_{i,j}:=\left\{ t_{k}\in\left[0,1\right]:\, i\leq k\leq j\right\} .
\end{equation}
Then we can define a set of $T_{i,j}$ sequences 
\begin{equation}
T_{s}\left[G_{\pi;i,j}^{\,0}\right]:=\left\{ T_{i+1,j}:\,0<\left(\sigma_{k}^{-}/\sigma_{k-1}^{+}\right)t{}_{k}\leq t_{k-1}\leq\left(s/\sigma_{i+1}^{-}\right)\right\} 
\end{equation}
and the associated set of trajectories $\mathcal{Q}_{s}[G_{\pi;i,j}^{\,0},\, T_{i+1,j}]\subseteq\mathcal{Q}$
\begin{multline}
\mathcal{Q}_{s}\left[G_{\pi;i,j}^{\,0},\, T_{i+1,j}\right]:=\{\varphi\in\mathcal{Q}:\, i+1\leq k\leq j-1;\,\varphi_{1}=s;\\
\tau^{-1}\varphi_{\tau}\in G_{\pi;j},\,\tau\in\left[0,t_{j}\right];\,\varphi_{\tau}=\sigma_{k+1}^{-}t_{k+1},\,\tau\in\left[t_{k+1},t'_{k+1}\right];\\
\tau^{-1}\varphi_{\tau}\in G_{\pi;k},\,\tau\in\left[t'_{k+1},t_{k}\right];\,\varphi_{\tau}=\sigma_{k}^{-}t_{k},\,\tau\in\left[t_{k},t'_{k}\right];\\
\tau^{-1}\varphi_{\tau}\in G_{\pi;i},\,\tau\in\left[t'_{i+1},1\right];\, t_{k}^{'}:=\left(\sigma_{k}^{-}/\sigma_{k-1}^{+}\right)t_{k}\},\label{eq:Q1}
\end{multline}
with $\lim_{\tau\rightarrow0}\tau^{-1}\varphi_{\tau}\in G_{\pi;j}$
and ending in $\varphi_{1}=s\in G_{\pi;j}$. At this point we can
define 
\begin{equation}
\mathcal{Q}\left[G_{\pi;i,j}^{\,0}\right]:=\bigcup_{s\in G_{\pi;i}}\ \bigcup_{T_{i+1,j}\in\, T_{s}\left[G_{\pi;i,j}^{\,0}\right]}\ {\textstyle \mathcal{Q}_{s}\,[G_{\pi;i,j}^{\,0},\, T_{i+1,j}]},\label{eq:QG1}
\end{equation}
which is the set of trajectories with $\lim_{\tau\rightarrow0}\tau^{-1}\varphi_{\tau}\in G_{\pi;j}$
and $\varphi_{1}\in G_{\pi;i}$ for which $S_{\pi}\left[\varphi\right]$
may still be a bounded quantity. We can do the same for the remaining
classes of sets. For $G_{\pi;i,j}^{\,1}$ we take $s\in G_{\pi;j}$,
define 
\begin{equation}
T_{s}\left[G_{\pi;i,j}^{\,1}\right]:=\left\{ T_{i,j-1}:\,0<\left(\bar{\sigma}_{k}^{-}/\bar{\sigma}_{k+1}^{+}\right)t{}_{k}\leq t_{k+1}\leq\left(\bar{s}/\bar{\sigma}_{i+1}^{-}\right)\right\} ,
\end{equation}
\begin{multline}
\mathcal{Q}_{s}\left[G_{\pi;i,j}^{\,1},\, T_{i,j-1}\right]:=\{\varphi\in\mathcal{Q}:\, i-1\leq k\leq j+1;\,\varphi_{1}=s;\\
\tau^{-1}\varphi_{\tau}\in G_{\pi;i},\,\tau\in\left[0,t_{i}\right];\,\varphi_{\tau}=\tau-\bar{\sigma}_{k-1}^{+}t_{k-1},\,\tau\in\left[t_{k-1},t'_{k-1}\right];\\
\tau^{-1}\varphi_{\tau}\in G_{\pi;k},\,\tau\in\left[t'_{k-1},t_{k}\right];\,\varphi_{\tau}=\tau-\bar{\sigma}_{k}^{+}t_{k},\,\tau\in\left[t_{k},t'_{k}\right];\\
\tau^{-1}\varphi_{\tau}\in G_{\pi;j},\,\tau\in\left[t'_{j-1},1\right];\, t_{k}^{'}:=\left(\bar{\sigma}_{k}^{+}/\bar{\sigma}_{k+1}^{-}\right)t_{k}\},\label{eq:Q2}
\end{multline}
to obtain set of trajectories with $\lim_{\tau\rightarrow0}\tau^{-1}\varphi_{\tau}\in G_{\pi;i}$
and $\varphi_{1}\in G_{\pi;j}$
\begin{equation}
\mathcal{Q}\left[G_{\pi;i,j}^{\,1}\right]:=\bigcup_{s\in G_{\pi;i}}\ \bigcup_{T_{i+1,j}\in\, T_{s}\left[G_{\pi;i,j}^{\,1}\right]}\ {\textstyle \mathcal{Q}_{s}\,[G_{\pi;i,j}^{\,1},\, T_{i,j-1}]}\label{eq:QG2}
\end{equation}
associated to $G_{\pi;i,j}^{\,1}$. Then we take some $s\in\bar{G}_{\pi;i-1}^{\,0}$,
define 
\begin{equation}
T_{s}\left[\bar{G}_{\pi;i,j}^{\,0}\right]:=\left\{ T_{i,j}:\,0\leq t_{k}\leq\left(\sigma_{k}^{-}/\sigma_{k-1}^{+}\right)t{}_{k}\leq t_{k-1}<1;\, t_{i}=\left(s/\sigma_{i}^{-}\right)\right\} ,
\end{equation}
\begin{multline}
\mathcal{Q}_{s}\left[\bar{G}_{\pi;i,j}^{\,0},\, T_{i,j}\right]:=\{\varphi\in\mathcal{Q}:\, i+1\leq k\leq j;\,\tau^{-1}\varphi_{\tau}\in G_{\pi;j},\,\tau\in\left[0,t_{j}\right];\\
\varphi_{\tau}=\sigma_{k}^{-}t_{k},\,\tau\in\left[t_{k},t'_{k}\right];\,\tau^{-1}\varphi_{\tau}\in G_{\pi;k-1},\,\tau\in\left[t'_{k},t_{k-1}\right];\\
\varphi_{\tau}=st_{i},\,\tau\in\left[t_{i},1\right];\, t'_{k}:=\left(\sigma_{k}^{-}/\sigma_{k-1}^{+}\right)t_{k}\}\label{eq:Q3}
\end{multline}
and define trajectories with $\lim_{\tau\rightarrow0}\tau^{-1}\varphi_{\tau}\in G_{\pi;j}$
and $\varphi_{1}\in\bar{G}_{\pi;i-1}^{\,0}$
\begin{equation}
\mathcal{Q}\left[\bar{G}_{\pi;i,j}^{\,0}\right]:=\bigcup_{s\in\bar{G}_{\pi;i-1}}\ \bigcup_{T_{i,j}\in\, T_{s}\left[\bar{G}_{\pi;i,j}^{\,0}\right]}\ {\textstyle \mathcal{Q}_{s}\,[\bar{G}_{\pi;i,j}^{\,0},\, T_{i,j}]}.\label{eq:QG3}
\end{equation}
Finally, let $s\in\bar{G}_{\pi;j}^{\,1}$, 
\begin{equation}
T_{s}\left[\bar{G}_{\pi;i,j}^{\,1}\right]:=\left\{ T_{i,j}:\,0\leq t_{k}\leq\left(\bar{\sigma}_{k}^{-}/\bar{\sigma}_{k+1}^{+}\right)t{}_{k}\leq t_{k+1}<1;\, t_{j}=\left(\bar{s}/\bar{\sigma}_{i}^{-}\right)\right\} ,
\end{equation}
\begin{multline}
\mathcal{Q}_{s}\left[\bar{G}_{\pi;i,j}^{\,1},\, T_{i,j}\right]:=\{\varphi\in\mathcal{Q}:\, i\leq k\leq j-1;\,\tau^{-1}\varphi_{\tau}\in G_{\pi;i},\,\tau\in\left[0,t_{i}\right];\\
\varphi_{\tau}=\tau-\bar{\sigma}_{k}^{+}t_{k},\,\tau\in\left[t_{k},t'_{k}\right];\,\tau^{-1}\varphi_{\tau}\in G_{\pi;k+1},\,\tau\in\left[t'_{k},t_{k+1}\right];\\
\varphi_{\tau}=\tau-\bar{s},\,\tau\in\left[t_{j},1\right];\, t'_{k}:=\left(\bar{\sigma}_{k}^{+}/\bar{\sigma}_{k+1}^{-}\right)t_{k}\},\label{eq:Q4}
\end{multline}
and the set of trajectories with $\lim_{\tau\rightarrow0}\tau^{-1}\varphi_{\tau}\in G_{\pi;i}$
and $\varphi_{1}\in\bar{G}_{\pi;j}^{\,1}$ be
\begin{equation}
\mathcal{Q}\left[\bar{G}_{\pi;i,j}^{\,1}\right]:=\bigcup_{s\in\bar{G}_{\pi;j}}\ \bigcup_{T_{i,j}\in\, T_{s}\left[\bar{G}_{\pi;i,j}^{\,1}\right]}\ {\textstyle \mathcal{Q}_{s}\,[\bar{G}_{\pi;i,j}^{\,1},\, T_{i,j}]}.\label{eq:QG4}
\end{equation}
By continuity of $\pi$ we observe that the number $N_{g}$ of connected
intervals in which $\pi$ is $0$ or $1$ is finite, then also is
the number of combination of contiguous intervals $G_{\pi;i,j}^{\,0}$,
$G_{\pi;i,j}^{\,1}$, $\bar{G}_{\pi;i,j}^{\,0}$, and $\bar{G}_{\pi;i,j}^{\,1}$
satisfying the condition $\alpha_{i}=\alpha_{i+1}=\,...\,=\alpha_{j}=\alpha\in\left\{ 0,1\right\} $.
Calling $N_{g}^{*}$ the number of these combination of intervals,
plus the elementary intervals $G_{\pi;i}$, we can considerably lighten
our notation by relabeling as $\mathcal{Q}_{k}$, $1\leq k\leq N_{g}^{*}$
their associated subsets of $\mathcal{Q}$ defined by Eq.s (\ref{eq:QG1}),
(\ref{eq:QG2}), (\ref{eq:QG3}) and (\ref{eq:QG4}). 

Since for any $\varphi$ that does not belong to $\mathcal{Q}_{k}$,
$1\leq k\leq N_{g}^{*}$ we will find $S_{\pi}\left[\varphi\right]=-\infty$
we can use the relation $\mathbb{P}\left(\chi_{n}\in\mathcal{B}\right)=\sum_{i\leq k\leq N_{g}^{*}}\mathbb{P}\left(\chi_{n}\in\mathcal{B}\cap\mathcal{Q}_{k}\right)$
to conclude that 
\begin{equation}
\lim_{n\rightarrow\infty}n^{-1}\log\mathbb{P}\left(\chi_{n}\in\mathcal{B}\right)=\sup_{1\leq k\leq N^{*}}\lim_{n\rightarrow\infty}n^{-1}\log\mathbb{P}\left(\chi_{n}\in\mathcal{B}\cap\mathcal{Q}_{k}\right)
\end{equation}
and restrict our attention to $\varphi\in\mathcal{Q}_{k}$. 
\end{proof}
\begin{onehalfspace}
\begin{figure}
\begin{singlespace}
\centering{}\includegraphics[scale=0.45]{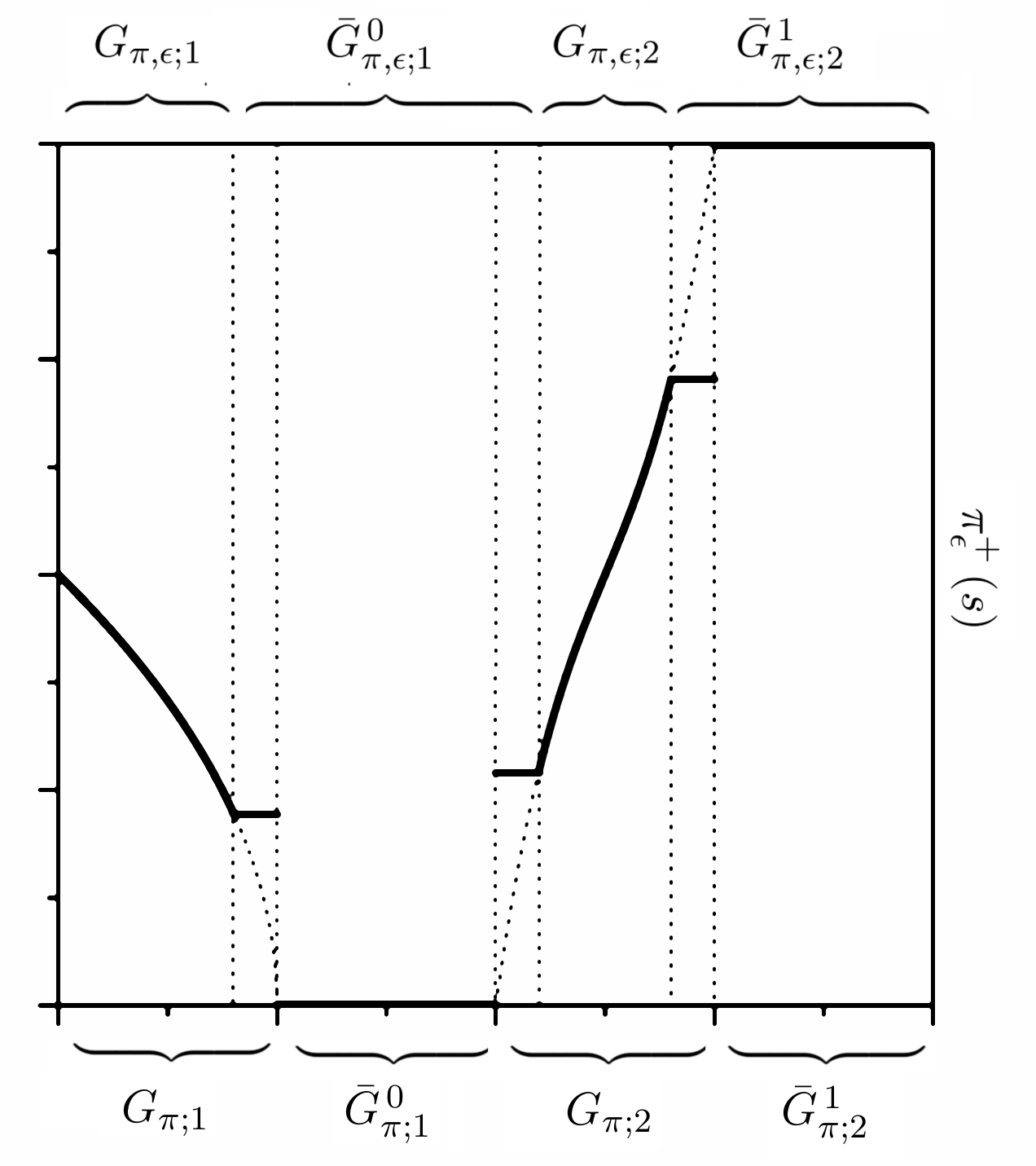}\\
~~\includegraphics[scale=0.45]{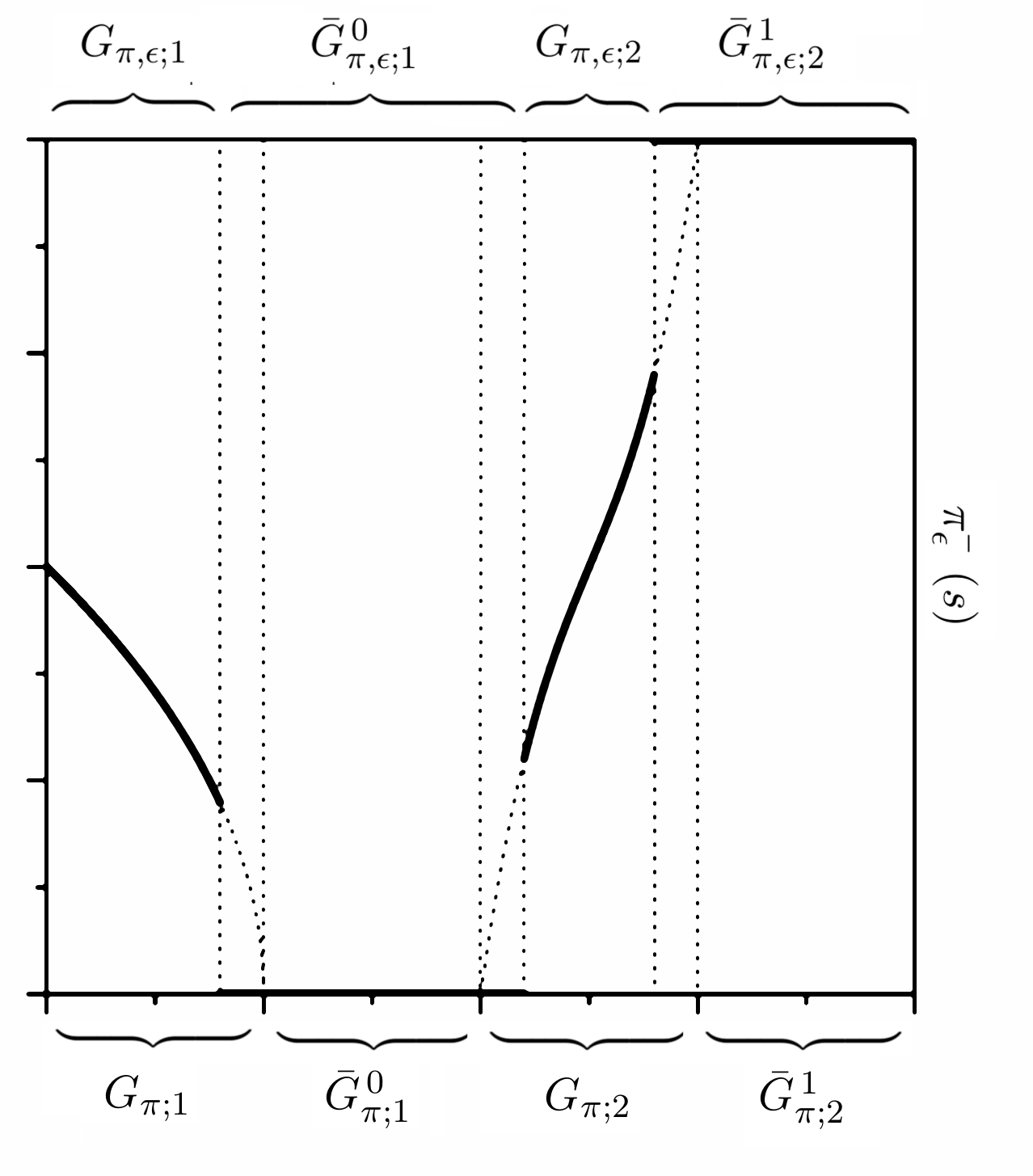}\caption{\label{fig:6}\textit{Functions $\pi_{\epsilon}^{+}$ (upper figure)
and $\pi_{\epsilon}^{-}$ (lower figure) as defined by Eq.s (\ref{eq:defp1})
and (\ref{eq:defp2}) from the same urn function in Figure \ref{fig:5}.}}
\end{singlespace}
\end{figure}

\end{onehalfspace}

\subsubsection{Extension to $\pi\in\left[0,1\right]$: singularities on the edges
of $G_{\pi,i}$.}

\label{Section3.1.4}The above argument fixes the problem of having
$\log\pi\left(\tau^{-1}\varphi_{\tau}\right)=-\infty$ when $\tau^{-1}\varphi_{\tau}\in\bar{G}_{\pi;i}^{\,0}$
(or $\log\bar{\pi}\left(\tau^{-1}\varphi_{\tau}\right)=-\infty$ when
$\tau^{-1}\varphi_{\tau}\in\bar{G}_{\pi;i}^{\,1}$), but we still
have $\pi\left(s\right)\rightarrow0$ or $1$ when $s\rightarrow\sigma_{i}^{\pm}$,
which prevent us from recovering Lemmas \ref{Lemma 5-1} and \ref{Lemma 5}.
To circumvent this last issue we can proceed as follows. 
\begin{proof}
Take some small $\epsilon>0$ and define $G_{\pi,\epsilon;i}$, $\bar{G}_{\pi,\epsilon;i}^{\,\alpha_{i}}$
as in Eq.s (\ref{eq:G1}), (\ref{eq:G2}), (\ref{eq:G3}), (\ref{eq:G4})
above with $\sigma_{i}^{-}+\epsilon$ in place of $\sigma_{i}^{-}$
and $\sigma_{i}^{+}-\epsilon$ in place of $\sigma_{i}^{+}$, such
that some $\delta_{\epsilon}>0$ exists for which 
\begin{equation}
\sup_{i}\sup_{s\in G_{\pi,\epsilon;i}}\pi\left(s\right)\geq\delta_{\epsilon},\ \sup_{i}\sup_{s\in G_{\pi,\epsilon;i}}\bar{\pi}\left(s\right)\geq\delta_{\epsilon}.
\end{equation}
Then, define the discontinuous functions $\pi_{\epsilon}^{+}\geq\pi$
and $\pi_{\epsilon}^{-}\leq\pi$ as follows:
\begin{multline}
\pi_{\epsilon}^{+}:=\{\pi_{\epsilon}^{+}\left(s\right),\, s\in\left[0,1\right]:\,\pi_{\epsilon}^{+}\left(s\right)=\pi\left(s\right),\, s\in G_{\pi,\epsilon;i}:=\left(\sigma_{i}^{-}+\epsilon,\sigma_{i}^{+}-\epsilon\right);\\
\pi\left(s\right)=\pi\left(\sigma_{i}^{-}+\epsilon\right),\, s\in\left[\sigma_{i}^{-},\sigma_{i}^{-}+\epsilon\right];\,\pi_{\epsilon}^{+}\left(s\right)=\pi\left(\sigma_{i}^{+}-\epsilon\right),\, s\in\left[\sigma_{i}^{+}-\epsilon,\sigma_{i}^{+}\right]\},\label{eq:defp1}
\end{multline}
\begin{multline}
\pi_{\epsilon}^{-}:=\{\pi_{\epsilon}^{-}\left(s\right),\, s\in\left[0,1\right]:\,\pi_{\epsilon}^{-}\left(s\right)=\pi\left(s\right),\, s\in G_{\pi,\epsilon;i}:=\left(\sigma_{i}^{-}+\epsilon,\sigma_{i}^{+}-\epsilon\right);\\
\pi_{\epsilon}^{-}\left(s\right)=\alpha_{\, i-1},\, s\in\left[\sigma_{i}^{-},\sigma_{i}^{-}+\epsilon\right];\,\pi_{\epsilon}^{-}\left(s\right)=\alpha_{i},\, s\in\left[\sigma_{i}^{+}-\epsilon,\sigma_{i}^{+}\right]\}.\label{eq:defp2}
\end{multline}
Our proof will consist in showing Theorem \ref{Theorem 1} for the
above modified urn functions and then provide an argument to take
$\epsilon\rightarrow0$. 

Let first consider $\pi_{\epsilon}^{+}$. Since by definition we can
bound $\pi_{\epsilon}^{+}\left(s\right)\geq\delta_{\epsilon}$ and
$\bar{\pi}_{\epsilon}^{+}\left(s\right)\geq\delta_{\epsilon}$ when
$s\in G_{\pi;i}$, it is clear that both Lemmas \ref{Lemma 5-1},
\ref{Lemma 5} would hold again for $\pi_{\epsilon}^{+}$ in each
metric space $(\mathcal{Q}_{k},\,\left\Vert \cdot\right\Vert )$,
with some $W_{\pi_{\epsilon}^{+}}\left(s,\,\epsilon\right)$ such
that $\lim_{s\rightarrow0,}W_{\pi_{\epsilon}^{+}}\left(s,\,\epsilon\right)=0$
for any $\epsilon>0$ in place of of $W_{\pi}\left(s\right)$. Then
we can apply the proof for $\pi\in\left(0,1\right)$ to the events
$\mathcal{B}\cap\mathcal{Q}_{k}$, obtaining for $\pi_{\epsilon}^{+}$
\begin{equation}
\limsup_{n\rightarrow\infty}n^{-1}\log\mathbb{P}\left(\chi_{n}\in\mathcal{B}\right)\leq-\inf_{1\leq k\leq N^{*}}\inf_{\varphi\in\mathcal{\mathrm{cl}}\left(\mathcal{B\,}\cap\mathcal{Q}_{k}\right)}I_{\pi_{\epsilon}}\left[\varphi\right]=-\inf_{\varphi\in\mathcal{\mathrm{cl}\left(\mathcal{B}\right)}}I_{\pi_{\epsilon}^{+}}\left[\varphi\right],\label{eq:LD1}
\end{equation}
\begin{equation}
\liminf_{n\rightarrow\infty}n^{-1}\log\mathbb{P}\left(\chi_{n}\in\mathcal{B}\right)\geq-\inf_{1\leq k\leq N^{*}}\inf_{\varphi\in\mathcal{\mathrm{int}}\left(\mathcal{B\,}\cap\mathcal{Q}_{k}\right)}I_{\pi_{\epsilon}^{+}}\left[\varphi\right]=-\inf_{\varphi\in\mathcal{\mathrm{int}\left(\mathcal{B}\right)}}I_{\pi_{\epsilon}^{+}}\left[\varphi\right].\label{eq:LD2}
\end{equation}
 We can produce an identical reasoning for $\pi_{\epsilon}^{-}$,
provided we consider $G_{\pi,\epsilon;i}$ on place of of $G_{\pi;i}$
in the definitions of the sets $\mathcal{Q}_{k}$, $1\leq k\leq N_{g}^{*}$:
we will relabel them as $\mathcal{Q}_{\epsilon;k}$, $1\leq k\leq N_{g}^{*}$
to emphasize the dependence on $\epsilon$ of the intervals. Then,
also for $\pi_{\epsilon}^{-}$ we can write 
\begin{equation}
\limsup_{n\rightarrow\infty}n^{-1}\log\mathbb{P}\left(\chi_{n}\in\mathcal{B}\right)\leq-\inf_{1\leq k\leq N^{*}}\inf_{\varphi\,\in\mathcal{\mathrm{cl}}\left(\mathcal{B}\,\cap\,\mathcal{Q}_{\epsilon;k}\right)}I_{\pi_{\epsilon}^{-}}\left[\varphi\right]=-\inf_{\varphi\,\in\mathcal{\mathrm{cl}\left(\mathcal{B}\right)}}I_{\pi_{\epsilon}^{-}}\left[\varphi\right],\label{eq:LD3}
\end{equation}
\begin{equation}
\liminf_{n\rightarrow\infty}n^{-1}\log\mathbb{P}\left(\chi_{n}\in\mathcal{B}\right)\geq-\inf_{1\leq k\leq N^{*}}\inf_{\varphi\,\in\mathcal{\mathrm{cl}}\left(\mathcal{B}\,\cap\,\mathcal{Q}_{\epsilon;k}\right)}I_{\pi_{\epsilon}^{-}}\left[\varphi\right]=-\inf_{\varphi\in\,\mathcal{\mathrm{int}\left(\mathcal{B}\right)}}I_{\pi_{\epsilon}^{-}}\left[\varphi\right].\label{eq:LD4}
\end{equation}
The last step is to prove that for any Borel subset $\mathcal{B}$
of $\mathcal{Q}$ 
\begin{equation}
\lim_{\epsilon\rightarrow0}\,\inf_{\varphi\,\in\,\mathcal{B\,}\cap\,\mathcal{Q}_{k}}\, I_{\pi_{\epsilon}^{+}}\left[\varphi\right]=\lim_{\epsilon\rightarrow0}\,\inf_{\varphi\,\in\,\mathcal{B\,}\cap\,\mathcal{Q}_{\epsilon;k}}\, I_{\pi_{\epsilon}^{-}}\left[\varphi\right]=\inf_{\varphi\,\in\,\mathcal{B\,}\cap\,\mathcal{Q}_{k}}\, I_{\pi}\left[\varphi\right].\label{eq:4.333}
\end{equation}
We will explicitly prove this relation only for subsets of the kind
$\mathcal{Q\,}[\, G_{\pi;i,j}^{\,0}]$, since all other cases can
be shown using the same technique with minimal modifications. Then
let $\mathcal{Q\,}[\, G_{\pi;i,j}^{\,0}]$ as in Eq. (\ref{eq:QG1})
and call $\mathcal{Q\,}[\, G_{\pi,\epsilon;i,j}^{\,0}]$ its version
with $\sigma_{k}^{+}-\epsilon$ on place of of $\sigma_{k}^{+}$ and
$\sigma_{k}^{-}+\epsilon$ on place of of $\sigma_{k}^{-}$. By Eq.
(\ref{eq:QG1}), to prove Eq. (\ref{eq:4.333}) it suffices to show
that
\begin{equation}
\lim_{\epsilon\rightarrow0}\ \inf_{\varphi\,\in\,\mathcal{B\,}\cap\,\mathcal{Q}_{s}[G_{\pi;i,j}^{\,0},\, T_{i+1,j}]}\, I_{\pi_{\epsilon}^{+}}\left[\varphi\right]=\lim_{\epsilon\rightarrow0}\ \inf_{\varphi\,\in\,\mathcal{B\,}\cap\mathcal{Q}_{s}[G_{\pi,\epsilon;i,j}^{\,0},\, T_{\epsilon,i+1,j}^{*}]}\, I_{\pi_{\epsilon}^{-}}\left[\varphi\right]=\inf_{\varphi\,\in\,\mathcal{B\,}\cap\,\mathcal{Q}_{s}[G_{\pi;i,j}^{\,0},\, T_{i+1,j}]}\, I_{\pi}\left[\varphi\right],\label{eq:hhhhhhhhhhhhhhhh}
\end{equation}
with $s\in G_{\pi;i}$, $T_{i+1,j}\in T_{s}[G_{\pi;i,j}^{\,0}]$ and
\begin{equation}
T_{\epsilon,i+1,j}^{*}:=\left\{ t_{\epsilon,k}:=\left(\sigma_{k}^{-}/\left(\sigma_{k}^{-}+\epsilon\right)\right)t_{k}:\, i+1\leq k\leq j\right\} .
\end{equation}
Then, define the optimal trajectories of the variational problems
for $\pi_{\epsilon}^{+}$ and $\pi_{\epsilon}^{-}$: 
\begin{equation}
\varphi^{+}:\, I_{\pi_{\epsilon}^{+}}[\varphi^{+}]=\,\,\inf_{\varphi\,\in\,\mathcal{B\,}\cap\,\mathcal{Q}_{s}[G_{\pi;i,j}^{\,0},\, T_{i+1,j}]}\,\, I_{\pi_{\epsilon}^{+}}\left[\varphi\right],
\end{equation}
\begin{equation}
\varphi^{-}:\, I_{\pi_{\epsilon}^{-}}[\varphi^{-}]=\inf_{\varphi\,\in\,\mathcal{B\,}\cap\mathcal{Q}_{s}[G_{\pi,\epsilon;i,j}^{\,0},\, T_{\epsilon,i+1,j}^{*}]}I_{\pi_{\epsilon}^{-}}\left[\varphi\right].
\end{equation}
Since $\varphi^{+}$ may not belong to $\mathcal{Q}_{s}[G_{\pi,\epsilon;i,j}^{\,0},\, T_{\epsilon,i+1,j}^{*}]$
it will be useful to introduce a modified trajectory $\varphi_{\epsilon}^{+}:=\left\{ \varphi_{\epsilon;\tau}^{+}:\,\tau\in\left[0,1\right]\right\} $,
defined as follows
\begin{equation}
\varphi_{\epsilon;\tau}^{+}:=\left\{ \begin{array}{l}
\underset{\,}{\sigma_{k}^{+}t'_{k+1},}\\
\inf\left\{ \left(\sigma_{k}^{-}+\epsilon\right)\tau,\,\sup\left\{ \varphi_{\tau}^{+},\left(\sigma_{k}^{+}-\epsilon\right)\tau\right\} \right\} ,\\
\overset{\,}{\sigma_{k}^{-}t_{k},}
\end{array}\ \begin{array}{l}
\underset{\,}{t'_{k+1}<\tau<t'_{\epsilon,k+1},}\\
\underset{\,}{t'_{\epsilon,k+1}\leq\tau<t{}_{\epsilon,k},}\\
t{}_{\epsilon,k}<\tau<t{}_{k},
\end{array}\right.
\end{equation}
with $i\leq k\leq j$ and $t'_{\epsilon,k}=\left(\left(\sigma_{k}^{-}+\epsilon\right)/\left(\sigma_{k-1}^{+}-\epsilon\right)\right)t_{\epsilon,k}$
as for $t'_{k}$. The scope of this modified trajectory will be clear
after we state the following auxiliary relations. By definition of
$\varphi_{\tau}^{-}$ as optimal trajectory for $I_{\pi_{\epsilon}^{-}}$
we find $I_{\pi_{\epsilon}^{-}}[\varphi_{\epsilon}^{+}]\geq I_{\pi_{\epsilon}^{-}}[\varphi^{-}]$,
while by definition of $\varphi_{\tau}^{+}$ we have $I_{\pi_{\epsilon}^{+}}[\varphi^{-}]\geq I_{\pi_{\epsilon}^{+}}[\varphi^{+}]$.
Now let $\Gamma_{\epsilon}:=I_{\pi_{\epsilon}^{+}}[\varphi^{+}]-I_{\pi_{\epsilon}^{+}}[\varphi_{\epsilon}^{+}]$.
By continuity of $I_{\pi_{\epsilon}^{+}}$ we can write 
\begin{equation}
\lim_{\epsilon\rightarrow0}\Gamma_{\epsilon}:=\lim_{\epsilon\rightarrow0}\,(\, I_{\pi_{\epsilon}^{+}}[\varphi^{+}]-I_{\pi_{\epsilon}^{+}}[\varphi_{\epsilon}^{+}]\,)=0.
\end{equation}
Then, consider $I_{\pi_{\epsilon}^{+}}[\varphi_{\epsilon}^{+}]$ and
$I_{\pi_{\epsilon}^{-}}[\varphi_{\epsilon}^{+}]$. Since $\pi_{\epsilon}^{+}\left(\tau^{-1}\varphi_{\epsilon;\tau}^{+}\right)=\pi_{\epsilon}^{-}\left(\tau^{-1}\varphi_{\epsilon;\tau}^{+}\right)$
for $\tau\in[t'_{\epsilon,k+1},t{}_{\epsilon,k},]$ by construction
their difference lies only in the intervals when $(t'_{k},\, t'_{\epsilon,k})$
and $\left(t{}_{\epsilon,k},\, t{}_{k}\right)$, so that we can bound
as
\begin{multline}
\left|\Delta_{\epsilon}\right|:=|I_{\pi_{\epsilon}^{+}}[\varphi_{\epsilon}^{+}]-I_{\pi_{\epsilon}^{-}}[\varphi_{\epsilon}^{+}]|=\sum_{i\leq k\leq j}\int_{\tau\in(t'_{k+1},\, t'_{\epsilon,k+1})\cup(t{}_{\epsilon,k},\, t{}_{k})}d\tau\,\left|\log\bar{\pi}_{\epsilon}^{+}\left(\tau^{-1}\varphi_{\tau}\right)\right|=\\
=\left(j-i\right)\left[\left(t'_{k+1}-t'_{\epsilon,k+1}\right)+\left(t{}_{\epsilon,k}-t{}_{k}\right)\right]\left|\log\bar{\pi}_{\epsilon}^{+}\left(t_{\epsilon}^{-1}\left(\sigma_{k}^{+}-\epsilon\right)\right)\right|\leq\\
\leq\left(j-i\right)\left(\left|t'_{k+1}-t'_{\epsilon,k+1}\right|+\left|t{}_{k+1}-t{}_{\epsilon,k+1}\right|\right)\delta_{\epsilon}.\label{eq:kkk1kkk1}
\end{multline}
The same considerations hold for $I_{\pi_{\epsilon}^{+}}[\varphi_{\epsilon}^{-}]$
and $I_{\pi_{\epsilon}^{-}}[\varphi_{\epsilon}^{-}]$, for which again
one finds $I_{\pi_{\epsilon}^{+}}[\varphi^{-}]-I_{\pi_{\epsilon}^{-}}[\varphi^{-}]=\Delta_{\epsilon}.$
Collecting the above relations we find 
\begin{equation}
I_{\pi_{\epsilon}^{-}}[\varphi^{-}]\leq I_{\pi_{\epsilon}^{-}}[\varphi_{\epsilon}^{+}]=I_{\pi_{\epsilon}^{+}}[\varphi_{\epsilon}^{+}]+\Delta_{\epsilon}=I_{\pi_{\epsilon}^{+}}[\varphi^{+}]-\Gamma_{\epsilon}+\Delta_{\epsilon},
\end{equation}
\begin{equation}
I_{\pi_{\epsilon}^{+}}[\varphi^{+}]\leq I_{\pi_{\epsilon}^{+}}[\varphi^{-}]=I_{\pi_{\epsilon}^{-}}[\varphi^{-}]+\Delta_{\epsilon},
\end{equation}
from which follows that 
\begin{equation}
\lim_{\epsilon\rightarrow0}I_{\pi_{\epsilon}^{+}}[\varphi^{+}]=\lim_{\epsilon\rightarrow0}I_{\pi_{\epsilon}^{-}}[\varphi^{-}].\label{eq:4.44}
\end{equation}
Now consider the optimal trajectory $\varphi^{*}$ of the variational
problem for the original $\pi$ 
\begin{equation}
\varphi^{*}:\, I_{\pi}[\varphi^{*}]=\inf_{\varphi\,\in\,\mathcal{B}\,\cap\,\mathcal{Q}_{s}\,[\, G_{\pi;i,j}^{\,0},\, T_{i+1,j}]}I_{\pi}\left[\varphi\right].
\end{equation}
By the above definition we have $I_{\pi}\left[\varphi^{*}\right]\leq I_{\pi}\left[\varphi^{-}\right]$
and since $\pi^{-}\left(\tau^{-1}\varphi_{\tau}^{-}\right)=\pi\left(\tau^{-1}\varphi_{\tau}^{-}\right)$
for $\tau\in[t'_{\epsilon,k+1},t{}_{\epsilon,k}]$ we can bound the
difference between $I_{\pi_{\epsilon}^{-}}\left[\varphi^{-}\right]$
and $I_{\pi}\left[\varphi^{-}\right]$ as
\begin{multline}
\left|\Delta'_{\epsilon}\right|:=|I_{\pi}\left[\varphi^{-}\right]-I_{\pi_{\epsilon}^{-}}\left[\varphi^{-}\right]|=\sum_{i\leq k\leq j}\int_{\tau\in(t'_{k+1},\, t'_{\epsilon,k+1})\cup(t{}_{\epsilon,k},\, t{}_{k})}d\tau\,\left|\log\bar{\pi}\left(\tau^{-1}\varphi_{\tau}\right)\right|\leq\\
\leq\left(j-i\right)\left(\left|t'_{k+1}-t'_{\epsilon,k+1}\right|+\left|t{}_{k+1}-t{}_{\epsilon,k+1}\right|\right)\delta_{\epsilon}.\label{eq:ghgh}
\end{multline}
 As $\pi\leq\pi_{\epsilon}^{+}$ by construction we can also conclude
that $I_{\pi_{\epsilon}^{+}}\left[\varphi^{*}\right]\leq I_{\pi}\left[\varphi^{*}\right]$,
while by definition of $\varphi^{+}$ as optimal trajectory for $I_{\pi_{\epsilon}^{+}}$
we can write $I_{\pi_{\epsilon}^{+}}\left[\varphi^{+}\right]\leq I_{\pi_{\epsilon}^{+}}\left[\varphi^{*}\right]$.
Collecting all those relations we obtain the following inequalities
\begin{equation}
I_{\pi_{\epsilon}^{+}}\left[\varphi^{+}\right]\leq I_{\pi}\left[\varphi^{*}\right]\leq I_{\pi_{\epsilon}^{-}}\left[\varphi^{-}\right]-\Delta'_{\epsilon},
\end{equation}
and by taking $\epsilon\rightarrow0$ we can finally write that 
\begin{equation}
\lim_{\epsilon\rightarrow0}I_{\pi_{\epsilon}^{+}}[\varphi^{+}]\leq I_{\pi}\left[\varphi^{*}\right]\leq\lim_{\epsilon\rightarrow0}I_{\pi_{\epsilon}^{-}}[\varphi^{-}],
\end{equation}
which, together with Eq. (\ref{eq:4.44}), proves Eq. (\ref{eq:hhhhhhhhhhhhhhhh}).
This completes our extension of Theorem \ref{Theorem 1} to the whole
set of urn function $\mathcal{U}$ in case we take as initial condition
$X_{n,1}$ uniformly distributed on $\left[0,1\right]$.
\end{proof}

\subsubsection{Initial conditions and time-inhomogeneous functions.}

\label{Section3.1.5}First we deal with the influence of initial conditions
on the large deviation properties of our urn process. Until now we
considered processes with initial condition $X_{n,1}$ uniformly distributed
on $\left[0,1\right]$, the following lemma shows that fixing $X_{n,m}$
for some $m>0$ will not affect the rate function if $\pi\in\left(0,1\right)$,
provided that $m$ is finite and $0\leq X_{n,m}\leq m$.
\begin{lem}
\label{Lemma 2}Let $X_{n}$ be a urn process with urn function $\pi\in\left(0,1\right)$
and initial conditions $0<X_{n,m}<m<\infty$. Then, the rate function
is independent from these initial conditions.\end{lem}
\begin{proof}
Let $\varphi\in\mathcal{Q}_{n}$, $x_{m,n}=m^{-1}X_{m,n}$ and $\epsilon_{n,\tau}$
as in Lemma \ref{Lemma 5}. If $\pi\in\left(0,1\right)$ then $\left\Vert \log\pi\right\Vert $
and $\left\Vert \log\bar{\pi}\right\Vert $ are bounded quantities
and we can use the estimates of Lemma \ref{Lemma 5} to obtain 
\begin{multline}
n^{-1}\left|\log\mathbb{P}\left(\chi_{n}=\varphi\,|\,\varphi_{m/n}=\left(m/n\right)x_{m,n}\right)-\log\mathbb{P}\left(\chi_{n}=\varphi\right)\right|\leq\\
\leq\int_{\tau\in\left[0,m/n\right]}d\varphi_{\tau}\,\left|\log\pi\left({\textstyle \left(\varphi_{\tau}+\epsilon_{n,\tau}\right)/\tau}\right)\right|+\int_{\tau\in\left[0,m/n\right]}d\tilde{\varphi}_{\tau}\,\left|\log\bar{\pi}\left(\left(\varphi_{\tau}+\epsilon_{n,\tau}\right)/\tau\right)\right|\leq\\
\leq\left(\left\Vert \log\pi\right\Vert +\left\Vert \log\bar{\pi}\right\Vert \right)m/n.\label{eq:initialcond}
\end{multline}
This difference vanishes as $n\rightarrow\infty$ for any $\varphi\in\mathcal{Q}_{n}$.
This obviously implies that the LDPs governing the two processes share
the same rate function. 
\end{proof}
Now consider $\pi\in\left[0,1\right]$. By applying the steps to extend
the proof of Theorem \ref{Theorem 1} we can easily convince that
the only influence on LDPs arising from fixing $\varphi_{m/n}=\left(m/n\right)x_{m,n}$
comes from the fact that some trajectories could be forbidden, since
by continuity of $\varphi_{\tau}$ a trajectory from $\varphi_{m/n}=\left(m/n\right)x_{m,n}$
to $\varphi_{1}=s$ may have to cross intervals where $\pi\left(\tau^{-1}\varphi_{\tau}\right)$
is $0$ or $1$ without having at the same time $d\varphi_{\tau}=0$
or $1$, which is a necessary condition to ensure that $S_{\pi}\left[\varphi\right]>-\infty$. 

As example, consider an urn function such that $\pi\left(s\right)=0$
for some $s\in\left[\sigma_{1}^{+},\sigma_{2}^{-}\right]$, $0<\sigma_{1}^{+}<\sigma_{2}^{-}<1$,
and $\pi\left(s\right)>0$ otherwise. As before, we can define the
intervals $G_{\pi;1}:=\left[0,\sigma_{1}^{+}\right)$, $\bar{G}_{\pi;1}^{\,0}:=\left[\sigma_{1}^{+},\sigma_{2}^{-}\right]$
and $G_{\pi;2}:=\left(\sigma_{2}^{-},1\right]$. Then, take $\left(m/n\right)^{-1}\varphi_{m/n}=x_{m,n}\in G_{\pi;1}$
for some $m<\infty$. Since any trajectory $\varphi$ that reach $G_{\pi;2}$
from $G_{\pi;1}$ would require that $\tau^{-1}\varphi_{\tau}$ crosses
$\bar{G}_{\pi;1}^{\,0}$ with some $d\varphi>0$, we conclude that
such trajectory will return $S_{\pi}\left[\varphi\right]=-\infty$.
Hence any allowed trajectory with $\left(m/n\right)^{-1}\varphi_{m/n}=x_{m,n}\in G_{\pi;1}$
would be confined in $G_{\pi;1}$, like a process with same initial
condition and a modified urn function $\pi^{*}\left(s\right)=\pi\left(s\right)$
for $s\in G_{\pi;1}$ and $\pi^{*}\left(s\right)=0$ otherwise.

In general, the allowed interval $\left[z_{-}^{*},z_{+}^{*}\right]$
of $\tau^{-1}\varphi_{\tau}$ for trajectories with $\left(m/n\right)^{-1}\varphi_{m/n}=x_{n,m}$
will run from the highest non isolated value of $s$ reachable from
$x_{n,m}$ and such that $\pi\left(s\right)=1$ to the lowest reachable
non isolated $s$ such that $\pi\left(s\right)=0$, since those points
acts as uncrossable walls for $\tau^{-1}\varphi_{\tau}$, while all
other values contained in $\left[z_{-}^{*},z_{+}^{*}\right]$ can
be crossed at least by trajectories of the type presented in the proof
of Theorem \ref{Theorem 1} above. 

Notice that in the above informal definition we specified that the
point must be non isolated, since isolated points may be eventually
crossed due to the discontinuous nature of the process at finite $n$.
To avoid this inconsistencies we define $Z_{\pi,x_{m,n}}^{*}$ as
the $\limsup$ of the subsets of $\left[0,1\right]$ that the process
$k^{-1}X_{n,k}$ is allowed to hit at time $k=n$ with positive probability
when we take $\mathbb{P}(x_{n,m}=m^{-1}X_{m})=1$ for some $m\leq n$,
$0\leq X_{m}\leq m$ and $n<\infty$. 
\begin{equation}
Z_{\pi,x_{m,n}}^{*}:=\limsup_{n\rightarrow\infty}\left\{ Z:\,\mathbb{P}\left(x_{n,n}\in Z\,|\, x_{n,m}=m^{-1}X_{m}\right)=1\right\} .\label{eq:Zdef}
\end{equation}
The above set is obviously an interval since, as said before, any
internal point can be reached by trajectories of the type described
in the proof of Theorem \ref{Theorem 1}. Hence, we can say that $Z_{\pi,x_{m,n}}^{*}:=\left[z_{-}^{*},z_{+}^{*}\right]$,
with $z_{-}^{*}$ and $z_{+}^{*}$ defined as in the statement of
Corollary \ref{cor:CorrollariINIT}.

That said, it is clear that computing a LDP for a process with initial
condition $x_{n,m}$ would be like computing it with initial condition
$X_{n,1}$ uniformly distributed on $\left[0,1\right]$ once we have
discarded from $\pi$ the forbidden zones. This can be done by considering
a modified $\pi^{*}$ with $\pi^{*}\left(s\right)=1$ in the forbidden
interval $s\in\left[0,z_{-}^{*}\right)$ on the left of $Z_{\pi,x_{m,n}}^{*}$
and $\pi^{*}\left(s\right)=0$ in $s\in\left(z_{+}^{*},1\right]$
on the right of $Z_{\pi,x_{m,n}}^{*}$,
\begin{equation}
\pi^{*}\left(s\right):=\mathbb{I}_{\{s\in\left[0,z_{-}^{*}\right)\}}+\pi\left(s\right)\mathbb{I}_{\{s\in\left[z_{-}^{*},z_{+}^{*}\right]\}},
\end{equation}
so that the probability mass initially distributed on $\left[0,1\right]$
gets pushed inside $Z_{\pi,x_{m,n}}^{*}$ in finite time, simulating
the initial condition at least for what concerns the LDPs computation.

It remains to prove Corollary \ref{cor:Coroll inhomog.} about time-inhomogeneous
functions. In this case we considered only the subclass $\pi\in\left(0,1\right)$,
for which the proof is straightforward
\begin{proof}
Let $\pi\in\mathcal{U}$ with $0<\pi<1$ and let $\pi_{n}\in\mathcal{U}$,
$\pi_{n}\in\left(0,1\right)$ such that $\left|\pi_{n}\left(s\right)-\pi\left(s\right)\right|\leq\delta_{n}$,
$\lim_{n}\delta_{n}=0$ for all $s\in\left[0,1\right]$. By lemma
\ref{Lemma 5} it suffices to show that $\left|S_{\pi_{n}}\left[\varphi\right]-S_{\pi}\left[\varphi\right]\right|\rightarrow0$
as $n\rightarrow0$. We can bound $\left|S_{\pi_{n}}\left[\varphi\right]-S_{\pi}\left[\varphi\right]\right|$
as follows
\begin{multline*}
\left|S_{\pi_{n}}\left[\varphi\right]-S_{\pi}\left[\varphi\right]\right|\leq\int_{\tau\in\left[0,1\right]}d\varphi_{\tau}\,\left|\log\pi_{n}\left(\varphi_{\tau}/\tau\right)-\log\pi\left(\varphi_{\tau}/\tau\right)\right|+\\
+\int_{\tau\in\left[0,1\right]}d\tilde{\varphi}_{\tau}\,\left|\log\bar{\pi}_{n}\left(\varphi_{\tau}/\tau\right)-\log\bar{\pi}\left(\varphi_{\tau}/\tau\right)\right|\leq\\
\leq\int_{\tau\in\left[0,1\right]}d\varphi_{\tau}\,\delta_{n}/\left|\pi\left(\varphi_{\tau}/\tau\right)\right|+\int_{\tau\in\left[0,1\right]}d\tilde{\varphi}_{\tau}\,\delta_{n}/\left|\bar{\pi}\left(\varphi_{\tau}/\tau\right)\right|\leq\\
\leq\left[1/\left(1-\left\Vert \bar{\pi}\right\Vert \right)+1/\left(1-\left\Vert \pi\right\Vert \right)\right]\delta_{n}.
\end{multline*}
Since for $0<\pi<1$ we have $\left\Vert \bar{\pi}\right\Vert <\infty$,
$\left\Vert \pi\right\Vert <\infty$, the above bound vanishes as
$\delta_{n}\rightarrow0$ and the proof is completed.
\end{proof}

\subsection{Entropy of the event $X_{n,n}=\left\lfloor sn\right\rfloor $.}

\label{Section3.2}In this section we use the variational representation
of Sample-Path LDPs to show Theorem \ref{Theorem 2} and Corollaries
\ref{Corollary 3}, \ref{Corollary4}, \ref{Corollary4.1}. Since
the event $\left\{ X_{n,n}=\left\lfloor sn\right\rfloor \right\} $
is slightly finer than those usually considered in large deviations
theory, its analysis requires some additional estimates. Moreover,
note that $\mathcal{Q}_{s}$ is not an $I_{\pi}-$continuity set because
of the fixed endpoint condition $\varphi_{1}=s$, which implies $\mathrm{cl}\left(\mathcal{Q}_{s}\right)=\varnothing$.
We circumvent this problem as follows
\begin{lem}
\label{lemma9}Let $s\in\left[0,1\right]$, $\delta>0$ and define
$\mathcal{Q}_{s,\delta}:={\textstyle \bigcup_{\, u-s\in\left[0,\delta\right]}}\,\mathcal{Q}_{u},$
where $\mathcal{Q}_{s}:=\left\{ \varphi\in\mathcal{Q}:\,\varphi_{1}=s\right\} $,
then 
\begin{equation}
\lim_{n\rightarrow\infty}n^{-1}\log\mathbb{P}\left(\left\lfloor sn\right\rfloor \leq X_{n,n}\leq\left\lfloor \left(s+\delta\right)n\right\rfloor \right)=-\inf_{\varphi\in\mathcal{Q}_{s,\delta}}I_{\pi}\left[\varphi\right].
\end{equation}
\end{lem}
\begin{proof}
Since $\mathcal{Q}_{s,\delta}:={\textstyle \bigcup_{\, u-s\in\left[0,\delta\right]}}\,\mathcal{Q}_{u}$
is an $I_{\pi}-$continuity set when $s\in\left[0,1\right]$ and $\delta>0$,
by Theorem \ref{Theorem 1} we have 
\begin{equation}
\lim_{n\rightarrow\infty}n^{-1}\log\mathbb{P}\left(\chi_{n}\in\mathcal{Q}_{s,\delta}\right)=-\inf_{\varphi\in\mathcal{Q}_{s,\delta}}I_{\pi}\left[\varphi\right].
\end{equation}
Then, let $0<\nu<\delta$ so that we can write
\begin{multline}
-\inf_{\varphi\in\mathcal{Q}_{s,\delta-\nu}}I_{\pi}\left[\varphi\right]=\lim_{n\rightarrow\infty}n^{-1}\log\mathbb{P}\left(\chi_{n}\in\mathcal{Q}_{s,\delta-\nu}\right)\leq\\
\leq\lim_{n\rightarrow\infty}n^{-1}\log\mathbb{P}\left(\left\lfloor sn\right\rfloor \leq X_{n,n}\leq\left\lfloor \left(s+\delta\right)n\right\rfloor \right)\leq\\
\leq\lim_{n\rightarrow\infty}n^{-1}\log\mathbb{P}\left(\chi_{n}\in\mathcal{Q}_{s,\delta+\nu}\right)=-\inf_{\varphi\in\mathcal{Q}_{s,\delta+\nu}}I_{\pi}\left[\varphi\right].\label{eq:-4}
\end{multline}
Since $I_{\pi}$ is continuous on $\left(\mathcal{Q},\left\Vert \cdot\right\Vert \right)$
and $\mathcal{Q}_{s,\delta'}\subset\mathcal{Q}_{s,\delta}\subset\mathcal{Q}$
for every $\delta'<\delta$, we can take the limit $\nu\rightarrow0$
and the proof is completed.
\end{proof}

\subsubsection{Proof of Theorem \ref{Theorem 2}.}

\label{Section3.2.1}Before starting, we remind some notation. Let
$\varphi:=\left\{ \varphi_{\tau}:\tau\in\left[0,1\right]\right\} $
and let $Y_{n,k}\left(\varphi\right):=n\varphi_{k}$, $\delta Y_{n,k}\left(\varphi\right):=n\left(\varphi_{k+1}-\varphi_{k}\right)$
as in Eq. (\ref{eq:2.10}). We also define the set of trajectories
\begin{equation}
\mathcal{Q}_{n,k}:=\left\{ \varphi\in\mathcal{Q}_{n}:\, Y_{n,n}\left(\varphi\right)=k\right\} ,
\end{equation}
where $\mathcal{Q}_{n}$ is the support of $\chi_{n}$ as defined
in Eq. (\ref{eq:2.9-1}). As for Theorem \ref{Theorem 1} we first
prove the result for $\pi\in\left(0,1\right)$
\begin{proof}
Let $\pi\in\left(0,1\right)$. We start from the variational representation
of $\mathbb{P}\left(\chi_{n}=\varphi\right)$ in Eq. (\ref{eq:2.1-1}):
by Lemma \ref{Lemma 5} we can rewrite $\mathbb{P}\left(X_{n,n}=k\right)$
as 
\begin{equation}
\mathbb{P}\left(X_{n,n}=k\right)=\sum_{\varphi\in\mathcal{Q}_{n,k}}\mathbb{P}\left(\chi_{n}=\varphi\right)=\sum_{\varphi\in\mathcal{Q}_{n,k}}e^{nS_{\pi}\left[\varphi\right]+O\left(n\cdot W_{\pi}\left(1/n\right)\right)}.\label{eq:3.5}
\end{equation}
First, we observe that the following inequalities holds: 
\begin{equation}
\mathbb{P}\left(X_{n,n}=k\right)\leq\mathbb{P}\left(k\leq X_{n,n}\leq k'\right)\leq\left(k'-k\right)\sup_{k\leq i\leq k'}\mathbb{P}\left(X_{n,n}=i\right):
\end{equation}
by defining $k^{*}:\,\mathbb{P}\left(X_{n,n}=k^{*}\right)=\sup_{k\leq i\leq k'}\mathbb{P}\left(X_{n,n}=i\right)$
we can rewrite them as 
\begin{equation}
\left|\log\mathbb{P}\left(k\leq X_{n,n}\leq k'\right)-\log\mathbb{P}\left(X_{n,n}=k\right)\right|\leq\log\left(k'-k\right)+\left|\log\mathbb{P}\left(X_{n,n}=k^{*}\right)-\log\mathbb{P}\left(X_{n,n}=k\right)\right|.\label{eq:-9}
\end{equation}
Let $T^{0}\left(\varphi\right):=\left\{ i\in\mathbb{N}\,:\,\delta Y_{n,i}\left(\varphi\right)=0\right\} $,
$T^{1}\left(\varphi\right):=\left\{ i\in\mathbb{N}\,:\,\delta Y_{n,i}\left(\varphi\right)=1\right\} $
and define the operator $\hat{u}_{h}$ such that $\hat{u}_{h}\varphi:=\left\{ (\hat{u}_{h}\varphi)_{\tau}:\,\tau\in\left[0,1\right]\right\} $,
\begin{equation}
(\hat{u}_{h}\varphi)_{\tau}:=\varphi_{\tau}+\left(\tau-{\textstyle \frac{1}{n}}\left\lfloor n\tau\right\rfloor \right)\mathbb{I}_{\left\{ n\tau\in\left[h-1,h\right]\right\} }+{\textstyle \frac{1}{n}}\mathbb{I}_{\left\{ n\tau\in\left[h,n\right]\right\} }.
\end{equation}
If we apply $m$ times this operator to $\varphi\in\mathcal{Q}_{n,k}$
with a suitable sequence of $h_{i}$, $1\leq i\leq m$ we can get
a $\hat{u}_{h_{m}}\,...\,\hat{u}_{h_{1}}\varphi\in\mathcal{Q}_{n,k+m}$.
By simple combinatorial arguments it's easy to convince that the following
relation holds
\begin{multline}
\sum_{\varphi\in\mathcal{Q}_{n,k+m}}e^{nS_{\pi}\left[\varphi\right]}=\prod_{j=1}^{m}\left(k+j\right)^{-1}\sum_{\varphi\in\mathcal{Q}_{n,k}}\ \ \,\ \sum_{h_{1}\in T^{0}\left(\varphi\right)}\ \sum_{h_{2}\in T^{0}\left(\hat{u}_{h_{1}}\varphi\right)}...\\
...\sum_{h_{m-1}\in T^{0}\left(\hat{u}_{h_{m-2}}\,...\,\hat{u}_{h_{1}}\varphi\right)}\ \sum_{h_{m}\in T^{0}\left(\hat{u}_{h_{m-1}}\,...\,\hat{u}_{h_{1}}\varphi\right)}\, e^{nS_{\pi}\left[\hat{u}_{h_{m}}\,...\,\hat{u}_{h_{1}}\varphi\right]};\label{eq:3.12}
\end{multline}
the product comes from noticing that $\left|T^{1}\left(\varphi\right)\right|=k+j$
when $\varphi\in\mathcal{Q}_{n,k+j}$: it corrects for the exceeding
copies of the same path which arise from summing over the $T^{0}\left(\,...\,\hat{u}_{h_{2}}\hat{u}_{h_{1}}\varphi\right)$
sets. Now, since by definition $\left\Vert \hat{u}_{h_{m}}\,...\,\hat{u}_{h_{1}}\varphi-\varphi\right\Vert =m/n$,
from Lemma \ref{Lemma 5-1} we have
\begin{equation}
n\left|S_{\pi}\left[\hat{u}_{h_{m}}\,...\,\hat{u}_{h_{1}}\varphi\right]-S_{\pi}\left[\varphi\right]\right|\leq n\, W_{\pi}\left(m/n\right),\label{eq:-2-1}
\end{equation}
and, given that$\left|T^{0}\left(\hat{u}_{h_{i}}\,...\,\hat{u}_{h_{1}}\varphi\right)\right|=n-k+i-1$
when $\varphi_{n}\in\mathcal{Q}_{n,k}$, from Eq.s. (\ref{eq:3.5}),
(\ref{eq:3.12}) and (\ref{eq:-2-1}) we can conclude that
\begin{multline}
\left|\log\mathbb{P}\left(X_{n,n}=k+m\right)-\log\mathbb{P}\left(X_{n,n}=k\right)\right|\leq\\
\leq\left|{\textstyle \sum_{i=1}^{m}}\log\left(\left(n-k+i-1\right)/\left(k+i\right)\right)\right|+n\, W_{\pi}\left(m/n\right)+O\left(W_{\pi}\left(1/n\right)\right).\label{eq:-3-1}
\end{multline}
Then, we can put together Eq.s. (\ref{eq:3.5}), (\ref{eq:-9}), (\ref{eq:-3-1})
and the inequality $k\leq k^{*}\leq k'$ to get the bound 
\begin{multline}
\left|\log\mathbb{P}\left(k\leq X_{n,n}\leq k'\right)-\log\mathbb{P}\left(X_{n,n}=k\right)\right|\leq\\
\leq|{\textstyle \sum_{i=1}^{k'-k}}\log\left(\left(n-k+i-1\right)/\left(k+i\right)\right)|+n\, W_{\pi}\left(m/n\right)+\log\left(k'-k\right)+O\left(W_{\pi}\left(1/n\right)\right),\label{eq:-3-1-1}
\end{multline}
By taking $k=\left\lfloor sn\right\rfloor $, $k'=\left\lfloor \left(s+\delta\right)n\right\rfloor $,
then the limit $n\rightarrow\infty$, we find that the sum in the
above inequality has the following limiting behavior 
\begin{multline}
\lim_{n\rightarrow\infty}n^{-1}{\textstyle \sum_{i=1}^{k'-k}}\,{\textstyle \log\left(\frac{n-k+i-1}{k+i}\right)}=\int_{u\in\left[0,\delta\right]}du\,\log\left(\left(\bar{s}+u\right)/\left(s+u\right)\right)=\\
=H_{1}\left(s+\delta\right)-H_{1}\left(s\right)-H_{1}\left(\bar{s}+\delta\right)+H_{1}\left(\bar{s}\right)=:H_{2}\left(s,\delta\right),\label{eq:}
\end{multline}
where $H_{1}\left(s\right)=s-s\log s$. Then, applying Lemma \ref{lemma9}
and the above relation to Eq. (\ref{eq:-3-1-1}) we finally obtain
the bound 
\begin{equation}
\left|\phi\left(s\right)+{\textstyle \inf_{\varphi\in\mathcal{Q}_{s,\delta}}}I_{\pi}\left[\varphi\right]\right|\leq\left|H_{2}\left(s,\delta\right)\right|+W_{\pi}\left(\delta\right)\label{eq:-5}
\end{equation}
In the end, since $I_{\pi}$ is continuous on $\left(\mathcal{Q},\left\Vert \cdot\right\Vert \right)$
and $\mathcal{Q}_{s}\subset\mathcal{Q}_{s,\delta}\subset\mathcal{Q}$,
taking $\delta\rightarrow0$ in the above equation will complete our
proof. Notice that our bound diverges for $s\in\left\{ 0,1\right\} $,
but in such cases the theorem's statement is trivially verified by
a direct computation, hence we can assume $s\in\left(0,1\right)$.

The extension to the case $\pi\in\left[0,1\right]$ can be performed
by proving the above result for $\pi_{\epsilon}^{+}$ and $\pi_{\epsilon}^{-}$
for each subset $\mathcal{Q}_{k}$, $1\leq k\leq N^{*}$ and then
take $\epsilon\rightarrow0$ as in the proof of Theorem \ref{Theorem 1}.
As example, for $\pi_{\epsilon}^{+}$ and $s\in G_{\pi;i}$ we can
consider 
\begin{equation}
\mathcal{Q}_{s,\delta}\left[G_{\pi;i,j}^{\,0},\, T_{i+1,j}\right]:=\bigcup_{\, u-s\in\left[0,\delta\right]}\mathcal{Q}_{u}\left[G_{\pi;i,j}^{\,0},\, T_{i+1,j}\right]
\end{equation}
in place of $\mathcal{Q}_{s,\delta}$, then $\mathcal{Q}_{n,k}\cap\mathcal{Q}_{s,\delta}[G_{\pi;i,j}^{\,0},\, T_{i+1,j}]$
in place of $\mathcal{Q}_{n,k}$ and proceed as for $\pi\in\left(0,1\right)$
case. We do the same for $\pi_{\epsilon}^{-}$, with $\sigma_{i}^{+}-\epsilon$,
$\sigma_{i}^{-}+\epsilon$ in place of $\sigma_{i}^{+}$, $\sigma_{i}^{-}$
and finally use the argument at the end of the proof of Theorem \ref{Theorem 1}
to take the limit $\epsilon\rightarrow0$. The procedure described
above is quite mechanical and does not require any conceptual addition.
Then, we avoid to explicitly repeat the computations of Theorem \ref{Theorem 1},
which would result in a heavy (and messy) notation surely much less
explicative than the above statements.
\end{proof}

\subsubsection{Proof of Corollaries \ref{Corollary 3}, \ref{Corollary4} and \ref{Corollary4.1}.}

\label{Section3.2.2}Before dealing with Corollaries \ref{Corollary 3},
\ref{Corollary4} and \ref{Corollary4.1} we still need an additional
result. We start by finding conditions on $\varphi$ such $I_{\pi}\left[\varphi\right]=0$.
From Theorem \ref{Theorem 2} we found that $\phi\left(s\right)=-{\textstyle \inf_{\varphi\in\mathcal{Q}_{s}}}I_{\pi}\left[\varphi\right]$,
and since $I_{\pi}\left[\varphi\right]\geq0$ our thesis would follow
if we can find a trajectory $\varphi\in\mathcal{Q}_{s}\cap\mathcal{AC}$
such that $I_{\pi}\left[\varphi\right]=0$. The following lemma provides
the desired condition on $\varphi$
\begin{lem}
\label{lemma 12}Let $\varphi^{*}:=\left\{ \varphi_{\tau}^{*}:\tau\in\left[0,1\right]\right\} $
such that $I_{\pi}\left[\varphi^{*}\right]=0$. Then, any of such
$\varphi^{*}$ must satisfy the homogeneous differential equation
$\dot{\varphi}_{\tau}^{*}=\pi\left(\varphi_{\tau}^{*}/\tau\right)$
with $\varphi^{*}\in\mathcal{Q}\cap\mathcal{AC}$ .\end{lem}
\begin{proof}
Let $\left(x,y\right)\in\left[0,1\right]^{2}$ and $\bar{x}=1-x$,
$\bar{y}=1-y$ as usual. Then, define the function $L:\left[0,1\right]^{2}\rightarrow\left(-\infty,0\right]$
as follows:
\begin{equation}
L\left(x,y\right):=x\log\left(y/x\right)+\bar{x}\log\left(\bar{y}/\bar{x}\right).
\end{equation}
Since by Theorem \ref{Theorem 2} and Lemma \ref{Lemma 3} we have
$I_{\pi}\left[\varphi\right]=\infty$ when $\varphi\notin\mathcal{AC}$,
we can restrict the search for minimizing strategies to the set $\mathcal{Q}\cap\mathcal{AC}$,
for which $\dot{\varphi}$ exists almost everywhere. Then, for every
$\varphi\in\mathcal{Q}\cap\mathcal{AC}$ we can write $I_{\pi}\left[\varphi\right]$
as 
\begin{equation}
I_{\pi}\left[\varphi\right]=-\int_{\tau\in\left[0,1\right]}d\tau\, L\left(\dot{\varphi}_{\tau},\pi\left(\varphi_{\tau}/\tau\right)\right).
\end{equation}
$L$ is a negative concave function for every pair $\left(x,y\right)\in\left[0,1\right]^{2}$,
with $L\left(x,y\right)=0$ if and only if $x=y$. Hence, any choice
of $\varphi$ for which $I_{\pi}\left[\varphi\right]=0$ must satisfy
the condition $\dot{\varphi}_{\tau}=\pi\left(\varphi_{\tau}/\tau\right)$
for every $\tau\in\left[0,1\right]$. 
\end{proof}
We can now prove the corollaries of Theorem \ref{Theorem 2} concerning
optimal trajectories. Since Corollary \ref{Corollary 3} is an almost
obvious consequence of \ref{Corollary4} and \ref{Corollary4.1},
we first concentrate on the last two, and prove Corollary \ref{Corollary 3}
in the end of this subsection. 
\begin{proof}
Lemma \ref{lemma 12} states that every trajectory for which $I_{\pi}\left[\varphi^{*}\right]=0$
is in $\mathcal{AC}$ and must satisfy the homogeneous differential
equation $\dot{\varphi}_{\tau}^{*}=\pi\left(\varphi_{\tau}^{*}/\tau\right)$.
Then our zero-cost trajectory, if existent, must be a solution to
the homogeneous Cauchy Problem
\begin{equation}
\dot{\varphi}_{\tau}^{*}=\pi\left(\varphi_{\tau}^{*}/\tau\right),\,\varphi_{1}^{*}=s.\label{eq:3.17}
\end{equation}
To characterize the solution we first define $u^{*}:\left[0,1\right]\rightarrow\left[0,1\right]$
as
\begin{equation}
u^{*}:=\left\{ u_{\tau}^{*},\tau\in\left[0,1\right]:u_{\tau}^{*}=\varphi_{\tau}^{*}/\tau\right\} ,
\end{equation}
such that we can rewrite the Cauchy problem (\ref{eq:3.17}) as 
\begin{equation}
\dot{u}_{\tau}^{*}={\textstyle \frac{1}{\tau}}\left[\pi\left(u_{\tau}^{*}\right)-u_{\tau}^{*}\right],\, u_{1}^{*}=s.\label{eq:4.41}
\end{equation}
If $a_{\pi,i}=0$ then $\pi\left(s\right)-s=0$ for $s\in K_{\pi,i}$,
and the solution is trivially $u^{*}=s$, then we concentrate on $a_{\pi,i}\neq0$.
We recall that for $a_{\pi,i}\neq0$ the boundary $\partial K_{\pi,i}$
of $K_{\pi,i}$ is a set of two isolated points. Then, let $\partial K_{\pi,i}=\{s_{i}^{*},s_{i}^{\dagger}\}$
with
\begin{equation}
s_{i}^{*}:=\mathbb{I}_{\left\{ a_{\pi,i}=1\right\} }\inf K_{\pi,i}+\mathbb{I}_{\left\{ a_{\pi,i}=-1\right\} }\sup K_{\pi,i},
\end{equation}
\begin{equation}
s_{i}^{\dagger}:=\mathbb{I}_{\left\{ a_{\pi,i}=-1\right\} }\inf K_{\pi,i}+\mathbb{I}_{\left\{ a_{\pi,i}=1\right\} }\sup K_{\pi,i},
\end{equation}
such that $\pi\left(s\right)-s$, $s\in K_{\pi,i}$ is always decreasing
in the neighborhood of $s_{i}^{*}$and increasing in that of $s_{i}^{\dagger}$
at least if $1\leq i\leq N-1$.

First, we notice that both constant trajectories $u_{\tau}^{*}=s_{i}^{\dagger}$
and $u_{\tau}^{*}=s_{i}^{*}$ satisfy the Cauchy problem in Eq. (\ref{eq:4.41}).
To simplify the exposition, we consider $a_{\pi,i}=-1$, such that
$s_{i}^{\dagger}<s_{i}^{^{*}}$ and, by Eq. (\ref{eq:4.41}), $u_{\tau}^{*}$
must be a decreasing function of $\tau\in\left[0,1\right]$ with $u_{\tau}^{*}\in\left[u_{1}^{*},\, u_{0}^{*}\right]\subseteq K_{\pi,i}\cup\partial K_{\pi,i}$. 

Given that, we have only two possible kinds of optimal trajectory
$u_{\tau}^{*}$ for the variational problem with $s\in K_{\pi,i}\cup\partial K_{\pi,i}$
. The first is that $u_{\tau}^{*}$ decreases from some $u_{0}^{*}<s_{i}^{*}$
to $u_{1}^{*}=s$, while the second is such that $u_{\tau}^{*}=s_{i}^{*}$
constant from $\tau=0$ to some $\tau_{s,i}^{*}\in\left[0,1\right)$,
and then it decreases from $s_{i}^{*}$to eventually reach $s$ at
$\tau=1$. Then, define 
\begin{equation}
F_{\pi}\left(s,u\right):=\int_{u}^{s}\frac{dz}{\pi\left(z\right)-z}
\end{equation}
for some $s\in K_{\pi,i}$, so that the solution to the Cauchy problem
can be written in implicit form as $F_{\pi}\left(s,u_{\tau}^{*}\right)=-\log\left(\tau\right)$.
We can easily see that $\tau\left(u\right)=e^{-F_{\pi}\left(s,u\right)}$
is a decreasing function with $\tau\left(u\right)=0$ only if $F_{\pi}\left(s,u\right)=\infty$.
Since by definition $F_{\pi}\left(s,u\right)$ can diverge only for
$u\rightarrow s_{i}^{*}$ we conclude that only trajectories of the
second kind, with $u_{\tau}^{*}=s_{i}^{*}$ until some $\tau_{s,i}^{*}\in\left[0,1\right)$,
can meet our requirements for being optimal. Moreover, we can compute
$\tau_{s,i}^{*}$ by integrating backward in time the solution from
$\tau=1$. We find that 
\begin{equation}
\tau_{s,i}^{*}:=\exp(-{\textstyle \lim_{\, a_{\pi,i}\left(u-s_{i}^{*}\right)\rightarrow0^{+}}}\left|F_{\pi}\left(s,u\right)\right|),
\end{equation}
where the above expression holds for both $a_{\pi,i}=1$ and $a_{\pi,i}=-1$.
Define the inverse function $F_{\pi,s}^{-1}:(\tau_{s,i}^{*},1]\rightarrow\left(s,s_{i}\right]$
of $\pi$ on $\left(s,s_{i}\right]$: 
\begin{equation}
F_{\pi,s}^{-1}:=\left\{ F_{\pi,s}^{-1}\left(q\right),\, q\in\left[0,\log\left(1/\tau_{s,i}^{*}\right)\right):\, F_{\pi,s}\left(F_{\pi,s}^{-1}\left(q\right)\right)=q\right\} 
\end{equation}
Then we can write the global solution to our Cauchy problem as 
\begin{equation}
u_{\tau}^{*}:=F_{\pi,s}^{-1}\left(\log\left(1/\tau\right)\right)\,\mathbb{I}_{\{\tau\in(\tau_{s,i}^{*},1]\}}+s_{i}^{*}\,\mathbb{I}_{\{\tau\in[0,\tau_{s,i}^{*}]\}},
\end{equation}
The same reasoning can be obviously applied to the case $a_{\pi,i}=1$,
with $\dot{u}_{\tau}^{*}>0$ and $u_{\tau}^{*}$ increasing in $\tau$.
We remark that the homogeneity of the above solution depends critically
on the integrability of $1/\left|\pi\left(u\right)-u\right|$ when
$\left|u-s_{i}^{*}\right|\rightarrow0$: if $\lim_{a_{\pi,i}\left(u-s_{i}^{*}\right)\rightarrow0^{-}}\left|F_{\pi}\left(s,u\right)\right|=\infty$,
then obviously $\tau_{s,i}^{*}=0$, while $0<\tau_{s,i}^{*}<1$ otherwise.

A similar reasoning can be applied to the case $u_{1}^{*}=s_{i}^{\dagger}$.
Let us again consider $u_{\tau}^{*}\in K_{\pi,i}\cup\partial K_{\pi,i}$,
$a_{\pi,i}=-1$ and take $s=s_{i}^{\dagger}$ in Eq. (\ref{eq:4.41}).
Here the picture is slightly more complex, since it also depends on
the behavior of $\left|F_{\pi}\left(s,u\right)\right|$, $s<u$, as
$s-s_{i}^{\dagger}\rightarrow0^{+}$.

In general, if $\left|F_{\pi}\left(s,u\right)\right|$, $s<u$, diverges
as $s-s_{i}^{\dagger}\rightarrow0^{+}$ then it is clear that the
only possible trajectory $u_{\tau}^{*}\in K_{\pi,i}\cup\partial K_{\pi,i}$
that ends in $s_{i}^{\dagger}$ is $u_{\tau}^{*}=s_{i}^{\dagger}$.
Anyway, if $\left|F_{\pi}\left(s,u\right)\right|$ remains finite
then we can have optimal trajectories that hit $s_{i}^{\dagger}$
at some time $\tau=t<1$ and stay in $s_{i}^{\dagger}$ for the remaining
$\tau\in\left[t,1\right]$. This is equivalent to set $u_{t}^{*}=s_{i}^{\dagger}$
as boundary condition of the Cauchy Problem in Eq. (\ref{eq:4.41}),
so that the implicit expression of the optimal trajectory is $F_{\pi}\left(s_{i}^{\dagger},u_{\tau}^{*}\right)=\log\left(t\right)-\log\left(\tau\right)$,
where $t\in\left[0,1\right]$ is free parameter. Since the above expression
is simply a shifted version of that for $u_{1}^{*}\in K_{\pi,i}$,
with $s_{i}^{\dagger}$ on place of of $s$, $t/\tau$ on place of
of $\tau$ and $\theta_{i}^{*}t$, 
\begin{equation}
\theta_{i}^{*}:=\exp\left(-{\textstyle \lim_{\, a_{\pi,i}\left(u-s_{i}^{*}\right)\rightarrow0^{+}}\lim_{\, a_{\pi,i}\left(s_{i}^{\dagger}-s\right)\rightarrow0^{+}}\left|F_{\pi}\left(s,u\right)\right|}\right),
\end{equation}
on place of of $\tau_{s,i}^{*}$, we can proceed as in the case $u_{1}^{*}\in K_{\pi,i}$
to find that 
\begin{equation}
u_{\tau}^{*}:=s_{i}^{\dagger}\mathbb{I}_{\left\{ \tau\in\left(t,1\right]\right\} }+\, F_{\pi,s}^{-1}\left(\log\left(t/\tau\right)\right)\,\mathbb{I}_{\left\{ \tau\in\left(\theta_{i}^{*}t,t\right]\right\} }+s_{i}^{*}\,\mathbb{I}_{\left\{ \tau\in\left[0,\theta_{i}^{*}t\right]\right\} }.
\end{equation}
It only remains to show that there is no solution to the Cauchy Problem
in Eq. (\ref{eq:4.41}) for boundary conditions $u_{1}^{*}\in K_{\pi,0}\cup K_{\pi,N}$.
Let consider $K_{\pi,0}$, for which always we have $a_{\pi,0}=1$
(the same result for $K_{\pi,N}$ can be obtained by a similar reasoning).
Since if $K_{\pi,0}\neq\varnothing$, then $\pi\left(0\right)>0$
and in this case $s_{0}^{\dagger}=0$ is not a zero-cost trajectory.
Then, $u_{\tau}^{*}$ should increase from some $u_{0}^{*}<u_{1}^{*}$
to some $u_{1}^{*}<s_{0}^{*}$, but the general form of the Cauchy
Problem in Eq. (\ref{eq:4.41}) rules out this possibility. We conclude
that no trajectory $\varphi_{\tau}^{*}=\tau u_{\tau}^{*}$, $u_{1}^{*}\in K_{\pi,0}$
such that $I_{\pi}\left[\varphi^{*}\right]=0$ exists, and by Lemma
\ref{lemma 12} this implies that $I_{\pi}\left[\varphi\right]>0$
for every $\varphi_{\tau}=\tau u_{\tau}$ with $u_{1}\in K_{\pi,0}$
as stated in Corollary \ref{Corollary 3}.
\end{proof}

\subsection{Cumulant Generating Function.}

\label{Section 3.3}In this section we use conditional expectations
and Picard-Lindelof theorem to prove a non-linear Cauchy problem for
$\psi\left(\lambda\right)$. Since the arguments are quite standard,
we won't indulge in details except this is necessary. Then, let define
the CGF up to time $n$ 
\begin{equation}
\psi_{n}\left(\lambda\right):=n^{-1}\log\mathbb{E}\left(e^{\lambda X_{n,n}}\right),\ \lambda\in\left(-\infty,\infty\right),\label{eq:1.18-1}
\end{equation}
so that $\psi\left(\lambda\right):=\lim_{n}\psi_{n}\left(\lambda\right)$.
Hereafter we denote by $\mathbb{P}_{\lambda}$ the tilted measure
\begin{equation}
\mathbb{P}_{\lambda}\left(X_{n,n}=X\right):=\exp\left[\lambda X-n\psi_{n}\left(\lambda\right)\right]\mathbb{P}\left(X_{n,n}=X\right)
\end{equation}
and by $\mathbb{E}_{\lambda}$ the tilted expectation. First we prove
some trivial properties for $\psi_{n}\left(\lambda\right)$.
\begin{lem}
Let $\psi_{n}\left(\lambda\right)$ in Eq. (\ref{eq:1.18-1}), and
define
\begin{equation}
\gamma_{n}\left(\lambda\right):=\left(n+1\right)\left[\psi_{n+1}\left(\lambda\right)-\psi_{n}\left(\lambda\right)\right],
\end{equation}
 then $\left|\psi_{n}\left(\lambda\right)\right|\leq\lambda$ , $\partial_{\lambda}\psi_{n}\left(\lambda\right)\in\left[0,1\right]$
and \textup{$\left|\gamma_{n}\left(\lambda\right)\right|\leq2\left|\lambda\right|$}
for all $n\in\mathbb{N}$, $\lambda\in\mathbb{R}$.\end{lem}
\begin{proof}
That $\left|\psi_{n}\left(\lambda\right)\right|\leq\lambda$ follows
directly from definitions: since $0\leq X_{n,n}\leq n$, then obviously
$n^{-1}\left|\log\mathbb{E}\left(e^{\lambda X_{n,n}}\right)\right|\leq\left|\lambda\right|$.
Similarly, from $\partial_{\lambda}\psi_{n}\left(\lambda\right)=n^{-1}\mathbb{E}_{\lambda}\left(X_{n,n}\right)$
follows $\partial_{\lambda}\psi_{n}\left(\lambda\right)\in\left[0,1\right]$.
We shall now find a recursive relation for the Moment Generating Function
$\mathbb{E}\left(e^{\lambda X_{n,n}}\right)$. Consider the conditional
expectation $\mathbb{E}\left(e^{\lambda X_{n+1,n+1}}|\mathcal{F}_{n}\right)$:
from Eq. (\ref{eq:1.2}) it's quite easy to check the Moment Generating
Function obeys the following recursion rule:
\begin{equation}
\mathbb{E}\left(e^{\lambda X_{n,k+1}}\right)-\mathbb{E}\left(e^{\lambda X_{n,k}}\right)=\left(e^{\lambda}-1\right)\mathbb{E}\left[\pi\left(x_{n,k}\right)e^{\lambda X_{n,k}}\right].\label{eq:3.3}
\end{equation}
After few manipulations we can write the above relation as
\begin{equation}
\gamma_{n}\left(\lambda\right)=-\psi_{n}\left(\lambda\right)+\log\left\{ 1+\left(e^{\lambda}-1\right)\mathbb{E}_{\lambda}\left[\pi\left(x_{n,n}\right)\right]\right\} ,
\end{equation}
Since by definition $\pi\left(x\right)\in\left[0,1\right],$ then
$\mathbb{E}\left[\pi\left(x_{n,n}\right)e^{\lambda X_{n,n}}\right]\leq\mathbb{E}\left(e^{\lambda X_{n,n}}\right)$
and $\mathbb{E}_{\lambda}\left[\pi\left(x_{n,n}\right)\right]\in\left[0,1\right]$,
so that $\left|\gamma_{n}\left(\lambda\right)\right|$ can be bounded
as 
\begin{equation}
\left|\gamma_{n}\left(\lambda\right)\right|\leq\left|\psi_{n}\left(\lambda\right)\right|+\left|\log\left(1+\left|e^{\lambda}-1\right|\right)\right|\leq2\left|\lambda\right|,\label{eq:1.16}
\end{equation}
which completes the proof.
\end{proof}
From last relation we found that $\lim_{n}\left|\psi_{n+1}\left(\lambda\right)-\psi_{n}\left(\lambda\right)\right|=0$,
but this is not enough to state whether $\lim_{n}\gamma_{n}\left(\lambda\right)=0$
for every $\lambda\in\mathbb{R}$. Before presenting our proof we
still need the following lemma
\begin{lem}
\label{Lemma14}Let $\left\{ f_{n},\, n\in\mathbb{N}\right\} $ be
a bounded real sequence. Then $g_{n}:=\left(n+1\right)\left(f_{n+1}-f_{n}\right)$
either converges to $0$ or does not converge.\end{lem}
\begin{proof}
Let suppose that $g_{n}$ converges to some $g>0$. Then $h>0$ and
$\epsilon>0$ exist such that $0<\epsilon\leq g_{n}$ for $n\geq h$.
Follows that $f_{n}\geq\epsilon\sum_{k=h}^{n-1}\left(k+1\right)^{-1}+f_{h}$
would diverge for $n\rightarrow\infty$, which contradicts that $f_{n}$
is bounded. A similar reasoning taking $g<0$ will lead to the conclusion
that $g$ can be neither strictly positive nor strictly negative,
hence we must have $g=0$.
\end{proof}

\subsubsection{Proof of Theorem \ref{Theorem 4}. }

\label{Section3.3.1}Before starting we remark that even if the the
statement of Theorem \ref{Theorem 4} asks for some additional properties
for $\pi\in\mathcal{U}$, the first part of this proof, devoted to
obtain the implicit ODE (\ref{eq:1.KKK}), does not.
\begin{proof}
Lemma \ref{Lemma14} implies that if both $\lim_{n}\psi\left(\lambda\right)$
and $\lim_{n}\mathbb{E}_{\lambda}\left[\pi\left(x_{n,n}\right)\right]$
exist, then we would have $\lim_{n}\gamma_{n}\left(\lambda\right)=0$.
The existence of $\psi\left(\lambda\right)$ follows from Theorem
\ref{Theorem 2}, while, since $\pi$ is continuous and bounded, that
of $\lim_{n}\mathbb{E}_{\lambda}\left[\pi\left(x_{n,n}\right)\right]$
follows from weak convergence. Moreover, since $\psi\in\mathcal{AC}$
by definition of CGF, weak convergence also imply that 
\begin{equation}
\lim_{n\rightarrow\infty}\mathbb{E}_{\lambda}\left[\pi\left(x_{n,n}\right)\right]=\pi(\lim_{n\rightarrow\infty}\mathbb{E}_{\lambda}\left(x_{n,n}\right))=\pi\left(\partial_{\lambda}\psi\left(\lambda\right)\right).
\end{equation}
Hence, from the above relations and by Lemma \ref{Lemma14} we obtain
the following non linear implicit ODE for $\psi$: 
\begin{equation}
\psi\left(\lambda\right)=\log\left[1+\left(e^{\lambda}-1\right)\pi\left(\partial_{\lambda}\psi\left(\lambda\right)\right)\right].
\end{equation}
The above ODE holds for every $\pi\in\mathcal{U}$, but its explicitation
obviously require that $\pi$ is invertible at least in the co-domain
of $\partial_{\lambda}\psi\left(\lambda\right)$. By Corollary \ref{Corollary 3}
we know that $\partial_{\lambda}\psi\left(\lambda\right)\in\left[z_{-}^{*},\inf C_{\pi}\right)$
for $\lambda\in\left(-\infty,0\right]$ and $\partial_{\lambda}\psi\left(\lambda\right)\in\left(\sup\, C_{\pi},z_{+}^{*}\right]$
for $\lambda\in\left[0,\infty\right)$, then we can restrict our invertibility
requirements to those domains. Notice that since for $\lambda\in\left[0,\infty\right)$
\begin{equation}
{\textstyle \inf_{\lambda}\left\{ \pi\left(\partial_{\lambda}\psi\left(\lambda\right)\right)\right\} }\leq\inf_{\lambda}\,\{\frac{{\textstyle e^{\psi\left(\lambda\right)}-1}}{e^{\lambda}-1}\}\leq\frac{e^{\psi\left(\lambda\right)}-1}{e^{\lambda}-1}{\textstyle \leq}\sup_{\lambda}\,\{\frac{e^{\psi\left(\lambda\right)}-1}{e^{\lambda}-1}\}\leq\sup_{\lambda}\left\{ \pi\left(\partial_{\lambda}\psi\left(\lambda\right)\right)\right\},\label{eq:4.9}
\end{equation}
then also $(e^{\psi\left(\lambda\right)}-1)/\left(e^{\lambda}-1\right)$
has co-domain $\left(\sup\, C_{\pi},\pi\left(z_{+}^{*}\right)\right]$. Similarly,
for $\lambda\in\left(-\infty,0\right]$, we find a co-domain $\left[\pi\left(z_{+}^{*}\right),\inf C_{\pi}\right)$
as for $\pi\left(\partial_{\lambda}\psi\left(\lambda\right)\right)$.

Let $\pi\in\mathcal{U}$ be an invertible function on $\left[z_{-}^{*},\inf C_{\pi}\right)$,
as required by the statement of Theorem \ref{Theorem 4}, and denote
by $\pi_{-}^{-1}:\left[\pi\left(z_{-}^{*}\right),\pi(\inf C_{\pi})\right)\rightarrow\left[z_{-}^{*},\inf C_{\pi}\right)$
its inverse. Moreover, let $\psi_{-}\left(\lambda_{-}^{*}\right)=\psi_{-}^{*}$
for some $\lambda_{-}^{*}\in\left(-\infty,0\right)$. Then, $\psi\left(\lambda\right)=\psi_{-}\left(\lambda\right)$,
with $\psi_{-}\left(\lambda\right)$ solution to the Cauchy problem
\begin{equation}
\partial_{\lambda}\psi_{-}\left(\lambda\right)=\pi_{-}^{-1}\left({\textstyle \frac{e^{\psi_{-}\left(\lambda\right)}-1}{e^{\lambda}-1}}\right),\,\psi_{-}\left(\lambda_{-}^{*}\right)=\psi_{-}^{*},\label{eq:1.21-1}
\end{equation}
If $\pi_{-}^{-1}\in\mathcal{AC}$ and Lipschitz, then we can apply
the Picard-Lindelof theorem, which ensure the existence and uniqueness
of $\psi_{-}$ for any $\lambda\in\left(-\infty,0\right)$. The same
proceeding can be applied to the case $\lambda\in\left(0,\infty\right)$:
let $\pi_{+}^{-1}:\left(\pi(\sup\, C_{\pi}),\pi\left(z_{+}^{*}\right)\right]\rightarrow\left(\sup\, C_{\pi},z_{+}^{*}\right]$
the inverse of $\pi$ on $\left(\sup\, C_{\pi},z_{+}^{*}\right]$,
let $\pi_{+}^{-1}\in\mathcal{AC}$ and Lipschitz, then for $\lambda\in\left(0,\infty\right)$
we have $\psi\left(\lambda\right)=\psi_{+}\left(\lambda\right)$,
with $\psi_{+}\left(\lambda\right)$ solution to the Cauchy problem
\begin{equation}
\partial_{\lambda}\psi_{+}\left(\lambda\right)=\pi_{+}^{-1}\left({\textstyle \frac{e^{\psi_{+}\left(\lambda\right)}-1}{e^{\lambda}-1}}\right),\,\psi_{+}\left(\lambda_{+}^{*}\right)=\psi_{+}^{*},\label{eq:1.21-1-1}
\end{equation}
and this completes our proof. Finally, that $\partial_{\lambda}\psi\left(\lambda\right)$
is continuous comes from the fact that both $\pi_{\pm}^{-1}$ and
$\left(e^{\psi_{\pm}\left(\lambda\right)}-1\right)/\left(e^{\lambda}-1\right)$
are continuous functions by definitions.

We proved that solutions are unique if $\lambda_{+}^{*}\in\left(0,\infty\right)$
and $\lambda_{+}^{*}\in\left(0,\infty\right)$ but since for $\lambda=0^{\pm}$
and $\lambda=\pm\infty$ the Lipschitz continuity in $\psi$ required
by the Picard-Lindelof theorem is not fulfilled we need an additional
argument to prove that the Caucy-Problem 
\begin{equation}
\partial_{\lambda}\psi_{+}\left(\lambda\right)=\pi_{+}^{-1}\left({\textstyle \frac{e^{\psi_{+}\left(\lambda\right)}-1}{e^{\lambda}-1}}\right),\,\lim_{\lambda\rightarrow0^{+}}\,\partial_{\lambda}\psi_{+}\left(\lambda_{+}^{*}\right)=\pi_{+}\left(\sup\, C_{\pi}\right),\,\lim_{\lambda\rightarrow\infty}\,\partial_{\lambda}\psi_{+}\left(\lambda_{+}^{*}\right)=z_{+}^{*},\label{eq:1.21-1-1-1-1}
\end{equation}
has a unique solution. Since the other cases can be shown by the same
way, we prove the result only for $\lambda>0$ and $\pi_{+}^{-1}$
strictly increasing. 

Let $\lambda>0$ and suppose that two solutions $\psi_{+}^{1}\left(\lambda\right)$
and $\psi_{+}^{2}\left(\lambda\right)$ exists for the Cauchy problem
\begin{equation}
\partial_{\lambda}\psi_{+}\left(\lambda\right)=\pi_{+}^{-1}\left({\textstyle \frac{e^{\psi_{+}\left(\lambda\right)}-1}{e^{\lambda}-1}}\right),\,\lim_{\lambda\rightarrow0^{+}}\,\partial_{\lambda}\psi_{+}\left(\lambda_{+}^{*}\right)=\pi_{+}\left(\sup\, C_{\pi}\right),\label{eq:1.21-1-1-1}
\end{equation}
such that $\psi_{+}^{1}\left(\epsilon\right)>\psi_{+}^{2}\left(\epsilon\right)$
for some $\epsilon>0$. Since we required $\pi_{+}^{-1}$ to be invertible,
$\mathcal{AC}$ and Lipschitz, it can be either strictly increasing
or strictly decreasing. For this setting we chose $\pi_{+}^{-1}$
strictly increasing, and then some $L^{*}>0$ exists such that
\[
\left[\partial_{\lambda}\psi_{+}^{1}\left(\lambda\right)-\partial_{\lambda}\psi_{+}^{2}\left(\lambda\right)\right]_{\lambda=\epsilon}>L^{*}\left(e^{\psi_{+}^{1}\left(\epsilon\right)}-e^{\psi_{+}^{2}\left(\epsilon\right)}\right).
\]
Since $\partial_{\lambda}\psi\left(\lambda\right)\in\left[0,1\right]$
by definition, then $\lim_{\lambda\rightarrow\infty}\partial_{\lambda}\psi_{+}^{1}\left(\lambda\right)\neq\lim_{\lambda\rightarrow\infty}\partial_{\lambda}\psi_{+}^{2}\left(\lambda\right)$
unless $\psi_{+}^{1}\left(\epsilon\right)=\psi_{+}^{2}\left(\epsilon\right)$.
This implies that the Cauchy problem in Eq. (\ref{eq:1.21-1-1-1-1})
has a unique solution that satisfies both boundary conditions. This
completes our proof.
\end{proof}

\subsubsection{Linear urn functions.}

\label{Section3.3.2}The last proof of this section is that of Corollary
\ref{Corollary8}, which gives the shape of $\psi$ in case $\pi$
is a linear function.
\begin{proof}
Let $\pi\left(s\right)$ as in Eq. (\ref{eq:2.12.2}). To ensure that
$\pi\left(0\right)>0$ and $\pi\left(1\right)<1$ we need at least
that $a>0$ and $a+b<1$. Given these conditions, let first consider
the case $\lambda>0$, so that the ODE to solve is 
\begin{equation}
a+b\,\partial{}_{\lambda}\psi\left(\lambda\right)=\frac{e^{\psi\left(\lambda\right)}-1}{e^{\lambda}-1}.\label{eq:4.11}
\end{equation}
We use the transformations $y\left(z\left(\lambda\right)\right)=e^{-\psi\left(\lambda\right)}-1$,
$z\left(\lambda\right)=1-e^{-\lambda}$, so that for $\lambda\in\left[0,\infty\right)$
we have $\psi\left(\lambda\left(z\right)\right)=-\log\left(1+y\left(z\right)\right)$,
$\lambda\left(z\right)=-\log\left(1-z\right)$ and
\begin{equation}
\partial{}_{z}y\left(z\right)=\left[{\textstyle \frac{a}{b\left(1-z\right)}+\frac{1}{bz}}\right]y\left(z\right)+\left[{\textstyle \frac{a}{b\left(1-z\right)}}\right],\label{eq:4.12}
\end{equation}
with $z\in\left[0,1\right]$. By Laplace method, we can rewrite the
above equation as 
\begin{equation}
\partial{}_{z}\left[y\left(z\right)\left(1-z\right)^{\frac{a}{b}}z^{-\frac{1}{b}}\right]={\textstyle \frac{a}{b}}\left(1-z\right)^{\frac{a}{b}-1}z^{-\frac{1}{b}}.
\end{equation}
Then, we define the function
\begin{equation}
B\left(\alpha,\beta;x_{1},x_{2}\right)=\int_{x_{1}}^{x_{2}}dt\,\left(1-t\right)^{\alpha-1}t^{\beta-1}.
\end{equation}
If $b>0$, since $a>0$ we have that $\left(1-z\right)^{\frac{a}{b}}z^{-\frac{1}{b}}$
is regular at $z=1$, then 
\begin{equation}
y\left(z;b>0\right)=\left(1-z\right)^{-{\textstyle \frac{a}{b}}}z^{\frac{1}{b}}\left[K_{1}^{*}-{\textstyle \frac{a}{b}}B\left({\textstyle {\textstyle \frac{a}{b}},\frac{b-1}{b};z,1}\right)\right],
\end{equation}
where $K_{1}^{*}$ depends on the initial conditions. Since when $\lambda\rightarrow\infty$
we must have $\partial{}_{\lambda}\psi\left(\lambda\right)\rightarrow1$,
from Eq. (\ref{eq:4.11}) we can write $\lim_{z\rightarrow1}y\left(z;b>0\right)=-1$.
Then, it can be shown that 
\begin{equation}
\lim_{z\rightarrow1}\left(1-z\right)^{-{\textstyle \frac{a}{b}}}z^{\frac{1}{b}}B\left({\textstyle {\textstyle \frac{a}{b}},\frac{b-1}{b};z,1}\right)=\frac{b}{a}.
\end{equation}
It follows that $K_{1}^{*}=0$, and substituting $y\left(z\left(\lambda\right)\right)=e^{-\psi\left(\lambda\right)}-1$,
$z\left(\lambda\right)=1-e^{-\lambda}$ we find the following expression
for $\lambda>0$, $b>0$ 
\begin{equation}
e^{-\psi_{+}\left(\lambda;b>0\right)}=1-{\textstyle \frac{a}{b}}e^{\frac{a}{b}\lambda}\left(1-e^{-\lambda}\right)^{\frac{1}{b}}B\left({\textstyle {\textstyle \frac{a}{b}},\frac{b-1}{b};1-e^{-\lambda},1}\right)
\end{equation}
If $b<0$, we have instead that $\left(1-z\right)^{\frac{a}{b}}z^{-\frac{1}{b}}$
is regular at $z=0$ and we take
\begin{equation}
y\left(z;d<0\right)=\left(1-z\right)^{-{\textstyle \frac{a}{b}}}z^{\frac{1}{b}}\left[K_{2}^{*}+{\textstyle \frac{a}{b}}B\left({\textstyle {\textstyle \frac{a}{b}},\frac{b-1}{b};0,}z\right)\right].
\end{equation}
This time we use $\lim_{z\rightarrow0}y\left(z;b<0\right)/z=-\pi\left(a/\left(1-b\right)\right)=-a/\left(1-b\right)$
and 
\begin{equation}
\lim_{z\rightarrow0}\left(1-z\right)^{-{\textstyle \frac{a}{b}}}z^{\frac{1}{b}-1}B\left({\textstyle {\textstyle \frac{a}{b}},\frac{b-1}{b};z,1}\right)=-\frac{b}{1-b}
\end{equation}
to find that $K_{2}^{*}=0$. Substituting as before we get the $\psi$
for $\lambda>0$ and $b>0$: 
\begin{equation}
e^{-\psi_{+}\left(\lambda;b<0\right)}=1+{\textstyle \frac{a}{b}}e^{\frac{a}{b}\lambda}\left(1-e^{-\lambda}\right)^{\frac{1}{b}}B\left({\textstyle {\textstyle \frac{a}{b}},\frac{b-1}{b};0,}1-e^{-\lambda}\right)
\end{equation}
Then, let consider the case $\lambda<0$: this time we take $y'\left(z'\left(\lambda\right)\right)=e^{\psi\left(\lambda\right)}-1$
and $z'\left(\lambda\right)=1-e^{\lambda}$ so that, again, $z'\in\left[0,1\right]$.
We can directly use the previous results for $\lambda>0$ by applying
the transformations $y\left(z\right)=-y'\left(z'\right)/\left[1+y'\left(z'\right)\right]$
and $z=-z'/1-z'$. Substituting in Eq. (\ref{eq:4.12}) and using
Laplace method we find 
\begin{equation}
\partial{}_{z}\left[{\textstyle \frac{y'\left(z'\right)}{1+y'\left(z'\right)}}\left(1-z'\right)^{\frac{1-a}{b}}\left(z'\right)^{-\frac{1}{b}}\right]={\textstyle \frac{a}{b}}\left(1-z'\right)^{\frac{1-a}{b}-1}\left(z'\right)^{-\frac{1}{b}}.
\end{equation}
Again, since $a\in\left[0,1\right]$ for $b>0$ the therm $\left(1-z'\right)^{\frac{1-a}{b}}\left(z'\right)^{-\frac{1}{b}}$
is regular at $z'=1$, then we take 
\begin{equation}
\frac{y'\left(z';b>0\right)}{1+y'\left(z';b>0\right)}=\left(1-z'\right)^{-\frac{1-a}{b}-1}\left(z'\right)^{\frac{1}{b}}\left[K_{3}^{*}-{\textstyle \frac{a}{b}}B\left({\textstyle \frac{1-a}{b},\frac{b-1}{b};z',1}\right)\right]
\end{equation}
and use $\lim_{z'\rightarrow1}y'\left(z';b>0\right)=-\pi\left(0\right)=-a$
and 
\begin{equation}
\lim_{z'\rightarrow1}\left(1-z'\right)^{-\frac{1-a}{b}-1}\left(z'\right)^{\frac{1}{b}}B\left({\textstyle \frac{1-a}{b},\frac{b-1}{b};z',1}\right)=-\frac{b}{1-a}
\end{equation}
to find that, again, $K_{3}^{*}=0$. Substituting $y'\left(z'\left(\lambda\right)\right)=e^{\psi\left(\lambda\right)}-1$
and $z'\left(\lambda\right)=1-e^{\lambda}$, for $\lambda<0$, $b>0$
we find
\begin{equation}
e^{-\psi_{-}\left(\lambda;b>0\right)}=1+{\textstyle \frac{a}{b}}e^{-\frac{1-a+b}{b}\lambda}\left(1-e^{\lambda}\right)^{\frac{1}{b}}B\left({\textstyle \frac{1-a}{b},\frac{b-1}{b};1-e^{\lambda},1}\right)
\end{equation}
Finally, if $b<0$ we can write down our solution as 
\begin{equation}
\frac{y'\left(z';b<0\right)}{1+y'\left(z';b<0\right)}=\left(1-z'\right)^{-\frac{1-a}{b}-1}\left(z'\right)^{\frac{1}{b}}\left[K_{4}^{*}+{\textstyle \frac{a}{b}}B\left({\textstyle \frac{1-a}{b},\frac{b-1}{b};0,z}'\right)\right].
\end{equation}
Then, from $\lim_{z\rightarrow0}y\left(z';b<0\right)/z'=-a/\left(1-b\right)$
and
\begin{equation}
\lim_{z'\rightarrow0}\left(1-z'\right)^{-\frac{1-a}{b}-1}\left(z'\right)^{\frac{1}{b}-1}B\left({\textstyle \frac{1-a}{b},\frac{b-1}{b};z',1}\right)=-\frac{b}{1-b}
\end{equation}
we find that also the last constant is $K_{4}^{*}=0$, and that 
\begin{equation}
e^{-\psi_{-}\left(\lambda;b<0\right)}=1-{\textstyle \frac{a}{b}}e^{-\frac{1-a+b}{b}\lambda}\left(1-e^{\lambda}\right)^{\frac{1}{b}}B\left({\textstyle \frac{1-a}{b},\frac{b-1}{b};0,1-e^{\lambda}}\right).
\end{equation}
This completes the proof. Notice that the boundary conditions we used
to compute $\psi$ are one for each equation, while in general Theorem
\ref{Theorem 4} would require two. The fact that our solutions are
univocally determined by a single boundary condition reflects the
analyticity of these solution in their proximity. It's easy to verify
that the above functions fulfill both initial conditions of Theorem
\ref{Theorem 4} anyway.
\end{proof}
We remark that in the above proof the case $b=0$ is not considered,
since we would get a Bernoulli process whose $\phi$ can be trivially
computed by elementary techniques. Anyway, taking the limit $b\rightarrow0$
in the above expressions will return the desired result.

\section{Acknowledgments.}

I would like to thank Pietro Caputo (Universita degli Studi Roma 3)
for his critical help in preparing this work. I would also like to
thank Giorgio Parisi (Sapienza Universita di Roma) and Riccardo Balzan
(Universite Paris-Descartes) for interesting discussions and suggestions,
and Woldek Bryc (University of Cincinnati) for bringing to my attention
reference \cite{Bryc}. Finally, I am grateful to Bernard Bercu, Michel Bonnefont and Adrien Richou (Universite Bordeaux) for spotting the inverted signs in Corollary 12, and for many other crucial comments.

\end{document}